\documentclass{book}
\usepackage[mems, lang=british]{ems-book} %% change to `lang=american' if you use American English

\usepackage{cite}

\newtheorem{theorem}{Theorem}[section]
\newtheorem{lemma}[theorem]{Lemma}
\usepackage{comment}
\theoremstyle{theorem}
\newtheorem{definition}[theorem]{Definition}

\usepackage{bookmark}
\theoremstyle{remark}

\newtheorem{lem}[theorem]{\sc \bf Lemma}

\newtheorem{cor}[theorem]{\sc \bf Corollary}
\newtheorem{quest}[theorem]{\sc \bf Question}

\newtheorem{proposition}[theorem]{\sc \bf Proposition}

\newtheorem{rem}[theorem]{\sc \bf Remark}
\newtheorem{remark}[theorem]{\sc \bf Remark}
\newtheorem{fact}[theorem]{\sc \bf Fact}
 \usepackage{color}
\usepackage{hyperref}					% Reference
\hypersetup{colorlinks,%				% Reference color setup	
	linkcolor=blue,%
	citecolor=blue}

\newcommand{\cA}{\mathcal{A}}

\newcommand{\cH}{\mathcal{H}}
\newcommand{\cF}{\mathcal{F}}

\newcommand{\cM}{\mathcal{M}}
\newcommand{\cP}{\mathcal{P}}
\newcommand{\cN}{\mathcal{N}}

\newcommand{\norm}[1]{\left\lVert#1\right\rVert}

\newcommand{\cB}{\mathcal{B}}

\newcommand{\supp}{\operatorname{supp}}

\numberwithin{equation}{section}

\begin{document}

\frontmatter

\title{Isometric Structure in  Noncommutative Symmetric Spaces}

\author{Kai Fang (Institute for  Advanced Study in  Mathematics of HIT,  Harbin, 150001, China), 
\email{kaifang8.25@gmail.com}\\
Tianbao Guo (Department of Mathematics, Harbin University of Science and Technology, Harbin, 150080, China),
\email{tianbaoguo.edu@gmail.com}\\
 Jinghao Huang (Institute for  Advanced Study in  Mathematics of HIT,  Harbin, 150001, China), \email{jinghao.huang@hit.edu.cn} (the corresponding author)\\ 
Fedor Sukochev (School of Mathematics and Statistics, University of New South Wales, Kensington, 2052, NSW, Australia) \email{f.sukochev@unsw.edu.au}}

\maketitle
%------
% Insert an abstract.
%------
\begin{abstract}This is a systematic study of isometries between noncommutative symmetric spaces. 
Let $\cM$ be a semifinite von Neumann algebra (or an atomic  von Neumann algebra with all atoms having the same trace) acting on a   separable  Hilbert space $\cH$ equipped with a semifinite faithful normal trace $\tau$.
We show that
for any noncommutative symmetric space corresponding to a symmetric function space  $E(0,\infty)$  in the sense of  Lindenstrauss--Tzafriri
such that $\norm{\cdot}_E\ne \lambda\norm{\cdot}_{L_2}$, $\lambda\in \mathbb{R}_+$,  any isometry on $E(\cM,\tau)$ is  of elementary form.
This answers a long-standing open question raised in the 1980s in the non-separable setting [Math. Z. 1989], while the case of separable symmetric function spaces was treated in 
  [Huang \& Sukochev, JEMS, 2024]. 
As an application, we obtain a noncommutative Kalton--Randrianantoanina--Zaidenberg Theorem, providing  a characterization of noncommutative $L_p$-spaces over finite von Neumann algebras and a necessary and sufficient condition for an operator on a noncommutative symmetric space to be an isometry.
Having this at hand,  we answer a question posed by Mityagin in 1970 [Uspehi
Mat. Nauk] and its noncommutative counterpart by showing the any symmetric space $E(\cM,\tau)\ne L_p(\cM,\tau)$ over a noncommutative probability is not isometric to a symmetric space over a von Neumann algebra equipped with a semifinite infinite faithful normal trace.
It is also shown that any noncommutative $L_p$-space, $1\le p<\infty$, affiliated with an atomless semifinite von Neumann algebra has a unique symmetric structure up to isometries. 
This contributes to the resolution of an isometric version of Pe\l czy\'nski's problem concerning the uniqueness of the symmetric structure in  noncommutative symmetric spaces. 
\end{abstract}

%------
% Optional: Dedication.
%------
%\dedication{This memoir is dedicated to XXX.}

%------
% Insert a list of keywords.
% -- Separate keywords with comma.
% -- Capitalize only the first keyword in the list.
% -- No final full stop.
%------

%------
% Insert MSC 2020 codes according to www.ams.org/msc/msc2020.html.
% -- There must be exactly *one* primary code.
% -- The number of secondary codes is not specified.
% -- No final full stop.
%------
\classification[47B15]{46B04, 46L52}

%------
% Insert acknowledgments.
%------
\begin{ack}
The authors would like to thank Professor Mikhail Zaidenberg for his interest and for pointing
us to the existing literature and related problems in this field. We also thank Professor Dmitriy Zanin for helpful discussions. 
\end{ack}

%------
% Insert information regarding funding.
%------
\begin{funding}
Jinghao Huang was supported by the NNSF of China  (No. 12031004, 12301160 and 12471134).
Fedor Sukochev was supported by  the Australian Research Council (DP230100434).
\end{funding}

\tableofcontents
\mainmatter

\chapter{Introduction}

The primary motivation of this paper is to resolve the isometric part of the 
 question
 asked by 
 Mityagin  \cite{Mityagin} in 1970. 
 \begin{quest}\label{question2}\cite[p.99]{Mityagin}
Can a  symmetric space $E$ of measurable functions on $[0, 1]$ other
than $L_p $, $ 1 \leq p < \infty $, be isomorphic (isometric) to any symmetric
space $ F $ on the semi-axis $[0, \infty)$ or the axis $(-\infty, +\infty)$?
\end{quest}
The full answer is given in Theorem  \ref{Mityagin}. 
The proofs for 
 Theorem  \ref{Mityagin}  and 
 its noncommutative counterpart  (see Chapter \ref{S:M}) rely 
 on the description of surjective isometries of commutative/noncommutative symmetric spaces initiated in the 1950s,
 which is the secondary major aim of this work. 
\begin{quest}\label{question1}\cite{CMS,Sukochev}
If $E (0,\infty)$ is a symmetric function on $(0,\infty)$ and if $(\cM,\tau)$ is a semifinite von Neumann algebra (on a separable Hilbert space) with a semifinite faithful normal trace $\tau$, then  how can one describe
the family of all isometries of the associated symmetric operator spaces  $E(\cM,\tau)$?
\end{quest} 
Recently, it was shown in \cite{HS} under the condition that $E(0,\infty )$ is separable  that any such isometry is of elementary form.  
Here, in Theorem \ref{11212},
 we describe all surjective isometries of $E(\cM,\tau)$ for the spaces $E(0,\infty)$ in the sense of Lindenstrauss and Tzafriri, which is not necessarily separable. 

 Our study also yields a pleasant byproduct, 
the uniqueness of symmetric structure of numerous examples of
 noncommutative symmetric spaces (up to isometries).
Such a question concerning the uniqueness was raised by 
 Pe\l czy\'nski in 1979 \cite{Arazy81,Arazy83}. 
 In particular,  we obtain 
 an isometric characterization of noncommutative $L_p$-spaces in the class of noncommutative symmetric spaces, see Theorem \ref{AZ111} below. 
 In the 1960s, such a characterization (in the setting of symmetric spaces on $(0,1)$) was obtained by Semenov~\cite{Semenov} and Abramovich and Zaidenberg~\cite{AZ}. 
 Our result, when specialized to symmetric function spaces yields an extension of the Semenov--Abramovich--Zaidenberg Theorem to the setting of symmetric function spaces on $(0,\infty )$.

\section{Isometries on symmetric function/sequence spaces}
The isometric theory of Banach spaces was born and developed in inseparable connection with other areas of the Banach spaces theory \cite{KK}. 
We shall omit the adjective ``linear" since non-linear isometries are not considered in this paper. The study of isometries on symmetric spaces  has a long history, beginning with  Stefan Banach \cite{Banach}, who showed that isometries of the space  $\ell_p$ preserve disjointness for any  $1\le p<\infty$ (and $p \neq 2$) in the 1930s.  
Banach  remarked that the proof for the description of isometries on  $L_p (0,1)$ will appear in Studia Mathematica IV.
However, the promised pulication never appeared.  
The full proof can be found in the paper by Lamperti \cite{Lamperti},  who
established the characterization of isometries on $L_p$-spaces
and  certain Orlicz function spaces over $\sigma$-finite measure spaces.
The resolution of the (not normed, nor quasi-normed) case when $p=0$ is obtained in \cite{BHS}.

 In the realm of complex function spaces, a method (the so-called hermitian operator method) pioneered by Lumer~\cite{Lumer1,Lumer}, has proven notably efficient. Lumer~\cite{Lumer} applied this approach to explore isometries within reflexive Orlicz spaces, and subsequently, Zaidenberg~\cite{Zaidenberg,Z77} utilized it to investigate isometries in the setting of   general complex symmetric function spaces.
    The case for separable  complex sequence spaces was treated by  Arazy ~\cite{Arazy85} and the non-separable case was addressed by Aminov and Chilin in \cite{AC}.
    Recall that following theorem due to Zaidenberg \cite[Theorem 1]{Zaidenberg} (see also \cite{Z77,FJ}):
\begin{theorem}\label{Z}
Let $E_1(\Omega_1,\Sigma_1,\mu_1)$ and $E_2(\Omega_2,\Sigma_2,\mu_2)$ be complex symmetric (complex) function spaces associated
with atomless $\sigma$-finite measure spaces    $(\Omega_1,\Sigma_1,\mu_1)$ and $(\Omega_1,\Sigma_1, \mu_1)$, respectively.
Assume that   the norm on $E_1(\Omega_1,\Sigma_1,\mu_1)$ is not proportional to $\left\|\cdot\right\|_2$.
Then, for any isometry  $T: E_1(\Omega_1,\Sigma_1,\mu_1) \rightarrow E_2(\Omega_2,\Sigma_2,\mu_2)$, there exists a measurable function $h$ and an invertible measurable transformation $\sigma: \Omega_2 \rightarrow \Omega_1$ such that
\begin{align}\label{Zaidenberg}
(Tf)(t) = h(t)f(\sigma(t)), \quad \text{for all } f \in E_1(\Omega_1,\Sigma_1,\mu_1).
\end{align}
\end{theorem}
%Consider a discrete measure space $(\Omega, \Sigma, \mu)$ where $\mu(\{\omega\}) = 1$ for every $\omega \in \Omega$. We denote by $\ell_p(\Omega)$, $1 \leq p \leq 1$, the $\ell_p$-space on $(\Omega, \Sigma, \mu)$ \cite[p.xi]{LT1}. While the $\ell_p(\Omega)$ space has been extensively studied (see, for example, \cite{LT1,SSU,HR98} and references therein), and the description of surjective isometries of $\ell_p(\Omega)$ follows from \cite{Yeadon,Sherman}, the scenario involving arbitrary symmetric spaces $E(\Omega)$ for uncountable $\Omega$ remains unaddressed. The description of isometric operators on these classical Banach spaces is included in our results, Theorem \ref{th:iso}.
Note that \eqref{Zaidenberg} is
a necessary but not sufficient   condition for an operator on a symmetric function space  to be an isometry.

We also highlight significant differences between the studies of  isometries on real Banach spaces and those on complex Banach spaces because
  Lumer's approach works only with complex Banach spaces.
  The case for real Banach spaces appears to be more difficult.
Isometries on real symmetric sequence spaces were treated in \cite{BS74,BS75}.
 Kalton and Randrianantoanina~\cite{KR2,Kalton_R93,R,Rthesis} obtained the description of isometries on real symmetric  function spaces on $(0,1)$ in the sense of Lindenstrauss--Tzafriri. Recall the following result by Zaidenberg, and Kalton and Randrianantonina.
%\begin{theorem}\label{KRZ1}\cite{KR2,Kalton_R93,R,Z80}
%Assume that $E(0, 1)$ is
%\begin{enumerate}
%  \item   either  a real  symmetric function space in the sense of Lindenstrauss--Tzafriri~~\cite{LT2};
%  \item or   a complex symmetric function space,
%\end{enumerate}
%  whose norm is nor proportional to the norm on  $L_2(0,1)$.
% Then,
%\begin{enumerate}
%    \item either $E = L_p$ up to an equivalence of norm for some $1 \leq p \leq \infty$,
%    \item or $E \neq L_p$, all surjective isometries $T$ from $E(0, 1)$ onto $F(0, 1)$ have the
%    form $$ax(\sigma(s)),$$  where $a:[0,1]\to \mathbb{R}$ is a non-vanishing Borel function such that  $|a| = 1$ a.e., and $\sigma : [0, 1] \to [0, 1]$ is an invertible
%    Borel measure-preserving map.
%\end{enumerate}
%\end{theorem}

\begin{theorem}\label{KRZ1}\cite{KR2,Kalton_R93,R,Z80,Rthesis}
Assume that $E(0, 1)$ and $F(0,1)$ (with $\norm{\chi_{_{(0,1)}}}_E=\norm{\chi_{_{(0,1)}}}_F=1$)
\begin{enumerate}
  \item   either  are real  symmetric function spaces in the sense of Lindenstrauss--Tzafriri\footnote{i.e., 
a symmetric function space has the Fatou property or is minimal (which is the closed linear span of the simple integrable functions) \cite[p.118]{LT2}.
Note that any non-separable minimal symmetric function space $E(0,1)$ 
coincides with $L_\infty (0,1) $ and $E(0,1)$ has the Fatou property. 
};
  \item or     complex symmetric function spaces.
\end{enumerate}
If there exists a surjective isometry $T$ from $E(0,1)$ onto $F(0,1)$, then 
\begin{enumerate}
    \item either $E(0,1) = L_p(0,1)$  up to an equivalent   norm for some $1 \leq p \leq \infty$, and $T$ is a surjective isometry from $L_p(0,1)$ onto $L_p(0,1)$;
    \item or $E(0,1) \neq L_p(0,1)$ and  the  isometry $T$ from $E(0, 1)$ onto $F(0, 1)$ has the
    form:$$a(s)x(\sigma(s)) \quad a.e.,$$ for any $x\in E(0,1)$ where $a$  is a non-vanishing Borel real (respectively, complex) function such that  $|a| = 1$ a.e., and $\sigma : (0, 1) \to (0, 1)$ is an invertible
    Borel measure-preserving map.
\end{enumerate}
\end{theorem}

In this paper, we exclusively focus on complex Banach spaces and surjective linear isometries.

\section{Isometries on noncommutative $L_p$-spaces}
The first achievement in the field of isometries on noncommutative spaces was due to  Kadison \cite{Kalton} in the 1950s,
who showed that a surjective isometry between two von Neumann algebras can be expressed as a Jordan $*$-isomorphism multiplied by a unitary operator.
This result is often viewed as a noncommutative version of the Banach--Stone Theorem.
The study   of   noncommutative $L_p$-spaces dates back to the 1950s when  Dixmier and Segal introduced noncommutative $L_p$-spaces~\cite{D50,Se}.
Following their works, a number of scholars including Broise~\cite{Broise}, Russo \cite{Russo}, Arazy~\cite{Arazy}, Tam \cite{Tam}, etc, initiated the  study on isometries  on these spaces.
 The complete description of isometries on noncommutative $L_p$-spaces (in the semifinite setting) was   accomplished by Yeadon~\cite{Yeadon}.

 \begin{theorem}\label{Yeadon}
    Every (not necessarily surjective) isometry $T$ from
  a  non-commutative $L_p$-space $L_p(\cM_1,\tau_1)$ into another $L_p(\cM_2,\tau_2)$ (over semifinite von Neumann algebras $\cM_1$ and $\cM_2$)
  has the following form
  $$T(x) =u BJ(x), ~x\in L_p(\cM_1,\tau_1)\cap \cM_1, $$
  where $u$ is  a partial isometry  in $\cM_2$,   $J$ is  a Jordan $*$-isomorphism from $\cM_1$ onto a weakly closed $*$-subalgebra of $\cM_2$, $B$ is a positive self-adjoint operator affiliated with $\cM_2$ such that every spectral projection of $B$ commutes with $J(x)$ for all $x\in \cM_1$ and $\tau_1(x) =\tau_2(B^p J(x))$ for all $0\le x\in \cM_1$.
 \end{theorem}
 
 For the case of type $III$ von Neumann algebras, we refer to   \cite{Sherman,Sherman06, Sherman06b, JS,JRS, Watanabe}. 
 In this paper, we work only with semifinite von Neumann algebras. 
Isometries on $\mathscr{C}_p$-spaces in nest algebras were studies in \cite{AnKa}. 
However, the description for isometries of noncommutative Hardy spaces in the general semifinite setting seems to be unknown. 

\section{Isometries on noncommutative symmetric  spaces}
Since the introduction of noncommutative symmetric   spaces in the 1970s, there has been extensive research on  general noncommutative symmetric  spaces on semifinite von Neumann algebras (see, for example, \cite{O,O2,DDP,S87,Kalton_S,DDP93} and references therein).
The initial motivation of the study of isometries on noncommutative symmetric  spaces is the question whether
 these isometries $T$ possess a natural description
  similar to Theorems \ref{Z}, \ref{KRZ1}  and \ref{Yeadon}.
  One of the most significant advancements in this field was made by Sourour~\cite{Sourour}, who provided a full  description of surjective isometries on separable symmetric operator ideals in $B(\cH)$.
  By adopting Sourour's techniques, the fourth author of the present paper \cite{Sukochev} managed to derive the description of
   surjective isometries on separable symmetric operator spaces associated with the
 hyperfinite type $II_1$-factor.
   However,  Sourour's proof   \cite{Sourour} strongly relies on the matrix representation of compact operators on a separable Hilbert space $\cH$, which
    does not seem to  be adaptable  to the  work of general   symmetric operator spaces.
     In the latter setting, several partial results have been achieved.
      For instance, the general form of isometries of Lorentz spaces on a finite von Neumann algebra was obtained in \cite{CMS} (see also \cite{MS2,KM,CHL,CT86}).
      When additional conditions are imposed on the isometries (such as disjointness-preserving, order-preserving, etc.), similar descriptions are available in \cite{MS, JC, JC2, SV,  CMS,HSZ,FJ2,Katavolos2,MZ,R3,Abramovich,Medzhitov,CLS}, offering partial answers to Question~\ref{question1}.

 Recently, the third and the fourth authors of the present paper~\cite{HS} answered Question \ref{question1} in the setting of noncommutative symmetric   spaces $E(\cM,\tau)$ having order continuous norm by showing the all surjective isometries on $E(\cM,\tau)$ are of elementary form, which resolved a question posed in \cite{Sukochev}.
In particular, this implies that the Banach--Mazur  rotation problem
(whether every separable Banach space with transitive group of isometries has to be isometric to a Hilbert space \cite{Banach}, see also \cite[Problem 9.6.2]{Rolewicz})
 has an affirmative   answer in the scope of separable commutative/noncommutative symmetric spaces.

 One of the  goals of the present paper is to treat the case when the norm on $E(\cM,\tau)$ is not necessarily order continuous.
We note that the proof in \cite{HS} applies to noncommutative symmetric spaces generated by  minimal symmetric function  spaces.
The following theorem is the first main result of this paper, which answers Question~\ref{question1} in the case when the norm on $E(\cM,\tau)$ is not necessarily order continuous.
This extends numerous results in the literature.

\begin{theorem}\label{11212}
Let $\cM_1$ and $\cM_2$ be  atomless $\sigma$-finite von Neumann algebras (or $\sigma$-finite atomic von Neumann algebras whose atoms have the same trace) equipped with semifinite faithful normal traces $\tau_1$ and $\tau_2$, respectively.
Let $E(\cM_1,\tau_1)$ and $F(\cM_2,\tau_2)$ be two strongly symmetric operator spaces having the Fatou property, whose norms  are not proportional to  $\norm{\cdot}_2$.
If $T:E(\cM_1,\tau_1)\to F(\cM_2,\tau_2)$ is a surjective isometry, then there exist
two sequences of elements $ A_i \in  F(\cM_2 ,\tau_2),~   i\in I $,  disjointly supported from the left and $ B_i \in F(\cM_2 ,\tau_2), ~i\in I$,    disjointly supported from the  right,   and a  surjective  Jordan $*$-isomorphism $J:\cM_1\to \cM_2$  and a central projection $z\in \cM_2$ such that
$$T(x) =  \sigma(F,F^\times)-  \sum_{i=1}^\infty     J(x)A_i z + B_i J(x) ({\bf 1}-z)  , ~x\in E(\cM_1,\tau_1)\cap \cM_1.$$
 In particular,
if the trace $\tau$ is finite, then there exist   elements $A,B\in F(\cM_2,\tau_2)$ such that
$$T(x) =    J(x)A  z + B  J(x) ({\bf 1}-z) , ~\forall x\in E(\cM_1,\tau_1)\cap \cM_1. $$
\end{theorem}

%This  extends many previous findings in this field (see, e.g., \cite{Arazy85,JC,JC2, Sourour,Sukochev,CMS,Russo,JC,JC2,Lumer, Sherman,Watanabe,Yeadon,MS,MS2,Medzhitov}), and Theorem yields the description of surjective isometries on symmetric operator spaces associated with  non-hyperfinite algebras. %It is worth noting that there are examples (see \cite[Example 6.2]{HS}) illustrating that the conditions on von Neumann restrictions cannot be removed.

 \section{A noncommutative Kalton--Randrianantoanina--Zaidenberg Theorem}

Note that the main result in  \cite{HS} and Theorem \ref{11212} above do not  provide a sufficient  condition for an operator on a noncommutative symmetric space $E(\cM,\tau)$ to be a surjective  isometry.
Another purpose of this paper is to   provide a necessary and sufficient condition for an operator on $E(\cM,\tau)$ to be an isometry by establishing a noncommutative version of the  Kalton--Randrianantoanina--Zaidenberg Theorem (see~Theorem~\ref{KRZ1}), see below.

%The following theorem is a noncommutative noncommutative version Kalton--Randrianantoanina--Zaidenberg Theorem.
\begin{theorem}\label{K-R-Z}
Let $E(0,1  ) $ and $ F(0,1)$ be   symmetric function spaces in the sense of Lindenstrauss and Tzafriri.
Let $\cM_1$ and $ \cM_2 $ be   atomless finite von Neumann algebras  equipped with   faithful normal tracial states  $\tau_1$ and $\tau_2$, respectively.
Let  $E(\cM_1,\tau_1)$ and $F(\cM_2,\tau_2)$ be the noncommutative  symmetric space corresponding to $E(0,1 )$ and $F(0,1)$,  respectively.
Let $T:E(\cM_1,\tau_1) \to F(\cM_2,\tau_2)$ be a surjective isometry.
Then,  \begin{enumerate}
        \item If $E(0,1)$ coincides with $L_p(0,1)$ (as sets)  for some $1\le p\le \infty $, then $T$ is a surjective isometry  from $L_p(\cM_1,\tau_1)$ onto $L_p(\cM_2,\tau_2)$;
        \item If $E(0,1)$ does not coincide with  $L_p(0,1)$ (as sets), then
$$T(x)=u\cdot J(x), ~\forall  x\in E(\cM_1,\tau_1), $$
where $u$ is a unitary operator in $\cM_2$ and $J:S(\cM_1,\tau_1)\to S(\cM_2,\tau_2) $ is a trace-preserving Jordan $*$-isomorphism.
       \end{enumerate}
\end{theorem}
It was shown by Potepun
 \cite{Potepun71} (see also \cite[Corollary 4]{Ab91}) that we have $E(0,1)=F(0,1)$ up to an equivalent norm whenever $E(0,1)$ is isometric to $F(0,1)$. Theorem~\ref{K-R-Z} (1) above implies 
a noncommutative version of this result. Precisely, 
if two noncommutative  symmetric spaces $E(\cM_1,\tau_1)$ and $F(\cM_2,\tau_2)$ over noncommutative probability spaces
 are isometric, then $E(0,1)=F(0,1)$ up to an equivalent norm.
This result has a direct connection to Pe\l czy\'nski's question, which we discuss in Section~\ref{Pel} below.

As mentioned in \cite[Section 1]{KR},
 there appear to be obstacles to extending the main result \cite[Theorem 1.1]{KR} to the setting of symmetric function   spaces on $[0, \infty)$.
In Chapter \ref{Sec:inf}, we also establish a version of Theorem \ref{K-R-Z} above
 in the setting of von Neumann algebras equipped with semifinite infinite faithful normal traces, which is new even in the commutative setting.

%This theorem establish a noncommutative version of Kalton--Randrianantoanina--Zaidenberg~\cite{R}\cite{Z77} \cite{KR2} in the context of symmetric spaces in the sense of Lindenstrauss--Tzafriri~\cite{LT2}, it broadens the scope of many previous findings in this area.

\section{Mityagin's question and its noncommutative counterpart}

Mityagin's question concerning isomorphisms between 
symmetric function spaces on $(0,1) $ and $(0,\infty)$
was one of the main motivations of the outstanding memoir \cite{JMST}, 
as well as of \cite{LT2}.
It  
 is known to have a positive answer for some special symmetric function spaces, see e.g. \cite{Astashkin11}, \cite[Theorem 2.f.1]{LT2} and \cite{JMST}.
 Below, we treat the part of  Mityagin's question concerning isometric  isomorphisms.
 % In this case,
%  the answer to the above question is negative.
 %  Note that
%Mityagin's question did not ask to exclude $ L_\infty $, which is  a trivial positive answer
%to the above question because
 %$ L_\infty(0, 1) $ is isometric to $ L_\infty(0, \infty) $.
 In our paper, we show that any complex  symmetric function space $E(0,1)\ne L_p(0,1)$, $1\le p\le \infty$, is not isometric to a symmetric function space $F(0,\infty)$.

%The following theorem answers Question \ref{question2}.
\begin{theorem}\label{Mityagin}
If a (complex) symmetric function space  $E(0, 1) \neq L_p(0, 1)$ (as sets), $1 \leq p \leq \infty$, then $E(0, 1)$ is not isometric to any symmetric function space $F(0, \infty)$.
\end{theorem}
For $E(0,1) = L_p(0,1)$ with equivalent  (but not equal) norms, the first
example of an isometry from $E(0,1)$ onto some symmetric function space $F(0,\infty)$ was obtained by Lamperti \cite{Lamperti}, see also \cite[p.639 Example]{Z77}.
See also \cite[Example 1]{R} for an  example of  an Orlicz space $E(0, 1)$   isometric to another Orlicz space $F(0, 1)$, both of which coincide with $L_p(0,1)$ as sets.

The following theorem delivers a negative answer to Mityagin's question in the noncommutative setting. 
\begin{theorem}
 Let $\cM$ and $\cN$ be two atomless von Neumann algebras on separable Hilbert spaces.
Let $\cM$ be   equipped with a   faithful normal tracial state $\tau$ and let $\cN$ be   equipped with a semifinite infinite faithful normal trace $\nu$.
If a  symmetric function space $E(0,1)$ in the sense of Lindenstrauss and Tzafriri  does not coincide with  $  L_p(0, 1)$ (as sets), $1 \leq p \leq \infty$, then
$E(\cM, \tau)$ is not isometric to any symmetric   space $F(\cN,\nu)$ over $(\cN,\nu)$.
\end{theorem}
We note that, it is shown in \cite{HRS} that for any $1\le p \ne 2 <\infty$, 
$L_p(\cM,\tau)$ over a finite von Neumann algebra is not isomorphic to $L_p(\cN,\nu)$ over an infinite semifinite von Neumann algebra. 

\section{Pe\l czy\'nski's question concerning the uniqueness of symmetric structure}\label{Pel}
Let $C_E$ be the symmetric operator ideal in $B(\ell_2)$ generated by a symmetric sequence space $E$.
We say that $C_E$ has a unique symmetric structure if $C_E$ isomorphic to some ideal $C_F$ corresponding to a symmetric sequence space $F$ implies that $E=F$ with equivalent norms.
At the International Conference on Banach Space Theory and its Applications at Kent, Ohio (August 1979), Pe\l czy\'nski posed the following question concerning the symmetric structure of ideals of compact operators on the Hilbert space $\ell_2$ (see also \cite[Question (B)]{Arazy81} and \cite[Problem A]{Arazy83}) :
 \begin{quote}
Does the ideal $C_E$ of compact operators corresponding to an arbitrary separable symmetric sequence space $E$ have a unique symmetric 
structure? 
\end{quote}
For readers who are interested in this topic, we refer to \cite{Arazy81,Arazy83,HSS}.
One may consider an analogue of Pe\l czy\'nski's question in the sense of isometric isomorphisms.
If a noncommutative symmetric space $E(\cM,\tau)$ isometric to $F(\cM,\tau)$ implies that 
$E(\cM,\tau)$ coincides with $F(\cM,\tau)$ up to a proportional norm, then we say that $F(\cM,\tau)$ has a unique symmetric structure 
up to isometries \cite{HS}.
It is an immediate consequence of Theorem  \ref{K-R-Z} that 
if
a symmetric function space (in the sense of Lindenstrauss and Tzafriri)
 $F(0,1)\ne L_p(0,1)$ (as sets), then for any atomless finite von Neumann algebra 
 $\cM$ equipped with a faithful normal tracial state $\tau$, $F(\cM,\tau)$ has a  unique symmetric structure 
up to isometries. 
If, in addition that, $\cM$ is a $II_1$-factor or $B(\cH)$, then all separable noncommutative symmetric spaces
 $F(\cM,\tau)$ (in the sense of Lindenstrauss--Tzafriri) has a unique symmetric structure up to isometries, see \cite[Section 5]{HS} and~\cite{AC2, Sourour, Sukochev}.

Semenov \cite{Semenov} showed that any symmetric function space $E(0,1)$ isometric to $L_2(0,1)$ coincides with $L_2(0,1)$ up to a proportional norm, see also \cite{Potepun}.
This result was later extended by 
  Abramovich and  Zaidenberg \cite{AZ}, who
obtained  a characterization of $L_p$-function spaces ($1\le p<\infty$) by showing that 
any symmetric function space $E(0,1)$ isometric to $L_p(0,1)$ coincides with 
 $L_p(0,1)$ (with
 $\norm{\cdot}_E =\lambda \norm{\cdot}_{L_p}$, $\lambda >0$).
Therefore, $L_p(0,1)$, $1\le p<\infty $, has a unique symmetric structure up to isometries in the class of symmetric function spaces on $(0,1)$. 

In the present paper, we establish a noncommutative version of the Semenov--Abramovich--Zaidenberg Theorem, 
proving the uniqueness of symmetric structure (up to isometries) in noncommutative $L_p$-spaces, 
see Chapter \ref{Sec:AZ}.
\begin{theorem}\label{AZ111}

Let $\cM$ be an atomless   von Neumann algebra equipped with a    faithful normal tracial   state $\tau$.
Let $E(0,1)$ be a symmetric space in the sense of Lindenstrauss and Tzafriri.
If
there exists an  atomless von Neumann algebra $\cN$ equipped with a  faithful normal tracial state  $\nu$ such that 
 $E(\cM,\tau)$ is isometric to $L_p(\cN,\nu)$, $1\le p<\infty$,   then $E(\cM,\tau)$ coincides with $L_p(\cM,\tau)$ and    $\norm{\cdot}_E =\lambda \norm{\cdot}_{L_p}$, $\lambda >0$.
\end{theorem}
Note that Theorem \ref{AZ111} holds in the setting of infinite semifinite von Neumann algebras,  which   is new even  in the commutative setting because it holds for  symmetric function spaces on infinite intervals $(0,\infty )$, see  Corollary \ref{AZ infinite} below.

%Recall the following result by Zaidenberg (Kalton-Randrianantonina):
%\begin{theorem}\cite{KR2,R,Z77} Let $E(0, 1)$ be a symmetric function space which is not $L_2$ (with proportional norms). Then,
%
%\begin{enumerate}
%    \item either $E = L_p$ up to an equivalence of norm for some $1 \leq p \leq \infty$,
%    \item or $E \neq L_p$, all surjective isometries $T$ from $E(0, 1)$ onto $F(0, 1)$ have the
%    form $ax(\sigma(s))$, where $|a| = 1$ a.e., and $\sigma : [0, 1] \to [0, 1]$ is an invertible
%    Borel measure-preserving map.
%\end{enumerate}
%\end{theorem}

%\begin{theorem}\cite{Zaidenberg} Let $E(0, 1)$ be a symmetric (complex) function space whose norm is not proportional to $\| \cdot \|_2$. Then, any surjective isometry $T : E(0, 1) \rightarrow F(0, \infty)$ has the form
%
%$$ T(x)(t) := a(t) x(\sigma(t)), \quad \forall t \in (0, \infty), $$
%
%where $a \in F(0, \infty)$ and $\sigma$ is an invertible measurable transformation from $(0, \infty)$ onto $(0, 1)$.
%\end{theorem}
\chapter{Preliminaries}\label{s:p}

In this chapter, we recall  several properties of generalized singular value functions, and define noncommutative   symmetric  spaces.
Throughout this paper, $\cH$ denotes a Hilbert space, which may not necessarily be separable, and $B(\cH)$ denotes the $*$-algebra of all bounded linear operators acting on $\cH$, where $ {\bf 1 }$ denotes the identity operator on $\cH $. Let $\cM$ be a von Neumann algebra over $\cH$.
Denote by  $P(\cM)$   the set of all projections of \( \cM \). We denote by \( \cM_p \) the reduced von Neumann algebra \( p\cM p \) generated by a projection \( p \in P(\cM) \). For a detailed exposition  of von Neumann algebra theory, we refer to  \cite{Blackadar}, \cite{Dixmier}, \cite{KR} or \cite{Sakai71,Tak}.
For general information about measurable operators, we refer to  \cite{Nelson}, \cite{Se}, \cite{DP2014}, \cite{DPS} and \cite{T3}.

\section{$\tau$-Measurable operators and generalized singular values}\label{s:p1}

A linear operator $x:\mathfrak{D}\left( x\right) \rightarrow \cH $ is termed {\it measurable} with respect to $\mathcal{M}$ if $x$ is
closed, densely defined, affiliated with $\mathcal{M}$ and there exists a sequence $\left\{ p_n\right\}_{n=1}^{\infty}$ in the set
$\cP\left(\mathcal{M}\right)$ of all projections of $\mathcal{M}$ such that $p_n\uparrow \mathbf{1}$,
$p_n(\cH)\subseteq\mathfrak{D}\left(x \right) $ and $\mathbf{1}-p_n$ is a finite projection (with respect to $\mathcal{M}$) for all
$n$. It should be noted that the condition $p _{n}\left(\cH\right) \subseteq \mathfrak{D}\left( x\right) $ implies that $xp _{n}\in
\mathcal{M}$. The collection of all measurable operators with respect to $\mathcal{M}$ is denoted by $S\left(
\mathcal{M} \right) $, which is a unital $\ast $-algebra with respect to strong sums and products (denoted simply by $x+y$ and $xy$ for all $x,y\in S\left( \mathcal{M}\right) $).

Let $x$ be a self-adjoint operator affiliated with $\mathcal{M}$. We denote its spectral measure by $\{e^x\}$. It is well known that
if $x$ is a closed operator affiliated with $\mathcal{M}$ with the polar decomposition $x = u|x|$, then $u\in\mathcal{M}$ and $e\in
\mathcal{M}$ for all projections $e\in \{e^{|x|}\}$. Moreover, $x \in S(\mathcal{M})$ if and only if $x $ is closed, densely, defined,
affiliated with $\mathcal{M}$ and $e^{|x|}(\lambda,\infty)$ is a finite projection for some $\lambda> 0$. It follows immediately that
in the case when $\mathcal{M}$ is a von Neumann algebra of type $III$ or a type $I$ factor, we have $S(\mathcal{M})= \mathcal{M}$.
For type $II$ von Neumann algebras, this is no longer true. From now on, let $\mathcal{M}$ be a semifinite von Neumann algebra
equipped with a faithful normal semifinite trace $\tau$.

For any closed and densely defined linear operator $x :\mathfrak{D}\left( x \right) \rightarrow \cH $, the \emph{null projection}
$n(x)=n(|x|)$ is the projection onto its kernel $\mbox{Ker} (x)$. The \emph{left support} $l(x )$ is the projection onto the closure
of its range $\mbox{Ran}(x)$ and the \emph{right support} $r(x)$ of $x$ is defined by $r(x) ={\bf{1}} - n(x)$.

An operator $x\in S\left( \mathcal{M}\right) $ is called $\tau$-measurable if there exists a sequence
$\left\{p_n\right\}_{n=1}^{\infty}$ in $\cP \left(\mathcal{M}\right)$ such that $p_n\uparrow \mathbf{1}$, $p_n\left( \cH
\right)\subseteq \mathfrak{D}\left(x \right)$ and $\tau(\mathbf{1}-p_n)<\infty $ for all $n$. The collection of all $\tau $-measurable
operators is a unital $\ast $-subalgebra of $S\left( \mathcal{M}\right) $,  denoted by $S\left( \mathcal{M}, \tau\right)$.
It is well known that a linear operator $x$ belongs to $S\left(\mathcal{M}, \tau\right) $ if and only if $x\in S(\mathcal{M})$ and
there exists $\lambda>0$ such that $$\tau(e^{|x|}(\lambda,\infty))<\infty.$$
Alternatively, an unbounded operator $x$ affiliated with $\mathcal{M}$ is  $\tau$-measurable (see \cite{FK}) if and only if $$\tau\left(e^{|x|}(n,\infty)\right)\rightarrow 0,\quad \mbox{as } n\to\infty.$$
%For any $x=x^*\in S\left( \mathcal{M}, \tau\right)$, we set $x_+=xe^{x}(0,\infty)$ and $x_-=xe^{x}(-\infty,0)$.

\begin{definition}%\label{mu}
Let a semifinite von Neumann  algebra $\mathcal M$ be equipped
with a faithful normal semi-finite trace $\tau$ and let $x\in
S(\mathcal{M},\tau)$. The generalized singular value function $\mu(x):t\rightarrow \mu(t;x)$, $t>0$,  of
the operator $x$ is defined by setting
$$
\mu(t;x)
=
\inf \left\{
\left\|xp\right\|_\infty :\ p\in \cP(\cM), \ \tau(\mathbf{1}-p)\leq t
\right\}.
$$
\end{definition}
An equivalent definition in terms of the distribution function of the operator $x$ is the following. For every self-adjoint
operator $x\in S(\mathcal{M},\tau) $,  setting $$d_x(t)=\tau(e^{x}(t,\infty)),\quad t>0,$$ we have (see e.g. \cite{FK} and \cite{LSZ})
$$
\mu(t; x)=\inf\{s\geq0:\ d_{|x|}(s)\leq t\}.
$$
Note that $d_x(\cdot)$ is a right-continuous function (see e.g.  \cite{FK} and \cite{DPS}).

Consider the algebra $\mathcal{M}=L^\infty(0,\infty)$ of all Lebesgue measurable essentially bounded functions on $(0,\infty)$.
The algebra $\mathcal{M}$ can be viewed as an abelian von Neumann algebra acting via multiplication on the Hilbert space
$\mathcal{H}=L^2(0,\infty)$, with the trace given by integration with respect to Lebesgue measure $m.$ It is easy to see that the
algebra of all $\tau$-measurable operators affiliated with $\mathcal{M}$ can be identified with the subalgebra $S(0,\infty)$ of the
algebra of Lebesgue measurable functions which consists of all functions $f$ such that $m(\{|f|>s\})$ is finite for some $s>0$. It
should also be pointed out that the generalized singular value function $\mu(f)$ is precisely the decreasing rearrangement $\mu(f)$ of the function $|f|$ (see e.g. \cite{KPS,Bennett_S,LT2}) defined by$$\mu(t;f)=\inf\{s\geq0:\ m(\{|f|\geq s\})\leq t\}.$$

For convenience of the reader,  we also recall the definition of the \emph{measure topology} $t_\tau$ on the algebra $S(\cM,\tau)$. For every $\varepsilon,\delta>0,$ we define the set
$$V(\varepsilon,\delta)=\{x\in S(\mathcal{M},\tau):\ \exists p \in \cP\left(\mathcal{M}\right)\mbox{ such that }
\left\|x(\mathbf{1}-p)\right\|_\infty \leq\varepsilon,\ \tau(p)\leq\delta\}.$$ The topology
generated by the sets $V(\varepsilon,\delta)$,
$\varepsilon,\delta>0,$ is called the measure topology $t_\tau$ on $S(\cM,\tau)$ \cite{DPS, FK, Nelson}.
It is well known that the algebra $S(\cM,\tau)$ equipped with the measure topology is a complete metrizable topological algebra \cite{Nelson}.
We note that a sequence $\{x_n\}_{n=1}^\infty\subset S(\cM,\tau)$ converges to zero with respect to measure topology $t_\tau$ if and only if $$\mbox{$\tau\big( e  ^{|x_n|}(\varepsilon,\infty)\big)\to 0$ as $n\to \infty$}$$ for all $\varepsilon>0$ \cite{DPS}.

The space  $S_0(\cM,\tau)$ of $\tau$-compact operators is the space associated to the algebra of functions from $S(0,\infty)$ vanishing at infinity, that is,
$$S_0(\cM,\tau) = \{x\in S(\cM,\tau) :  \ \mu(\infty; x) =0\}.$$
The two-sided ideal $\cF(\tau)$ in $\cM$ consisting of all elements of $\tau$-finite range is defined by
$$\cF(\tau)=\{x\in \cM ~:~ \tau(r(x)) <\infty\} = \{x \in \cM ~:~ \tau(s(x)) <\infty\}.$$
Note that  $S_0(\cM,\tau)$ is the closure of $\cF(\tau)$ with respect to the measure topology \cite{DP2014}.

A further important vector space topology on $S(\cM,\tau)$ is the \emph{local measure topology} \cite{DP2014,DPS}.
A neighbourhood base for this topology is given by the sets $V(\varepsilon, \delta; p )$, $\varepsilon, \delta>0$, $p\in \cP(\cM)\cap \cF(\tau)$, where
$$V(\varepsilon,\delta;  p ) = \{x\in S(\cM,\tau): pxp \in V(\varepsilon,\delta)\}. $$
It is clear that the local  measure topology is weaker than the measure topology~\cite{DP2014,DPS}.
If $\{x_\alpha\}\subset S(\cM,\tau)$ is a net and if $x_\alpha \rightarrow_\alpha x \in S(\cM,\tau)$ in local measure topology, then $x_\alpha y\rightarrow xy $ and $yx _\alpha \rightarrow yx $ in the local measure topology for all $y \in S(\cM,\tau)$ \cite{DP2014,DPS}.
If $\{ a_i\} $ is an increasing net of positive elements  in $S(\cM,\tau)$ and if $a\in S(\cM,\tau)$ is such that $a=\sup a_i$, then we write $0\le a_i \uparrow a$~\cite[p.212]{DP2014}.
If $\{x_i\}$ is an increasing net in $S(\cM,\tau)_+$ and $x \in S(\cM,\tau)_+$ such that $x_i\to x$ in the local measure topology, then $x_i\uparrow x$ (see e.g. \cite[Chapter II, Proposition 7.6 (iii)]{DPS}).

\section{Symmetric  spaces of $\tau$-measurable operators}

We now come to the definition of the main object of this paper.
\begin{definition}\label{opspace}
Let $\cM $ be a semifinite von Neumann  algebra  equipped
with a faithful normal semi-finite trace $\tau$.
Let $\mathcal{E}$ be a linear subset in $S({\mathcal{M}, \tau})$
equipped with a complete norm $\left\|\cdot \right \|_{\mathcal{E}}$.
We say that
$\mathcal{E}$ is a \textit{symmetric    space}  if
for $x \in
\mathcal{E}$, $y\in S({\mathcal{M}, \tau})$ and  $\mu(y)\leq \mu(x)$ imply that $y\in \mathcal{E}$ and
$\left\|y\right\|_\mathcal{E}\leq \left\|x\right\|_\mathcal{E}$.
\end{definition}
Let $E(\cM,\tau)$ be a symmetric   space.
%Since $\mu(axb)\le \mu(\norm{a}_\infty \norm{b}_\infty x) $, $a,b\in \cM$, $x\in E(\cM,\tau)$, it follows that
%every symmetric  space is an $\cM$-bimodule.
It is well-known that any symmetrically normed space $E(\cM,\tau)$ is a normed $\cM$-bimodule (see e.g.  \cite{DP2014} and \cite{DPS}). %that is $AXB\in E$ for any $X\in E$, $A,B\in \cM$ and $\|AXB\|_E\leq \|A\|_\infty\|B\|_\infty\|X\|_E$, where $\|\cdot\|_\infty$ is the uniform operator norm.
That is, for any symmetric operator space $E(\cM,\tau)$, we have
$\left\| axb\right\|_E \le  \left\|a\right\|_\infty \left\|b\right\|_\infty \left\|x\right\|_E, ~a, b \in \cM,~ x\in E(\cM,\tau)$.
It is known that whenever $E(\cM,\tau)$ has   order continuous norm $\norm{\cdot}_E$, i.e.,  $\norm{x_\alpha}_E\downarrow 0$ whenever $0\le x_\alpha \downarrow 0\subset E(\cM,\tau)$, we  have $E(\cM,\tau)\subset S_0(\cM,\tau)$~\cite{DPS,HSZ,DP2014}.
 A symmetric space $E(\cM,\tau)\subset S(\cM,\tau)$ is said to have the Fatou property if for every upwards directed net  $\{a_\beta\}$ in \(E(\cM,\tau)^+\), satisfying $\sup_\beta \norm{a_\beta}_E < \infty$,
there exists $a \in E(\cM,\tau)^+$ such that $a_\beta \uparrow a$ in $E(\cM,\tau)$ and $\norm{a}_E = \sup_\beta \norm{a_\beta}_E$.

 The so-called K\"{o}the dual is identified with an important part of the dual space.
If $E (\cM,\tau) \subset S(\cM,\tau)$ is a symmetric space, then the K\"{o}the dual $E(\cM,\tau)^\times $ of $E$ is defined by setting
$$ E(\cM,\tau)^\times =\left\{   x\in S(\cM,\tau) : \sup_{\|y\|_E\le 1, y\in E(\cM,\tau)}\tau (|xy|)   <\infty    \right\}.$$
%The K\"{o}the dual $E(\cM,\tau)^\times  $ can be identified as   a subspace of the Banach dual $E(\cM,\tau)$ via the trace duality~\cite[p.228]{DP2014}.

If $y \in E(\cM,\tau)^\times$, then the linear functional $$\phi_y:x \mapsto \tau(xy) , ~ x \in E(\cM,\tau),$$ is continuous on $E(\cM,\tau)$, see \cite[Proposition 22 (i).]{DP2014}. In addition, $\norm{\phi_y}_{E(\cM,\tau)^*}=\norm{y}_{E(\cM,\tau)^\times}$, where  $E(\cM,\tau)^*$ is the dual of  $E(\cM,\tau)$ (see, e.g. \cite{DP2014,DPS}). The K\"{o}the dual $E(\cM,\tau)^\times  $ can be identified as   a subspace of the Banach dual $E(\cM,\tau)$ via the trace duality~\cite[p.228]{DP2014}. If $y \in E(\cM,\tau)^\times$, then the functional $ \phi_y$ is normal, that is, $ x_\alpha \downarrow 0 $ in $ E(\cM,\tau)$ implies that $ \phi_y(x_\alpha) \rightarrow_\alpha 0 $. Note that every normal functional $ \phi \in E(\cM,\tau)^*$ is of the form $\phi_y$ for some $ y \in E(\cM,\tau)^\times $ (see e.g. \cite[Theorem 37.]{DP2014}). Since $E(\cM,\tau)^ \times$ separates the point of $E(\cM,\tau)$ (see \cite[Corollary 4.3.9]{DPS}), it follows that the weak topology $\sigma(E(\cM,\tau),E(\cM,\tau)^\times)$ is a Hausdorff topology.  For simplicity, we use $\sigma(E,E^\times)$ instead of $\sigma(E(\cM,\tau),E(\cM,\tau)^\times)$.

%It is known that there is a similar characterization of the Fatou property for a strongly symmetrically normed space $E(\cM,\tau)\subset S(\cM,\tau)$. That is $E(\cM,\tau)$ has the Fatou property if and only if $E=E^{\times\times}$ and $\norm{x}_E=\norm{x}_{E^{\times\times}}$ for all $x\in E(\cM,\tau)$.

Recall that $$x\in L_1(\cM,\tau)+\cM:=\{a\in S(\cM,\tau):\mu(a)\in L_1(0,\infty)+L_\infty(0,\infty)\}$$ can be equipped with a norm $\norm{x}_{L_1+L_\infty}= \int_0^1 \mu(s;x)ds $
and $$x\in L_1(\cM,\tau)\cap\cM:=\{a\in S(\cM,\tau):\mu(a)\in L_1(0,\infty)\cap L_\infty(0,\infty)\}$$ can be equipped with a norm $\norm{x}_{L_1\cap L_\infty}:=\max \{\norm{x}_1,\norm{x}_\infty\}$.
In particular, we have $\left(L_1(\cM,\tau)\cap\cM\right)^\times= L_1(\cM,\tau)+\cM$ and $L_1(\cM,\tau)\cap\cM= \left( L_1(\cM,\tau)+\cM\right)^\times$~\cite[Example 4]{DP2014}.

%{\color{blue}In this case, whenever  $E(\cM,\tau)$ has order continuous norm, then $E(\cM,\tau)^\times $ is isometrically isomorphic to $E(\cM,\tau)^*$ (see e.g. \cite{DDP93}, \cite[Proposition 6.4]{DP2012} or \cite[Proposition 47(v)]{DP2014}).}

%The fact that every (normed) symmetric operator space $\mathcal{E}$ is (anabsolutely solid) $\mathcal{M}$-bimodule of $S\left(\mathcal{M},\tau \right)$ is well known (see e.g. \cite{Kalton_S, SCh_90} and references therein).

%It is clear that in the special case, when $\cM=L_\infty(0,1)$, or $\cM=L_\infty(0,\infty)$, or $\cM=l_\infty$, the definition of symmetric $\Delta$-normed operator spaces coincides with  definition of the symmetric function (or sequence) spaces.%
%In the case, when $\mathcal{M}=B(H)$ and $\tau$
%is a standard trace ${\rm Tr}$, we shall call a symmetric $\Delta$-normed operator
%space introduced in Definition \ref{opspace} a symmetric $\Delta$-normed operator ideal (for the symmetrically normed ideals we refer to \cite{GK1, GK2,
%Simon}).

There exists a strong connection between symmetric function spaces (see \cite{RGMP,LT2,KPS}) and
operator spaces exposed in \cite{Kalton_S} (see also \cite{LSZ}).
The operator space $E(\cM,\tau)$ defined by
\begin{equation*}
E(\mathcal{M},\tau):=\{x \in S(\mathcal{M},\tau):\ \mu(x )\in E(0,\infty)\},
\ \left\|x \right\|_{E(\mathcal{M},\tau)}:=\left\|\mu(x )\right\|_E\}
\end{equation*}
 is a complete symmetric  space  whenever $(E(0,\infty),\left\|\cdot\right\|_E)$ is    a complete  symmetric  function space on $(0,\infty)$  \cite{Kalton_S}.
In particular, for any symmetric function space $E(0,\infty)$, we have \cite[Lemma 18]{DP2014} $$F(\tau)\subset E(\cM,\tau).$$
In the special case when $E(0,\infty)=L_p(0,\infty)$, $1\le  p\le \infty$, $E(\cM,\tau)$ is the noncommutative $L_p$-spaces affiliated with $\cM$ and we denote the norm by $\norm{\cdot}_p$.
We note that if $E(0,\infty )$ is separable (i.e. has order continuous norm), then $E^\times (\cM,\tau) $ is isometrically isomorphic to $E(\cM,\tau)^*$~\cite[p.246]{DP2014}.
Recall that every  symmetric sequence/function  space $E$ in the sense of Lindenstrauss and Tzafriri  is fully symmetric, that is,
 if $x\in E$ and $y\in \ell_\infty$ (resp. $y\in S(0,\infty)$) with
 $$ \int_0^t \mu(t;y)dt \le \int_0^t\mu(t;x)dt, ~t\ge 0 $$
 (denoted by $y\prec \prec x$),
 then $y\in E$ with $\norm{y}_E\le \norm{x}_E$ (see e.g.  \cite[Proposition 2.a.8]{LT2}, \cite[Chapter II, Theorem 4.10]{KPS} or \cite[Theorem 4.5.7]{DPS}).

The \emph{carrier projection} $c_E\in \cM$ of   $E(\cM,\tau)$ is defined by setting
$$c_E := \vee \{p:p\in \cP(E)\}.$$
 It is clear that $c_E$ is in the center $Z(\cM)$ of $\cM$, see \cite{DPS} or \cite{DP2014}.
It is often assumed that the carrier projection $c_E$ is equal to ${\bf 1}$.
Indeed, for any  symmetric function  space  $E(0,\infty)$, the carrier projection of the corresponding operator space $E(\cM,\tau)$ is always ${\bf 1}$ (see e.g.
\cite{DP2014,Kalton_S}).
On the other hand, if $\cM$ is atomless or is atomic and all atoms have equal trace, then any non-zero symmetric space $E(\cM,\tau)$ is necessarily $\bf 1$~\cite{DP2014,DPS}.

Recall the following well-known fact.
\begin{lemma}\label{dense}
Let $\cM$ be a   von Neumann algebra equipped with a  semifinite  faithful normal trace $\tau$.
Let $E(\cM,\tau)  $ be a symmetric space affiliated with $\cM$.
Let $\{e_i\}$ be a net of $\tau$-finite projections in $\cM$ increasing to ${\bf 1}$. We have 
$\cup_{i}e_i \cM e_i$ is dense in   $E(\cM,\tau)$ with respect to the $\sigma(E,E^\times)$-topology.
%
%Then, $E(\cM,\tau)\cap \cM$ is dense in $E(\cM,\tau)$ with respect to the $\sigma(E,E^\times)$-topology.
\end{lemma}
\begin{proof}
Let $x\in E(\cM,\tau)\cap \cM$. 
We claim that $e_i x e_i \to_i x $ in the $\sigma(E,E^\times)$-topology.
For any $y \in E(\cM,\tau)^\times $, we have 
$$\tau(ye_i x e_i  - yx ) =\tau(ye_i x e_i -ye_ix  ) +\tau(ye_ix-yx). $$
Clearly, $\tau(ye_ix-yx)\to_i 0$. 
On the other hand, we have 
$$\mu(2t; ye_i x e_i -ye_ix  ) \stackrel{\mbox{\tiny \cite[Corollary 2.3.16]{LSZ}}}{ \le} \mu(t; ye_i )\mu(t;  x({\bf 1} -e_i)) \le \mu(t; y )\mu(t;  x({\bf 1} -e_i))  $$
and 
$$\mu(2t; ye_i x -e_i ye_ix  ) \stackrel{\mbox{\tiny \cite[Corollary 2.3.16]{LSZ}}}{ \le} \mu(t; ({\bf 1} -e_i) y )\mu(t; e_i x) \le  \mu(t; ({\bf 1} -e_i) y )\mu(t;  x)  .$$
Note that either $x$ or $y$ is $\tau$-compact. 
By \cite[Proposition 2.6.4]{DPS}, either $\mu(ye_i x e_i -ye_ix  )$ or $\mu(ye_i x -e_i ye_ix   ) $ converges to $0$ in measure. 
Therefore, by \cite[Theorem 3.4.21]{DPS}, we have $\tau(ye_i x e_i -ye_ix  ) \to_i 0$. 
This proves the claim.  

Now, 
it suffices to observe that 
 $E(\cM,\tau)_+\cap \cM_+$ is dense in $E(\cM,\tau)_+$ with respect to  the $\sigma(E,E^\times)$-topology.
Let $x_n=x e^{x}(0,n)$, $n\ge 1$.
We have $x-x_n\downarrow 0$ as $n\to \infty $.
For any $y = y_1 -y_2 +i y_3 -i y_4 \in E(\cM,\tau)^\times  $, where $y_i$ is positive for each $i=1,2,3,4,$ we have 
$$\tau( (x-x_n) y_i ) =\tau\left(
y_i^{1/2}(x-x_n) y _i^{1/2}
\right  ) \to 0$$
as $n\to \infty $, that is, 
$x_n\to x$  as $n\to \infty $ with respect to the  $\sigma(E,E^\times)$-topology.
\end{proof}

\section{Hermitian operators}

Let $X$ be a Banach space.
Recall that a \emph{semi-inner product} (abbreviated \emph{s.i.p.}) on $X$ is a mapping $\langle \cdot,\cdot\rangle$ of $X\times X$ into the field   of complex numbers such that:
\begin{enumerate}
  \item $\langle x+y,z\rangle= \langle x,z\rangle+\langle y,z\rangle$ for $x,y,z\in X$;
  \item $\langle \lambda x,y \rangle=\lambda \langle x,y \rangle$ for $x,y\in X$ and $\lambda\in \mathbb{C}$;
  \item $\langle x,x\rangle >0$ for $0\ne x\in X$;
  \item $|\langle  x,y\rangle|^2 \le \langle x,x\rangle\langle y,y\rangle$ for any $x,y\in X$.
\end{enumerate}
When a s.i.p is defined on $X$, we call $X$ a \emph{semi-inner-product space} (abbreviated \emph{s.i.p.s.}).
If $X$ is a s.i.p.s., then $\langle x,x \rangle^{\frac12}$
 is a norm on $X$.
 On the other hand, every Banach space can be made into a s.i.p.s. (in general, in infinitely many ways) so that the s.i.p. is consistent with the norm, i.e., $\langle x,x\rangle^\frac 12 = \norm{x}$ for any $x\in X$~\cite{FJ}.
 By virtue of the Hahn--Banach theorem, this can be accomplished by choosing  one bounded linear functional $f_x$ for each $x\in X$  such that $\norm{f_x} =\norm{x}$
and $f_x(x)=\norm{x}^2$ ($f_x$ is called a \emph{support functional} of $x$), and then setting $\langle x,y\rangle = f_y (x)$ for arbitrary $x,y\in X$~\cite{Berkson, Lumer,FJ,Giles}.
Given a linear transformation $T$ mapping a s.i.p.s. into itself, we denote by $W(T)$ the \emph{numerical range} of $T$, that is,  $\{\langle Tx,x\rangle| \langle x,x\rangle=1, x\in X\}$.
Let $T$ be an operator on a Banach space $(X,\norm{\cdot})$. Although in
principle there may be many different s.i.p. consistent
 with $\norm{\cdot}$,
 nonetheless if the numerical range of $T$ relative to one such
s.i.p. is real, then the numerical range relative to any such s.i.p. is real  (see e.g. \cite[p.107]{FJ}, \cite[Section 6]{Lumer} and \cite[p.377]{Berkson}).
If this is the case, $T$ is said to be a \emph{hermitian} operator on $X$.

\chapter{The ball topology and weak continuity of isometries}\label{S:ball}
The main purpose of this Chapter is to prove the $\sigma(E, E^\times)$-continuity of isometries/hermitian operators on a noncommutative symmetric space $E(\cM,\tau)$ (see Theorems \ref{weak-continous of T} and  \ref{continous of hermitian} below).
The reader who is only concerned with separable Banach spaces may
observe that any bounded linear operator on   $E(\cM,\tau)$ is automatically $\sigma(E, E^\times)$-continuous when
 $E(\cM,\tau)$ has order continuous norm
 (in particular, it is separable).
%  skip this section.
%
%Readers focusing solely on the separable case can note that if the symmetric operator space is separable, it must be order
%continuous, thus ensuring that any isometry $T$ is $\sigma(E,E^\times)$)-continuity,
However, for the case when $E(\cM,\tau)$ does not have  order continuous norm,
the  $\sigma(E,E^\times)$-continuity of an isometry/hermitian operator $T$ on $E(\cM,\tau)$ is not obvious.

 Kalton and Randrianantoanina~\cite[Proposition 2.5]{KR2}  obtained  Theorem \ref{continous of hermitian} below for a K\"{o}the function space with the Fatou property.
 Recently,  Aminov and Chilin~\cite[Proposition 4.3]{AC2}  treated  perfect norm ideals of compact operators on a separable Hilbert space.
The main tool is the so-called ball topology of a Banach space introduced and developed by Godefroy and Kalton~\cite{GN}.

Let $(X, \left\|\cdot\right\|_X)$ denote a real Banach space, and let $b_X$ denote its ball topology, which is the coarsest topology such that every closed ball
\[ B(x, \varepsilon) = \{ y \in X : \left\|y - x\right\|_X \leq \varepsilon \}, \quad \varepsilon > 0, \]
is closed in $b_X$, see e.g. \cite{GN}. The family
\[ X \setminus \bigcup_{i=1}^{n} B(x_i, \varepsilon_i), \quad x_i \in X, \quad \left\|x_0 - x_i\right\|_X > \varepsilon_i, \quad i = 1, \ldots, n, \quad n \in \mathbb{N}, \]
forms a neighborhood basis of the point $x_0 \in X$ in $b_X$. Consequently, 
we have the following characterization of convergence in the ball topology.
\begin{lemma} \label{ball continuous} \cite[p.200]{GN} Let $X$ be a Banach space.
Let    $\{x_\alpha\}$ in  a net  in $X$ and let $x\in X$. 
Then, $x_\alpha  \xrightarrow{b_X} x$ if and only if
\[ \liminf \left\|x_\alpha - y\right\|_X \geq \norm{x - y}_X \]
for all $y \in X$. In particular, every  isometry $T$ from  $(X, \norm{\cdot}_X)$  
into another Banach space $(Y,\norm{\cdot}_Y)$
is  $b_X-b_Y$-continuous.
\end{lemma}

It is known that   the ball topology is not Hausdorff~\cite{GN}. 
The following theorem  compensates for this deficiency. 
Recall that an absolutely convex subset $A$ of a Banach space $X$ is called a Rosenthal set if it is bounded and contains no basic sequence equivalent to the unit vector basis of $\ell_1$ (equivalently,
  every sequence in $ A $ possesses a weakly Cauchy subsequence~\cite{Rosenthal}).

\begin{theorem}\label{theorem Rosenthal set}
\cite[Theorem 3.3]{GN} Let $(X,\norm{\cdot})$ be a real/complex Banach space,and let $A$ be a bounded absolutely convex Rosenthal subset of $X$. Then $(A,b_{X})$ is a Hausdorff space.
\end{theorem}

%Note that in a finite von Neumann algebra $\cM$ equipped with a faithful normal tracial state $\tau$, we have $\mu(a+b)>\mu(b)$ for any $0 < a,b\in S(\cM,\tau)$ \cite[33]{HSZ}. Therefore, we have the following lemma:% because $\tau(a+b)>\tau(b)$

 Before proceeding to the proof of the main theorem of this chapter,  we need several auxiliary results.
\begin{lemma}\label{lemma singular with uparrow}
  For any  finite von Neumann algebra $\cM$ equipped with a faithful normal tracial state $\tau$, if $$\mbox{$0\le x_\lambda \uparrow_\lambda \le    x\in S(\cM,\tau)$ and $\mu(x_\lambda )\uparrow \mu(x)$,}$$ then we have 
$$x_\lambda\uparrow x.$$
\end{lemma}
\begin{proof}
Assume by contradiction that
$0\le x_\lambda \uparrow   x\in S(\cM,\tau)$ and $\mu(x_\lambda )\uparrow \mu(x)$ but there exists a positive non-zero element
 $b\in S(\cM,\tau)$ such that  $$x_\lambda +b \le x$$ for all $\lambda $.
By \cite[Proposition 2(ii)]{DP2014}, 
we have 
$$a=\sup_\lambda x_\lambda $$ exists in $S(\cM,\tau)$.
By \cite[Lemma 4.3]{HSZ} (see also \cite[Proposition 2.2]{CKS2}), 
we have $$  \mu(x_\lambda)\le \mu(a )   \stackrel{\tiny \mbox{\cite[Lemma 4.3]{HSZ}}}{<}  \mu(a+b)\le \mu(x) $$
for all $\lambda$. That is,
$$\sup_\lambda \mu  (x_\lambda ) \le \mu(a)<\mu(x),$$
 which is a contradiction.
\end{proof}
% can be done by taking $\frac{t}{1+t}$ (operator monotone).

The following result is certainly known to experts. However, due to the lack of a suitable reference, we provide a complete proof below.
\begin{lemma}\label{lemma net and sequence}
Let $\cM$ be a $\sigma$-finite semifinite von Neumann algebra equipped with a semifinite faithful normal trace $\tau$. For any
net $\{x_\lambda \}$ in $S(\cM,\tau)$ such that
$x_\lambda\uparrow x \in S(\cM,\tau)$ (respectively,
 $x_\lambda \downarrow 0$),
 there exists a  sequence $\{x_n\}\subset \{x_\lambda \}$ such that  $ x_n\uparrow x$ (respectively, $x_n\downarrow 0$).

Consequently,
for any semifinite (not necessarily $\sigma$-finite) von Neumann algebra  $\cM$ equipped with a semifinite faithful normal trace $\tau$, and for any
net $\{x_\lambda \}$ in $S(\cM,\tau)$ such that   $x_\lambda\uparrow x \in   S_0(\cM,\tau)$ (respectively, $S_0(\cM,\tau)\ni x_\lambda \downarrow 0$), there exists a  sequence $\{x_n\}\subset \{x_\lambda \}$ such that 
$x_n\uparrow x \in S_0(\cM,\tau)$ (respectively, $S_0(\cM,\tau)\ni x_n \downarrow 0$).
\end{lemma}
 \begin{proof}It suffices to consider the case when  $x_\lambda \uparrow x$.
  Since $\cM$ is $\sigma$-finite, it follows that there exists an increasing sequence $\{p_n\}_{n\ge 1}$ of $\tau$-finite projections  with $p_n\uparrow {\bf 1}$.
Note that for any $0\le y\in S(\cM,\tau)$, we have  $$\mu\left(y p_n y \right) = \mu \left( (p_ny)^*   p_n y \right)  \stackrel{\mbox{\tiny 
\cite[Prop. 3.2.10]{DPS}}}{=}\mu\left( p_n  y (p_ny)^* \right)  =\mu\left(p_n y^2 p_n\right). $$

 For each $n$,  we have $$ p_n x_\lambda p_n \uparrow_\lambda   p_n x p_n$$ (in particular, 
 $ p_n x_\lambda p_n \to_\lambda   p_n x p_n$
in measure topology~\cite[Proposition 2(iv)]{DP2014}).
 Note that
 $$\mu(p_n x_\lambda p_n)\le \mu(x_\lambda)\le \mu(x).$$
Let  $x_1 $ be an  element in $\{x_\lambda \}$ such that
$$ m\left\{|\mu(p_{1 }x_{ 1}p_{1 }) -\mu(p_{1}xp_{1} )|>\frac{1}{2 }\right\}<\frac{1}{2 } . $$
There exists an element $x_2 \ge x_1$ in $\{x_\lambda \}$ such that
$$\sup_{1\le k\le 2}m\left\{|\mu(p_{k }x_{2}p_{k }) -\mu(p_{k}xp_{k} )|>\frac{1}{2^{2 }}\right\}<\frac{1}{2^{2}} . $$
Constructing inductively, we obtain an increasing  sequence $\{x_n\}_{n\ge 1}\subset \{x_\lambda \}$
  such that
$$\sup_{1\le k\le n}m\left\{|\mu(p_{k }x_{n}p_{k }) -\mu(p_{k}xp_{k} )|>\frac{1}{2^{n }}\right\}<\frac{1}{2^{n }} . $$
In particular, for each $k$, we have
 $$ \lim_{n\to \infty } m\left\{|\mu(p_{k }x_{n }p_{k }) -\mu(p_{k}xp_{k} )|>\frac{1}{2^{n}}\right\} = 0 . $$
Since $ p_kx_np_k \uparrow \le  p_k xp_k$~\cite[Proposition 1(iii)]{DP2014}, it follows from Lemma \ref{lemma singular with uparrow} that for each $k$, we have 
\begin{align}\label{uparrow xn}
 p_kx_np_k \uparrow    p_k xp_k 
 \end{align}
as $n\to \infty$. 

Assume by contradiction that there exists a non-zero positive element $a\in S(\cM,\tau)$ such that
\begin{align}\label{contr  a}
x_n +a \le x
\end{align}
for all $n\ge 1$.
Since $p_k\to {\bf 1}$, it follows that there exists a sufficiently large $k $ such that
$p_k a p_k \ne 0$.
We have
$$p_k x p_k  \stackrel{\eqref{contr  a}}{\ge}   p_k \left(x_n   +  a\right)  p_k 
\stackrel{\eqref{uparrow xn}}{\uparrow} p_k x p_k + p_k ap_k   
  ,$$
 i.e., $p_k ap_k \le 0$, which is a contradiction.

 The second assertion follows from the facts that the left and the right supports of an element from $S_0(\cM,\tau)$
are  necessarily $\sigma$-finite~\cite[Lemma 5.5.8]{DPS}, and $l(x_\lambda ) \le l(x)$ and $r(x_\lambda ) \le r(x)$.
 \end{proof}
% In particular,  $p_n x_k p_n\uparrow _k p_n x p_n$ for each $n$ (indeed, if not, then there exists $a>0$ such that $p_n x_k p_n+a\uparrow b+a  \le p_nxp_n$. That is, $\mu(p_nx_kp_n)\le \mu(b)<\mu(p_nxp_n)$, which is a contradiction).
%

The  commutative counterpart of the following lemma is well-known, see e.g. \cite{KR2}.   For the sake of completeness, we provide a  full  proof below.
\begin{lemma}\label{lemma adjoint operator}

Let $\cM_1$ and $\cM_2$ be  semifinite von Neumann algebras equipped with semifinite faithful normal traces $\tau_1$ and $\tau_2$, respectively.
Let $E(\cM_1,\tau_1)$ and $F(\cM_2,\tau_2)$ be two noncommutative strongly  symmetric  spaces.   
A bounded linear operator $T$ from $E(\cM_1,\tau_1)$ to $F(\cM_2,\tau_2)$  is  $\sigma(E,E^\times)-\sigma(F,F^\times)$-continuous if and only if $T^*$ maps $F(\cM_2,\tau_2)^\times $ into $E(\cM_1,\tau_1)^\times$\footnote{Here, $T^*:F(\cM_2,\tau_2)^*\to E(\cM_1,\tau_1)^*$ stands for  the adjoint operator of $T$.}.
\end{lemma}

\begin{proof}
$(\Leftarrow)$
Let $y\in F(\cM_2,\tau_2)^\times$.  
By assumption, we have
$$y (T)=T^*(y) \in E(\cM_1,\tau_1)^\times  . $$
For any $x_\lambda \to 0$ in  $\sigma(E,E^\times)$-topology,
 we have
$$\tau_2(yT (x_\lambda ))= T^*(y)( x_\lambda ) \to 0. $$
That is, $T(x_\lambda) \to 0$ in  $\sigma(F,F^\times)$-topology.

$(\Rightarrow)$
Assume that $T$ is  $\sigma(E,E^\times)-\sigma(F,F^\times)$-continuous. Let
$y$ be an arbitrary element in $ F(\cM_2,\tau_2)^\times $.
By \cite[Theorem 37]{DP2014}, it suffices to prove that
$$ T^*(y )   $$
is a normal functional on $E(\cM_1,\tau_1)$.

Let $\{x_\lambda \} $ be a net in $E(\cM_1,\tau_1)$ with   $x_\lambda \downarrow 0$. 
Let $z$ be an arbitrary element in $ E(\cM_1,\tau_1)^\times $.
There exists a decomposition
$$z = z_1 -z_2 +i (z_3 -z_4),$$
where $z_i$, $i=1,2,3,4$, is a positive element in $E(\cM_1,\tau_1)^\times $.
Note that $$\tau_1(z_i x_\lambda )\stackrel{\tiny \mbox{\cite[Proposition 3.4.30]{DPS}}}{=}\tau_1(z_i^{1/2}x_\lambda z_i^{1/2})\stackrel{\tiny \mbox{\cite[Proposition 1(vi)]{DP2014}}}{\downarrow} 0$$ for $i=1,2,3,4$.
 We have  $$\tau_1(zx_\lambda ) =\tau_1(z_1 x_\lambda )  -\tau_1(z_2  x_\lambda)+i \cdot \tau_1 (z_3 x_\lambda )  -i \cdot   \tau_1(z_4 x_\lambda ) \to  0, $$ i.e., $x_\lambda \to 0$ in  $\sigma(E,E^\times)$-topology. By the $\sigma(E,E^\times)-\sigma(F,F^\times)$-continuity  of $T$, we have
$$T^*(y) (x_\lambda )  =\tau_2(yT (x_\lambda )) \to 0. $$
That is, $T^*(y)$ is a normal functional on $  E(\cM_1,\tau_1)  $~\cite[Definition 9]{DP2014}, which completes the proof.
\end{proof}

For  Banach spaces $X$ and $Y$, we denote by $\cB(X,Y)$ the space of all bounded linear operators on
from $X$ into $Y$.
The following proposition  is a semifinite version of~\cite[Proposition 2.4]{AC2}. 
\begin{proposition}\label{proposition weak continous} 
Let $\cM_1$ and $\cM_2$ be  semifinite von Neumann algebras equipped with semifinite faithful normal traces $\tau_1$ and $\tau_2$, respectively.
Let $E(\cM_1,\tau_1)$ and $F(\cM_2,\tau_2)$ be two  noncommutative  strongly  symmetric   spaces.   
Let $T_{n}\in\mathcal{B}(E(\cM_1,\tau_1) , F(\cM_2,\tau_2))$, $n\ge 1$, and $T \in \mathcal{B}(E(\cM_1,\tau_1),F(\cM_2,\tau_2))$. 
If $T_n$'s
are $\sigma(E,E^{\times})-\sigma(F,F^\times)$-continuous  and $$
 \left\|T_{n}-T\right\|_{\mathcal{B}(E(\cM_1,\tau_1),F(\cM_2,\tau_2)}\rightarrow0  \mbox{ as } n\rightarrow\infty,$$ then $T$ is also $\sigma(E,E^{\times})-\sigma(F,F^\times)$-continuous.
\end{proposition}

\begin{proof}
By Lemma \ref{lemma adjoint operator}, it suffices to show that $$T^{*}(f)\in E(\cM_1,\tau_1)^{\times}$$ for any functional $f\in F(\cM_2,\tau_2) ^\times $.
 Since $T_{n}\in\mathcal{B}(E(\cM_1,\tau_1) , F(\cM_2,\tau_2))$ are $\sigma(E,E^{\times})-\sigma(F,F^\times)$-continuous operators, it follows from Lemma \ref{lemma adjoint operator}
 that  $T_{n}^{*}(f)\in E (\cM_1,\tau_1)^{\times} $ for any functional $f\in F (\cM_2,\tau_2)^{\times}$, $n\in\mathbb{N}$.
Noting that   $E (\cM_1,\tau_1)^{\times}$ is a
  closed subspace in $(E(\cM_1,\tau_1)^{*},\left\|\cdot\right\|_{E^{*}})$~\cite[p.319]{DPS} and
\begin{eqnarray*}
\left\|T^{*}_{n}(f)-T^* (f)\right\|_{E^{*}}
&\leq&\left\|T_{n}^{*}-T^{*}\right\|_{\mathcal{B}(F(\cM_2,\tau_2)^*,E(\cM_1,\tau_1)^* )}\left\|f\right\|_{F^{*}}\\
&\stackrel{\mbox{\tiny\cite[Chapter VI, Prop. 1.4]{Conway}}}{=}& \left\|T_{n}-T\right\|_{\mathcal{B}(E(\cM_1,\tau_1),F(\cM_2,\tau_2) ) }\left\|f\right\|_{F^{*}}\\
&\rightarrow&0 
\end{eqnarray*} \mbox{as $n\to \infty$, }
we conclude that $T^{*}(f)\in  E (\cM,\tau)^{\times}  $.
\end{proof}

 The following proposition is known   in the setting of   symmetrically normed ideals of compact operators in $B(\cH)$ (see \cite[Proposition 2.8]{DCJ}).
\begin{proposition}\label{proposition normal functional}
Assume that $\cM$ is a semifinite von Neumann algebra equipped with a semifinite faithful normal trace $\tau$.
 Let $E(\cM,\tau)$ be a noncommutative strongly symmetric  space. 
Assume that 
\begin{enumerate}
\item either $\cM$ is   $\sigma$-finite, 
\item or $E(\cM,\tau)\subset S_0(\cM,\tau)$.
\end{enumerate}
If $\phi\in E(\cM,\tau)^*$,
then the following statements are equivalent.
\begin{enumerate}
      \item There exists a unique element $y\in E(\cM,\tau)^\times$ such that $$\phi(x)=\tau(xy), \qquad \forall  x\in E(\cM,\tau).$$
          \item $\phi(x_\lambda )\to 0$   for any net  $\{x_\lambda \}_{\lambda } \subset E(\cM,\tau)$ with $x_\lambda \downarrow0$.
      \item $\phi(x_n)\to 0$   for any sequence $\{x_n\}_{n\ge1} \subset E(\cM,\tau)$ with $x_n\downarrow0$.
\end{enumerate}
\end{proposition}%If $\phi\in E(\cM,\tau)^{*}$,then the following statements are equivalent.\\
%
%\begin{enumerate}
%    \item $\phi$ is normal, that is $x_{\alpha}\downarrow_{\alpha}0$ in $E(\cM,\tau)$ implies  $\phi(x_{\alpha})\xrightarrow{\alpha}0$.\\
%
%    \item $\phi$ is completely additive.\\
%
%    \item There exists a unique element $y\in E(\cM,\tau)^{\times}$ such that
%\begin{eqnarray}
%\phi(x)=\tau(xy),\quad x\in E(\cM,\tau).
%\end{eqnarray}
%
%In addition, if $\cM$ is $\sigma$-finite, then the above conditions are equivalent to:\\
%
%    \item $\phi(x_{n})\rightarrow0$ as $n\rightarrow\infty$ for all sequences $\{x_{n}\}\subset E(\cM,\tau)$ with $0\leq x_{n}\downarrow_{n}0$.
%\end{enumerate}

\begin{proof}
Without loss of generality, we may assume that $\phi$ is positive, see e.g. \cite[Proposition 4.2.2]{DPS}.

The equivalence $(1)\Leftrightarrow (2)$ is well-known, see e.g. \cite[Theorem 37]{DP2014} or \cite{DPS}.
The implication $(2)\Rightarrow(3)$ is trivial.
For the implication   $(3)\Rightarrow (2)$, one only observe that for any net $\{x_\lambda \}$ with $x_\lambda \downarrow 0$, there exists a sequence $\{x_n\}_{n=1}^\infty
\subset \{x_\lambda \}$
with $x_n\downarrow0$, see Lemma \ref{lemma net and sequence} above.
By statement (3), we have $\phi(x_n)\to 0$ as $n \to \infty $, which implies that $$\phi(x_\lambda )\to_\lambda 0.$$
%
%
%$(1)(\Rightarrow)(2)$ If there exists a unique element $y\in E(\cM,\tau)^\times$ such that $\psi(x)=\tau(xy)$, $x\in E(\cM,\tau)$ ), by  \cite[Theorem 37]{DP2014}, $\phi$ is normal, that is, $x_\lambda\downarrow  0$ in $E(\cM,\tau)$ implies that $\psi(x_\lambda)\to 0$. Thus, the result is obtained naturally.
%
%$(2)(\Rightarrow)(1)$ Assume by contradiction that (2) is incorrect, then by \cite[Theorem 37]{DP2014}, $\phi$ is not normal. Then there exists a net $x_{\lambda}\downarrow0$ and a indicator $\lambda_{0}$ and a number $\varepsilon>0$ such that for all $\lambda\ge \lambda_0$ always have $\phi(x_{\lambda_0})>\varepsilon$. However, by Lemma \ref{lemma net and sequence} there exits a subsequence $\{x_n\}\subset \{x_\lambda\}$ such that $x_n\downarrow0$. Hence, $\phi(x_n)\to 0$ as $n\to\infty$, which is a contradiction.
%
The proof is complete. 
\end{proof}

\begin{rem}
If any of the preceding equivalent assertions are valid, then $\psi$ is self-adjoint if and only if $y=y^*$,
and $\psi$ is positive if and only if $y\ge 0$ (see \cite[Theorem 37]{DP2014} or \cite[Theorem 5.2.9]{DPS}).

\end{rem}
%Suppose that $0\leq\phi\in E^{*}$ satisfies condition (iv) let $\phi=\phi_{n}+\phi_{s}$ where $\phi_{n}\in E_{n}^{*}$ and $\phi_{s}\in E_{s}^{*}$. Moreover $\phi_{n}\geq0$ and $\phi_{s}\geq0$,there exists a unique $0\leq a\in E^{\times}$ such that $\phi_{n}(x)=\tau(xa)$,$x\in E$. If $0\leq a\in E$, by \cite[Proposition 2.5.14]{DPS} then there exist increasing sequence $\{a_{n}\}\subset \mathcal{F}(\tau)$ such that $a_{n}\rightarrow a$ in measure, and it follows from \cite[Proposition 2.6.1(ii)]{DPS} that $a_{n}\uparrow_{n}a$. That is $a-a_{n}\downarrow0$. Since $E\neq \{0\}$ then $E^{oc}=E^{b}$ (see \cite[Proposition 5.54.8]{DPS}), then $\mathcal{F}(\tau)\subset E^{oc}$ and it follows from \cite[Proposition 5.4.3]{DPS} that $\phi(x)=0$ for all $x\in\mathcal{F}(\tau)$ and $\phi\in E_{s}^{*}$, then we have $\phi_{s}(a_{n})=0$ for all $n\in\mathbb{N}$, and
%\begin{eqnarray*}
%\phi_{s}(a)&=&\phi_{s}(a)-\lim\limits_{n\rightarrow\infty}\phi_{s}(a_{n})=\lim\limits_{n\rightarrow\infty}\phi_{s}(a-a_{n})\\
%&=&\lim\limits_{n\rightarrow\infty}(\phi(a-a_{n})-\phi_{n}(a-a_{n}))=0-0=0.
%\end{eqnarray*}
%Thus, $\phi_{s}=0$, it follows that $\phi=\phi_{n}$ that is $\phi$ is normal.

The following result 
shows that $\cF(\tau)$ is dense in a noncommutative symmetric space
in the   weak topology $\sigma(E,E^\times)$, which   is a semifinite version of \cite[Proposition 2.3]{AC2}.
\begin{proposition}\label{proposition weakly closed}
Let $E(\mathcal{M},\tau)$ be a  noncommutative  strongly symmetric space affiliated with a $\sigma$-finite
von Neumann algebra $\mathcal{M}$ equipped with a  semifinite faithful normal trace $\tau$.
 If $x\in E(\cM,\tau)$,  then there exists a sequence $\{x_{n}\}\subset\mathcal{ F}(\tau)$  such that  $$x_{n}\stackrel{\sigma(E,E^{\times})}{\longrightarrow}x.$$
\end{proposition}

\begin{proof}
Let  $a\in E(\cM,\tau)$ and  $a=v|a|$ is the polar decomposition \cite[Proposition 2.1.4]{DPS}.
By  \cite[Proposition 2.3.12]{DPS},
 there exists an upwards directed net $\{a_{\lambda}\}\subset\mathcal{ F}(\tau)$ such that $0\leq a_{\lambda}\uparrow_{\lambda}|a|$, that is, $$(|a|-a_{\lambda})\downarrow_\lambda 0.$$
 Since $\cM$ is $\sigma$-finite, it follows from Lemma \ref{lemma adjoint operator}  that there exists a subsequence $\{a_n\}_{n\ge 1}\subset \{a_\lambda\}$ such that $$(|a|-a_{n})\downarrow_n 0,$$ By Proposition \ref{proposition normal functional}, 
we have 
$$\mbox{
$\tau( (|a|-a_{n})yv )=\tau(v(|a|-a_{n})y)\rightarrow0$ as $n \to \infty  $}$$
 for any  $y\in E(\cM,\tau)^{\times}$, that is,
\begin{eqnarray*}
\mathcal{ F}(\tau)\ni a_n\stackrel{\sigma(E,E^{\times})}{\longrightarrow}|a|,
\end{eqnarray*}
equivalently,
\begin{eqnarray*}
\mathcal{ F}(\tau)\ni va_n\stackrel{\sigma(E,E^{\times})}{\longrightarrow}a,
\end{eqnarray*}
which completes the proof.
\end{proof}

For convenience, we recall the following well-known result \cite[Section 3, C.1 $\S$5]{PS2}.
\begin{proposition}\label{proposition convergent sequence}
 Let $p_{nk}$'s be real numbers, $n,k\in\mathbb{N}$, such that $\sum_{k=1}^{n}|p_{nk}|=1$ for all $n\in\mathbb{N}$.
If the limit $\lim\limits_{n\rightarrow\infty}p_{nk}=p_{k}$ exists for every fixed $k\in\mathbb{N}$, then the sequence
\begin{eqnarray*}
s_{n}=p_{n1}r_{1}+p_{n2}r_{2}+...+p_{nn}r_{n}
\end{eqnarray*}
converges for every convergent sequence $\{r_{n}\}
_{n\ge 1}$.
\end{proposition}

The following result is a semifinite version of \cite[Theorem 3.3]{AC2}.
\begin{theorem}\label{theorem: unique of element}
Let $E(\mathcal{M},\tau)$ be a  symmetric space affiliated with a   semifinite
von Neumann algebra $\mathcal{M}$ equipped with a semifinite faithful normal trace $\tau$. Let $0\leq x_{n}\downarrow0\in E(\cM,\tau)$. Then the sequence $\{x_{n}\}$   converges  with respect to the topology $b_{E(\cM,\tau)}$ to no more than one element.
\end{theorem}
\begin{proof}
It is well known that  $E(\cM,\tau)_h$ (the self-adjoint part of $E(\cM,\tau)$) is a  norm closed real subspace of $E(\cM,\tau)$ \cite[Chapter 4.1]{DPS}, 
and
the space $E(\cM,\tau)_h^*$ can be identified with the Banach dual of the real normed space $E(\cM,\tau)_{h}$ (see \cite[Chapter 4.2]{DPS}).
Let $\{x_n\}_{n\ge 1}$ be a  sequence of $E(\cM,\tau)^+$ with $x_{n}\downarrow 0$. Let $A$ be the absolutely
convex hull of $\{x_n\}_{n\ge 1}$.
By the monotonicity of the norm $\norm{\cdot}_E$, it follows that $A$ is a norm bounded subset of $E(\cM,\tau)$. Below, we prove that   $A$ is a Rosenthal subset of $E(\cM,\tau)_h$.

Recall that \cite[Proposition 4.2.2]{DPS} a linear functional $f \in E(\cM,\tau)^*_h$ can be decomposed into the difference of two positive  functionals in $E(\cM,\tau)_h^*$.
Since   $0 \leq x_{n+1} \leq x_n$, it follows that
  $\lim_{n \to \infty} f(x_n) $  exists for every functional $f\in E(\cM,\tau)_h^{*}$. For any sequence $\{y_{n}\}_{n=1}^\infty \subset A$, we have
\begin{eqnarray*}
y_{n}=p_{n1}x_{1}+p_{n2}x_{2}+...+p_{nk(n)}x_{k(n)},\quad \mbox{ where } \sum_{i=1}^{k(n)}|p_{n_i}|=1.
\end{eqnarray*}
In particular, $p_{ni} \in [-1,1]$ for all $i \in 1,\ldots,k(n)$. Consider the sequence $$q_n = (p_{n1},p_{n2},\ldots,p_{nk(n)},0,0,\ldots) \in \prod_{i=1}^{\infty} [-1,1].$$
By the Tychonoff Theorem \cite[Ch.5, Theorem 13]{KL}, the set $\prod_{i=1}^{\infty} [-1,1]$ is compact with respect to the product topology.
Moreover, by \cite[Ch.4, Theorem 17]{KL},
 $\prod_{i=1}^{\infty} [-1,1]$ is a metrizable compact set. Consequently, there exists a convergent subsequence $$\{q_{n_i}\}_{i=1}^{\infty}\subset \{q_n\}_{n=1}^\infty.$$ In particular, the numerical series $\{p_{n_i k}\}_{i\ge 1}$ converges  for every fixed $k \in \mathbb{N}$. Therefore, by Proposition \ref{proposition convergent sequence}, the sequence $$f(y_{n_i})=\sum_{j=1}^{k(n_i)} p_{n_{ij}} f(x_j)=p_{n_{i1}} f(x_1)+p_{n_{i2}} f(x_2)+\ldots+ p_{n_{ik(n_i)}} f(x_{k(n_i)}) $$
is a convergent sequence for all $f \in E(\cM,\tau)_h^*$. Consequently, the sequence $\{y_n\}_{n=1}^{\infty} \subset A$ has a weakly Cauchy subsequence $\{y_{n_i}\}$. This shows that $A$ is a Rosenthal set of $E(\cM,\tau)_h$.

By Theorem \ref{theorem Rosenthal set}, the topology $(A,b_{E(\cM,\tau)_h})$ is Hausdorff. Therefore, the sequence $\{x_n\} \subset A$ has at most  one limit with respect to the topology $b_{E(\cM,\tau)_h}$ (see e.g., \cite[Ch.2, Theorem 3]{KL}).
Note that  the restriction $b_{E(\cM,\tau)}|_{E(\cM,\tau)_h}$ is finer than $b_{E(\cM,\tau)_h}$. Indeed,
let
\begin{align*}
B_E(x,\varepsilon)=\{y\in E(\cM,\tau):\norm{y-x}\le\varepsilon\},\quad \varepsilon>0
\end{align*}
and
\begin{align*}
B_{E_h}(x,\varepsilon)=\{y\in E(\cM,\tau)_h:\norm{y-x}\le\varepsilon\},\quad \varepsilon>0
\end{align*}
be closed balls in $E(\cM,\tau)$ and $E(\cM,\tau)_h$, respectively. Obviously, $B_{E_h}(x,\varepsilon)\subset B_E(x,\varepsilon)$ for all $x\in E(\cM,\tau)_h$.  For any neighborhood of any $x_0$ in $E(\cM,\tau)_h$
in $b_{E(\cM,\tau)_h}$:
\begin{align*}
E(\cM,\tau)_h\setminus \bigcup_{i=1}^{n}B_{E_h}(x_i,\varepsilon_i),\quad x_i \in E(\cM,\tau)_h,& \quad \left\|x_0 - x_i\right\|_E > \varepsilon_i, \\
&\quad  i = 1, \ldots,n\in \mathbb{N},
\end{align*}
there exists  a neighborhood  of the point $x_0$ in $(b_{E(\cM,\tau)})|_{E(\cM,\tau)_h}$ such that 
\begin{align*} \left ( 
E(\cM,\tau)\setminus\bigcup_{i=1}^{n}B_E(x_i,\varepsilon_i)\right)\cap E(\cM,\tau)_h&=E(\cM,\tau)_h\setminus\bigcup_{i=1}^{n}B(x_i,\varepsilon_i)\\
&\subset E(\cM,\tau)_h\setminus \bigcup_{i=1}^{n}B_{E_h}(x_i,\varepsilon_i) ,\\
x_i \in E(\cM,\tau)_h,\quad \|x_0 - x_i\|_E > \varepsilon_i,& \quad i = 1, \ldots,n\in \mathbb{N}.
\end{align*}
Thus, the sequence $\{x_n\}_{n=1}^{\infty}$   has at most   one limit with respect to the topology $b_{E(M,\tau)}$.  This completes the proof.
\end{proof}

%herefore, the sequence $\{x_n\} \subset A$ cannot have more than one limit with respect to the topology $b_{E(\cM,\tau)}${\color{red}ref}.
%Note that  the restriction $(b_{E(\cM,\tau)})|_{E(\cM,\tau)_h}$ is finer than $b_{E(\cM,\tau)_h}$. Indeed,
%Let,
%\begin{align*}
%B_E(x,\varepsilon)=\{y\in E(\cM,\tau):\norm{y-x}\le\varepsilon\},\quad \varepsilon>0
%\end{align*}
%and
%\begin{align*}
%B_{E_h}(x,\varepsilon)=\{y\in E(\cM,\tau)_h:\norm{y-x}\le\varepsilon\},\quad \varepsilon>0
%\end{align*}
%be closed balls in $E(\cM,\tau)$ and $E(\cM,\tau)_h$, respectively. Obviously, $B_{E_h}(x,\varepsilon)\subset B_E(x,\varepsilon)$ for all $x\in E(\cM,\tau)_h$.  For any neighborhood of any $x_0$ in $E(\cM,\tau)_h$
%in $b_{E(\cM,\tau)_h}$:
%\begin{align*}
%E(\cM,\tau)_h\setminus \bigcup_{i=1}^{n}B_{E_h}(x_i,\varepsilon_i)),\quad x_i \in E(\cM,\tau)_h, \quad \|x_0 - x_i\|_E > \varepsilon_i, \quad i = 1, \ldots,n\in \mathbb{N},
%\end{align*}
%there exists  a neighborhood  of the point $x_0$ in $(b_{E(\cM,\tau)})|_{E(\cM,\tau)_h}$
%\begin{align*}
%(E(\cM,\tau)\setminus\bigcup_{i=1}^{n}B_E(x_i,\varepsilon_i))\cap E(\cM,\tau)_h&=E(\cM,\tau)_h\setminus\bigcup_{i=1}^{n}B(x_i,\varepsilon_i))\\
%&\subset E(\cM,\tau)_h\setminus \bigcup_{i=1}^{n}B_{E_h}(x_i,\varepsilon_i)) ,\\
%x_i \in E(\cM,\tau)_h,\quad \|x_0 - x_i\|_E > \varepsilon_i, \quad i = 1, \ldots,n\in \mathbb{N}.
%\end{align*}

Below, we extend  \cite[Proposition 3.4]{AC2} to the setting of stronlgy symmetric spaces affiliated with a semifinite von Neumann algebra.
\begin{proposition}\label{proposition:Comparison of topologies}
Let $\cM$ be a semifinite
von Neumann algebra    equipped with a semifinite faithful normal trace $\tau$.
Let $E(\mathcal{M},\tau)$ be a  noncommutative strongly symmetric space affiliated $\cM$ having the Fatou property. Then,
the topology $\sigma(E,E^\times)$ finer than $b_{E(\cM,\tau)}$.
\end{proposition}
\begin{proof}
It suffices to show that $$B(0,1):=\{x\in E(\cM,\tau):\norm{x}_E\le1\}$$ is
closed with respect to the weak topology  $\sigma(E,E^\times)$. 
%We show that \mbox{$E(\cM,\tau)\setminus B(0,1)$} is open with respect to the weak topology $\sigma(E,E^\times)$. 
Let $x_\alpha\in B(0,1)$ and $x_\alpha \to x\in E(\cM,\tau)$ in the $\sigma(E,E^\times)$-topology. 
%For any $x\in E(\cM,\tau)\setminus B(0,1)$. 
Assume by contradiction that $$x\notin B(0,1).$$ 
Since $E(\cM,\tau)$ has the Fatou property, it follows from   \cite[Theorem 32]{DP2014} that
\begin{align*}\norm{x}_E=\norm{x}_{E^{\times\times}}&~ \quad =\qquad \sup_{\left\|y\right\|_{E^\times}\le 1, y\in E(\cM,\tau)^\times}\tau (|xy|)\\
&\stackrel{\tiny \mbox{\cite[Prop. 22]{DP2014}}}{=} \sup_{\left\|y\right\|_{E^\times}\le 1, y\in E(\cM,\tau)^\times}|\tau (xy)|.
\end{align*}
Then there exists $\varepsilon>0$ and $y_0\in E(\cM,\tau)^\times$ with $\norm{y_0}_{E^\times} \le1$ such that 
\begin{align}\label{xy0>1}
|\tau(xy_0)|>1+\varepsilon.
\end{align} 
%Define $$U(x,\varepsilon):= \left\{ z\in E(\cM,\tau):|\tau((z-x)y)|\le\frac{\varepsilon}{2}\quad \mbox{for all}\quad y\in E(\cM,\tau)^\times \right\}.$$ 
%We  claim
 %that $U(x,\varepsilon)\cap B(0,1)=\emptyset$. Indeed,
On the other hand, $x_\alpha \to x $ in the $\sigma(E,E^\times)$-topology implies that 
$$\tau(x_\alpha y_0)\to_\alpha \tau(xy_0).$$
However, $x_\alpha\in B(0,1) $ implies that $|\tau(xy_0)| \le 1$. 
There exists $\alpha $ such that $ | \tau( (x_\alpha -x )y_0)| <\varepsilon$, and  
$$1+\varepsilon\stackrel{\eqref{xy0>1}}{<}|\tau(xy_0)| \le |\tau(x_\alpha y_0)| + | \tau( (x_\alpha -x )y_0)| <1+\varepsilon, $$
which is a contradiction. 
This completes the proof. 
\end{proof}

%Assume that $x_\lambda \to x \in E(\cM,\tau)$ in the $\sigma(E, E^\times )$-topology.
%%
%%
%%It suffices to show that $B_{E(\cM,\tau)}=\{x\in E(\cM,\tau):\|x\|_{E}\leq1\}$ is closed with respect to the weak topology $\sigma(E,E^{\times})$. Let $\{x_n\}_{n\ge1}\in B_{E(\cM,\tau)}$ be a sequence with $x_n\stackrel{\sigma(E,E^{\times})}{\longrightarrow}x$ as $n\to\infty$,  $x\in E(\cM,\tau)$.
%Assume that $x\notin B_{E(\cM,\tau)}$, i.e. $\|x\|_{E}=q>1$, and let $\varepsilon>0$ such that $q>1+\varepsilon$. By \cite[Theorem 32]{DP2014}, $E(\cM,\tau)=E(\cM,\tau)^{\times\times}$ and $\norm{x}_E=\norm{x}_{E^{\times\times}}$ for all$x\in E(\cM,\tau)$. That is
%
%$$\norm{x}_E=\norm{x}_{E^{\times\times}}=\sup_{\|y\|_{E^\times}\le 1, y\in E(\cM,\tau)^\times}\tau (|xy|)\stackrel{\tiny \mbox{\cite[Proposition 22]{DP2014}}}{=} \sup_{\|y\|_{E^\times}\le 1, y\in E(\cM,\tau)^\times}|\tau (xy)|$$
%Then, there exists $y\in E(\cM,\tau)^{\times}$ such that $\|y\|_{E^{\times}}\leq1$ and $q\geq|\tau(xy)|>q-\varepsilon$.
%
%On the other hand, $x_n\stackrel{\sigma(E,E^{\times})}{\longrightarrow}x$ implies that $|\tau((x_n y))|\rightarrow |\tau(xy)|$. Since $\norm{x_n}_{E}\leq1$, it follows that $|\tau(x_n y)|\leq1$, this implies that $\tau(xy)\leq1$, which is impossible. Thus $B_E(\cM,\tau)$ is a closed set in $\sigma(E,E^{\times})$.

For a   series $\sum_{n=1}^\infty x_n$ in a Banach space $X$ is called   a weakly unconditional series if the numerical series $\sum_{n=1}^\infty f(x_n)$ convereges absolutely for every $f\in X^*$~\cite[Ch. 2, $\S$3]{PW}.

\begin{proposition}\label{Pp6}\cite[Ch. 2, $\S$ 3]{PW}
Let $(X, \norm{\cdot}_X)$ be a Banach space, and let $x_n \in X$, $n \in \mathbb{N}$. Then the following conditions are equivalent:
\begin{enumerate}
    \item the series $\sum_{n=1}^{\infty} x_n$ converges weakly unconditionally;

    \item there exists a constant $C > 0$ such that
    \[
    \sup_{N} \left\| \sum_{n=1}^{N} t_n x_n \right\|_X \leq C \left\| \{t_n\}_{n=1}^{\infty} \right\|_{\infty} \quad \text{for all} \ \{t_n\}_{n=1}^{\infty} \in \ell^{\infty}.
    \]
\end{enumerate}

\end{proposition}

The following result is an immediately consequence   of the above proposition. 

\begin{cor}\label{cp6}\cite[Corollary 4]{AC}
Let $(X, \norm{\cdot}_X)$ and $(Y,\norm{\cdot}_Y)$ be two Banach spaces. If $T:X\to Y$ is a surjective linear isometry, and a series $\sum_{n=1}^{\infty}x_{n}$ converges weakly unconditionally in $X$, then the series $\sum_{n=1}^{\infty}T(x_{n})$ also
converges weakly unconditionally in $Y$.
\end{cor}

We can now demonstrate that   surjective linear isometries on noncommutative symmetric spaces $E(\cM,\tau)$  are  continuous in the $\sigma(E, E^\times)$-topology.
The following result is a semifinite version of \cite[Theorem 4.4]{AC2} and  \cite[Proposition 2.5]{KR2}.
\begin{theorem}\label{weak-continous of T}
Let $\cM_1$ and $\cM_2$ be  atomless $\sigma$-finite von Neumann algebras equipped with semifinite faithful normal traces $\tau_1$ and $\tau_2$, respectively.
Let $E(\cM_1,\tau_1)$ and $F(\cM_2,\tau_2)$ be two  noncommutative strongly  symmetric   spaces having the  Fatou property. Let $T:E(\cM_1,\tau_1)\to F(\cM_2,\tau_2)$ be a surjective linear isometry.  Then $T$ is $\sigma(E,E^{\times})-\sigma(F,F^{\times})$-continuous.
\end{theorem}

\begin{proof}
By Lemma \ref{lemma adjoint operator}, we only need to show that $T^{*}(y)\in E(\cM_1,\tau_1)^\times$ for each $y\in F(\cM_2,\tau_2)^\times$. By Proposition \ref{proposition normal functional}, it suffices to show that 
\begin{align}\label{aim conv}
\tau_2(yT(x_n))= T^*(y) (x_n) \to0
\end{align} 
for any sequence $\{x_n\}_{n=1}^\infty\subset E(\cM_1,\tau_1)^+$ with  $x_n\downarrow 0$. 

For any positive linear functional $f\in E(\cM_1,\tau_1)^*$, we have
\begin{align*}
\sum_{n=1}^m|f(x_n-x_{n+1})|
&=f\left(\sum_{n=1}^m(x_n-x_{n+1})\right)\\
&=f(x_1-x_{m+1})\\
&\le f(x_1)
\end{align*}
for all $m\in \mathbb{N}$. Hence, the numerical series $$\sum_{n=1}^\infty f(x_n-x_{n+1})$$ converges absolutely. Since every functional $f\in E(\cM_1,\tau_1)^*_h$ is the difference of two positive functionals from  $E(\cM_1,\tau_1)^*_h$ (see \cite[Proposition 4.2.2]{DPS}) and  the space $E(\cM_1,\tau_1)^*_h$  may be identified with the Banach dual of the real normed space $E(\cM_1,\tau_1)_h$ (see~\cite[Section 4.2]{DPS}), it follows that the series $$\sum_{n=1}^\infty(x_n-x_{n+1}) $$ converges weakly unconditionally in $E(\cM_1,\tau_1)_h$.
Let $f$ be an arbitrary element in $  E(\cM_1,\tau_1)^*$. Set
$$
u(x)={\rm Re}~f(x)=\frac{f(x)+\overline{f(x)}}{2},~x\in E(\cM_1,\tau_1), $$
and 
$$v(x)={\rm Im}~f(x)=\frac{f(x)-\overline{f(x)}}{2i},~x\in E(\cM_1,\tau_1).
$$
In particular, $u,v\in E(\cM_1,\tau_1)^*_{h}$. 
Therefore, 
the series $$ \sum_{n=1}^\infty u(x_n-x_{n+1}) \mbox{  and }  \sum_{n=1}^\infty v(x_n-x_{n+1}) $$
 converge absolutely. Hence, the series 
$$\sum_{n=1}^\infty f(x_n-x_{n+1})$$
 also converges absolutely. 
That is,  the series $\sum_{n=1}^\infty(x_n-x_{n+1})$
converges weakly unconditionally in $E(\cM_1,\tau_1)$.

By Proposition \ref{cp6}, the series
$$\sum_{n=1}^\infty(Tx_n-Tx_{n+1})$$
 also converges weakly unconditionally in $F(\cM_2,\tau_2)$. Thus, $$\sum_{k=1}^{n}f(Tx_n-Tx_{n+1})=f(Tx_1)-f(Tx_n)$$
converges for every $f\in F(\cM_2,\tau_2)^*$. In particular, the limit $\lim\limits_{n\to\infty}f(Tx_n)$ exists. Therefore,
 the sequence $$\{T(x_n)\}_{n=1}^\infty$$ is a $\sigma(F,F^\times)$-Cauchy sequence. By \cite[Proposition 3.1]{DK},
 we know that $F(\cM_2,\tau_2)$ is $\sigma(F,F^\times)$-sequentially complete and $T$ is a bijection, it follows that there exists $x_0\in E(\cM_1,\tau_1)$ such that 
\begin{align}\label{Txnconvergence}
\mbox{$T(x_{n})\xrightarrow{\sigma(F,F^{\times})}T(x_0)$
}
\end{align}
 as $n\to\infty$.

Next we show that $T(x_0)=0$.
 By Proposition \ref{proposition:Comparison of topologies}, 
we have
$$T(x_{n})\xrightarrow{b_{F(\cM_2,\tau_2)}}T(x_{0})$$
as $n\to \infty$. 
Since   $T^{-1}:F(\cM_2,\tau_2)\to E(\cM_1,\tau_1)$ is also a surjective isometry, it follows that $T^{-1}$ is   $b_{F(\cM_2,\tau_2)}-b_{E(\cM_1,\tau_1)}$-continuous (see Lemma~\ref{ball continuous}), therefore, $$x_{n}\xrightarrow{b_{E(\cM_1,\tau_1)}}x_{0}.$$
 Now, taking into account that $x_n\downarrow0$, we obtain $$x_{n} \xrightarrow{\sigma(E,E^{\times})} 0$$ (see Proposition \ref{proposition normal functional}). 
By Proposition \ref{proposition:Comparison of topologies}, we have 
$$x_n\xrightarrow{b_{E(\cM_1,\tau_1)}}0,$$
 and it follows from Theorem \ref{theorem: unique of element} that $x_0=0$.
By \eqref{Txnconvergence}, we have 
\begin{align*} 
\mbox{$T(x_{n})\xrightarrow{\sigma(F,F^{\times})}0$
}
\end{align*}
 as $n\to\infty$, which proves \eqref{aim conv}. 
\end{proof}

We now prove  the 
\( \sigma(E, E^\times) \)-continuity of 
a hermitian operator on a symmetric space $E(\cM,\tau)$ with the Fatou property, which is a semifinite version of \cite[Theorem~5]{AC}.
\begin{theorem}\label{continous of hermitian}
Let $\cM$ be a $\sigma$-finite von Neumann algebra  equipped with a semifinite faithful normal trace $\tau$. 
Let $E(\mathcal{M},\tau)\subset S(\cM,\tau)$ be a noncommutative strongly symmetric space having the  Fatou property.
Then, any  bounded hermitian  operator $T:E(\cM ,\tau)\to E(\cM ,\tau )$  is necessarily  $\sigma(E,E^{\times})-\sigma(E,E^{\times})$-continuous.

\end{theorem}

\begin{proof}
Consider the non-negative continuous function 
defined by 
$$\alpha(t)=\norm{e^{itT}- I}_{E(\cM,\tau)\to  E(\cM,\tau)},$$ where $t\in \mathbb{R}$. As $\alpha(0)=0$, it follows that there exists $0\ne t_{0}\in \mathbb{R}$ such that $\alpha(t_{0})<1$. Since
the operator $T$ is  hermitian on $E(\cM,\tau)$, it follows from \cite[Theorem 5.2.6]{FJ} that the operator $V=e^{it_{0}T}$ (and hence $V^{-1}=e^{-it_{0}T}$) is an isometry from $E(\cM,\tau)$ to $E(\cM,\tau)$. By Theorem~\ref{weak-continous of T} above, the operator  $S:=V - I$ is $\sigma(E,E^{\times})$-continuous. In addition, we have
$$\norm{S}_{E(\cM,\tau)\to E(\cM,\tau)}=\alpha(t_{0})<1.$$
 Now, since (see e.g. \cite[Lemma 2.13]{Douglas})
\begin{eqnarray*}
it_{0}T=\ln(I+S)   = \sum_{n=1}^{\infty}\frac{(-1)^{n-1}S^{n}}{n}
\end{eqnarray*}
and $\frac{(-1)^{n-1}S^{n}}{n}$ 
 is $\sigma(E,E^{\times})$-continuous for every $n\ge 1$, 
it follows from Proposition~\ref{proposition weak continous}  that $it_{0}T$ is $\sigma(E,E^{\times})$-continuous. Therefore, $T$ is  $\sigma(E,E^{\times})$-continuous.
\end{proof}

\begin{remark}
Even though the results in this chapter are stated for symmetric spaces affiliated with a $\sigma$-finite semifinite von Neumann algebra,
the same arguments apply to  the setting of strongly  symmetric spaces  (having the Fatou property) of $\tau$-compact operators affiliated with a (not necessarily $\sigma$-finite) semifinite  von Neumann algebra.
\end{remark}

\chapter{Elementary form of surjective isometries on noncommutative symmetric spaces}\label{C6}
The goal of this chapter
 is to answer   the question posed in \cite{Sukochev,CMS} in the setting   when the norm on $E(\cM,\tau)$ is not necessarily order continuous. Throughout this chapter, unless stated otherwise, we assume that $\cM$ is either an atomless semifinite von Neumann algebra or an atomic semifinite von Neumann algebra with all atoms having the same trace, and we assume that $\tau$ is a semifinite, faithful, normal trace on $\cM$.

\section{Jordan $*$-isomorphisms} 

Let $(\cM_1,\tau)$ and $(\cM_2,\tau)$ be two semifinite von Neumann algebras equipped with semifinite faithful normal traces.
A  complex-linear  map $J:\cM_1 \longrightarrow
\cM_2$ ($\mbox{respectively,}$ $S(\cM_1,\tau_1) \to S(\cM_2,\tau_2))$ is called a  Jordan $*$-homomorphism if  
$$ J(x^*)=J(x)^* \mbox{ and }  J(x^2)=J(x)^2 , ~ x\in \cM_1~(\mbox{respectively,}\ S(\cM_1,\tau_1)) , $$
 equivalently, 
$$ J(xy+yx)=J(x)J(y)+J(y)J(x) , ~ x,y\in \cM_1  \mbox{ (respectively, $ S(\cM_1,\tau_1)$)}$$ (see e.g. \cite{Yeadon,Sherman,BR}). A Jordan $*$-homomorphism $J$ is called a Jordan $*$-monomorphism, if $J$ is injective.
If $J$  is
a bijective Jordan $*$-homomorphism, then it is called a Jordan $*$-isomorphism.

A   Jordan $*$-homomorphism
 $J:\cM_1\to \cM_2$
 is called normal if it is completely additive (equivalently, ultraweakly continuous).
Alternatively, we adopt the following equivalent definition: $$\mbox{$J(x_\alpha)\uparrow J(x)$ whenever $x_\alpha\uparrow x\in (\cM_1)_+$}$$
 (see e.g \cite[Chapter I.4.3]{Dixmier}).
If $J :\cM_1\to \cM_2$ is a  Jordan $*$-isomorphism, then $J$ is necessarily normal~\cite[Appendix A]{RR}.
A Jordan $*$-homomorphism from $S(\cM_1,\tau_1)$ into $S(\cM_2,\tau_2)$ is said to be normal if 
 $$\mbox{$J(x_\alpha)\uparrow J(x)$ whenever $x_\alpha\uparrow x\in S(\cM_1,\tau_1)_+$.}$$

A characterization of continuity of Jordan $*$-homomorphisms with respect to the
measure topology is given in the following lemma, see also \cite[Lemma 2.9.1]{DPS} for the case of $*$-homomorphisms.

\begin{lemma}\label{cvb}
Let $\cM_1$ and $\cM_2$ be two von Neumann algebras equipped 
with semifinite faithful normal traces $\tau_1$ and 
$\tau_2$, respectively.

For a Jordan $*$-homomorphism $J:S(\cM_1,\tau_1)\to S(\cM_2,\tau_2)$ the following
statements are equivalent:
\begin{enumerate}
       \item $J$ is continuous with respect to the measure topology;
       \item $J|_{\cM_1}$ is continuous with respect to the measure topology;
       \item $\tau_2(J(p_n))\to0$ for every sequence $\{p_n \}$ in $P(\cM_1)$ satisfying $\tau_1(p_n)\to 0$ as $n\to\infty$.
\end{enumerate}
\end{lemma}
\begin{proof}
Evidently, (1) implies (2). The implication  $(2) \Rightarrow  (3)$ follows from  \cite[Corollary 2.5.8]{DPS}. The implication $(3) \Rightarrow (1)$ is an immediate consequence
of \cite[Theorem 3.13]{Weigt} (see also \cite{Lab}).
\end{proof}

For convenience of reference, we state the following auxiliary result, which is important in our discussion below.
%The following lemma is a well-known result. % We include a brief proof for completeness.
\begin{cor} \label{lemma measure-measure}
Let $\cM $ and $\cN$ be  finite von Neumann  algebras  equipped
with   faithful normal tracial states  $\tau_1$ and $\tau_2$, respectively.
 Let $J$ be a normal Jordan $*$-homomorphism from $\cM$ into $\cN$. Then, $J$ is  measure-measure continuous.
\end{cor}

The following is a direct consequence of Lemma \ref{cvb} and the density of $\cM$ in $S(\cM,\tau)$ 
with respect to the measure topology
(see \cite[Proposition 2.5.4]{DPS}).  See also~\cite[Proposition 2.9.2]{DPS} or the proof of \cite[Theorem 4.3]{Weigt}.
\begin{proposition}\label{lga}
Let $\cM_1$ and $\cM_2$ be two von Neumann algebras equipped 
with semifinite faithful normal traces $\tau_1$ and 
$\tau_2$,  respectively. 
If $J:\cM_1\to \cM_2$ is a  Jordan $*$-homomorphism (respectively, 
 Jordan $*$-isomorphism) which is continuous
with respect to the measure topology, then $J$ has a unique extension to a Jordan $*$-homomorphism  (respectively,  Jordan $*$-isomorphism) $\tilde{J}:S(\cM_1,\tau_1)\to S(\cM_2,\tau_2)$, which is continuous with respect to the measure topology.
\end{proposition}

%\begin{proof}
%Since $\cM_1$ is dense $S(\cM_1,\tau_1)$ (see \cite[Proposition 2.5.4]{DPS}) and $S(\cM_2,\tau_2)$ is complete
%with respect to the measure topology (by \cite[Theorem 2.5.12]{DPS}), it is clear that $J$ has a
%unique continuous linear extension $\tilde{J}:S(\cM_1,\tau_1)\to S(\cM_2,\tau_2)$. Since $S(\cM_1,\tau_1)$ and %$S(\cM_2,\tau_2)$. are
%topological $*$-algebras, $\tilde{J}$ is a $*$-homomorphism.  By Lemma \ref{cvb} (equivalence of
%(1) and (2)), any $*$-homomorphism from $S(\cM_1,\tau_1)$ into $S(\cM_2,\tau_2)$ extending $J$, is continuous
%with respect to the measure topology and so, the uniqueness of $\tilde{J}$ is also clear. Furthermore, it should be noted that $J$ is injective (that is, $J$ is a $*$-isomorphism) if and only if $\tilde{J}$ is injective. Indeed, if $0\ne x\in S(\cM_1,\tau_1)$, then there exists $a\in \cM_1^{+}$ such that $0<a\le |x|$ and so, $$0<J(a)=\tilde{J}(a)\le \tilde{J}(|x|)=|\tilde{J}(x)|,$$ which show that $\tilde{J}(x)\ne 0$.
%\end{proof}

%\section{Isometries on symmetry operator spaces}
 
\section{Proof of  Theorem \ref{11212} }\label{proof14}

Before proceeding to the proof of Theorem \ref{th:iso}, we need  the following auxiliary tool, which follows from the same argument as that in  \cite[Corollary 4.1]{HS} (with \cite[Theorem 3.10]{HS} replaced by Theorem \ref{th:her} below).
Therefore, we omit the proof. 
 \begin{cor}\label{4.2}Let $(\cM,\tau )$ be an atomless $\sigma$-finite von Neumann algebra or an atomic $\sigma$-finite von Neumann algebra whose atoms having the same trace.
 Let $E(\cM,\tau)$ be a  noncommutative  symmetric  space  whose norm is not proportional to $\norm{\cdot}_2$ with the Fatou property.
Let $T$ be a bounded hermitian operator on $E(\cM,\tau)$. % such that $T^2$ is also a hermitian operator on $E(\cM,\tau)$.
Then, $T^2$ is also a hermitian operator on $E(\cM,\tau)$ if and only if $$T(y)=a   y+y   b  ,~\forall y\in L_1(\cM,\tau )\cap \cM ,$$
for some   self-adjoint operators   $a \in \cM_w $
 and $b\in \cM_{{\bf 1}-w}$, where $w \in P(Z(\cM))$.

\end{cor}
%\begin{proof}
%$(\Leftarrow)$.
% Note that $T^2(y)= a^2   y+y   b^2  $, ~$\forall y\in L_1(\cM,\tau)\cap \cM$.
%It follows from Theorem \ref{th:her} that $T^2$ is a hermitian operator.
%
%
%$(\Rightarrow)$. Recall that, by Theorem \ref{th:her}, we have $T(y)=ay+yb$, $y\in \cM$, for some self-adjoint elements $a,b\in \cM$.
% Due to the assumption that $T^2$ is also  hermitian, there exist self-adjoint operators $c,d\in \cM$ such that
% \begin{align*}
% T^2(y)= cy+yd   = a^2y + 2 ayb +  yb^2, ~\forall y\in \cM.
% \end{align*}
% The claim follows from Theorem \cite[Theorem A.6 in Appendix A]{HS}.
%\end{proof}
\begin{rem}\label{rem42}\cite[Remark 4.2.]{HS}
Let $T$, $a$, $b$, $w$ be defined as in Corollary \ref{4.2}.
In particular,   $l(a)\le w $ and $l(b)\le {\bf 1} - w $.
Define
$$z_a: = \sup \left\{ p\in \cP(Z(\cM)): p\le w, ~  pa  \in Z(\cM) \right  \}$$
and
$$z_b: = \sup \left\{ p\in \cP(Z(\cM)): p\le {\bf 1}- w , ~  pb  \in Z(\cM) \right  \}.$$
%Let $a$ and $b$ be defined as in Corollary \ref{4.2}.
%Define
%$$z_c: = \sup \left\{ z\in \cP(Z(\cM)):  za, zb \in Z(\cM) \right  \}.$$
%For any elements $z_1,z_2 \in \left\{ p\in \cP(Z(\cM)): p\le w , ~  pa  \in Z(\cM) \right  \} $ and $d \in \cM$, we have
%\begin{align*}
%(z_1\vee z_2)a d&= (z_1 +z_2 -z_1z_2)a d = z_1 ad + z_2 ({\bf 1}-z_1)ad \\
%&= z_1 a  d + z_2 a ({\bf 1}-z_1)d = d z_1a +  d ({\bf 1}-z_1)  z_2 a   =d  (z_1\vee z_2) a.
%\end{align*}
%That is, $z_1\vee z_2 \in \left\{ p\in \cP(Z(\cM)): p\le w , ~  pa  \in Z(\cM) \right  \} $.
%Hence, $\left\{ p\in \cP(Z(\cM)): p\le w , ~  pa  \in Z(\cM) \right  \} $ is an increasing  net with  the partial order $\le $ of projections.
%Therefore,   by Vigier's theorem \cite[Theorem 2.1.1]{LSZ}, we obtain that
%$z_a a \in Z(\cM)$.
%That is,
% $$z_a \in \left\{ p\in \cP(Z(\cM)): p\le w, ~  pa  \in Z(\cM) \right  \} .$$
% Arguing similarly,
%  $$z_b  \in \left\{ p\in \cP(Z(\cM)): p\le {\bf 1}-w , ~  pb  \in Z(\cM) \right  \} .$$
 For any $y\in L_1(\cM,\tau )\cap \cM$, we have
\begin{align}%\label{remarkequ}
T(y)=a (w-z_a)  y+y  b(({\bf 1}-w) -z_b) + (az_a +bz_b)  y  .
\end{align}
In particular,
\begin{enumerate}
  \item $w-z_a$, $({\bf 1}-w) -z_b$, $z_a+z_b$ are
pairwise orthogonal projections in $Z(\cM)$;
  \item if $p\in Z(\cM)$ such that $p\le w-z_a$ and  $ap\in Z(\cM,\tau)$ (or $p\le ({\bf 1}-w) -z_b$ and $bp\in Z(\cM)$), then  $p=0$;
  \item   $w-z_a$ (resp., $ ({\bf 1}-w) -z_b $) is the central support of $a (w-z_a)$ (resp.,   $   b(({\bf 1}-w ) -z_b)$).
\end{enumerate}
\end{rem}

\begin{rem}
Let $\cM$ be a     semifinite factor.
It is an immediate consequence of Corollary \ref{4.2} that
if  $T$  a bounded hermitian operator on  $\cM$,
then  $T^2$ is hermitian if and only if $T$ is a left(or right)-multiplication by a self-adjoint operator in $\cM$
(see also the proof of \cite[Corollary 2]{Sourour} and \cite[Corollary 3.2]{Sukochev}).
\end{rem}

The following theorem is the main result of this chapter.
Even though the main  idea of the following proof is same as  that in \cite[Theorem 4.4]{HS}, there are technical differences between them.

\begin{theorem}\label{th:iso}

Let $\cM_1$ and $\cM_2$ be  atomless $\sigma$-finite von Neumann algebras (or $\sigma$-finite atomic von Neumann algebras whose atoms have the same trace) equipped with semifinite faithful normal traces $\tau_1$ and $\tau_2$, respectively.
Let $E(\cM_1,\tau_1)$ and $F(\cM_2,\tau_2)$ be two strongly symmetric operator spaces having the  Fatou property, whose norms  are not proportional to  $\norm{\cdot}_2$.

If $T:E(\cM_1,\tau_1)\to F(\cM_2,\tau_2)$ is a surjective isometry, then there exist
two sequences of elements $ A_i \in  F(\cM_2 ,\tau_2)  $  disjointly supported from the left, 
 and $ B_i \in F(\cM_2 ,\tau_2) $    disjointly supported from the  right,   and a  surjective  Jordan $*$-isomorphism $J:\cM_1\to \cM_2$  and a central projection $z\in \cM_2$ such that
$$T(x) =  \sigma(F,F^\times)-  \sum_{i=1}^\infty     J(x)A_i z + B_i J(x) ({\bf 1}-z)  , ~x\in E(\cM_1,\tau_1)\cap \cM_1.$$

\end{theorem}
\begin{proof}
%It suffices to prove that  there exist unitary operators $u,v\in B(\cH)$ such that  $T(x) = u xv$ or $T(x) = u x^{t}v$
%for every $x\in C_E$.
%Indeed, if the claim is proved, then   $\norm{x}_E=\norm{T(x)}_F =\norm{uxv}_F=\norm{x}_F$ or $\norm{x}_E=\norm{T(x)}_F =\norm{ux^tv}_F=\norm{x}_F$ .
%This implies that $E=F$.

The indices of von Neumann algebras  $\cM_1$ and $\cM_2$ play no role in the proof below. To reduce  notations,   we may assume that $$\cM_1=\cM_2=\cM.$$
We denote by $L_a$ (resp. $R_a$) the left (resp. right) multiplication by $a\in \cM$,
that is, $$L_a(x)=ax$$ (resp. $R_a(x)=xa$) for all $x\in  S(\cM,\tau)$, which is a hermitian operator on $E(\cM,\tau)$ (see Theorem \ref{th:her}). 
For any self-adjoint operator $a\in \cM $,  
 operators 
$TL_aT^{-1}$ and $TL_a^2T^{-1}$ are hermitian on $ F(\cM,\tau)$ (see e.g. \cite[Lemma 2.3]{FJ89} or \cite{JL}).
% Corollary  \ref{4.2}, for any self-adjoint $a\in \cM $,
%there exists a central projection $z_a\in \cM$ such that
%$$Tl_aT^{-1}= l_{J_1(a)z_a}+r_{J_2(a)({\bf 1} -z_a)}.$$  % is either a left multiplication or a right multiplication by a self-adjoint operator.

%We divide the proof into several steps.

%\textbf{Step 1.}
The same argument as that in \cite[Theorem 4.4]{HS} shows that   there exists  $z\in \cP(Z(\cM)) $ (does not depend on $a$ below) such that
\begin{align}\label{tlt}
TL_aT^{-1}= L_{J_1(a)z }+R_{J_2(a)({\bf 1} -z )},~a=a^*\in \cM,
\end{align}
 where $J_1(a),J_2(a)$ are self-adjoint operators  in $\cM$.
Note that $$L_{J_1(a)^2z } +R_{J_2(a)^2 ({\bf 1} -z)} \stackrel{ \eqref{tlt} } {=} (TL_aT^{-1})^2= TL_{a^2}T^{-1} \stackrel{\eqref{tlt}}{= } L_{J_1(a^2)z } +R_{J_2(a^2)({\bf 1} -z)}   $$ for every $a=a^*\in \cM$.
 By standard    argument   (see e.g. \cite[p.117]{Sukochev}),  we obtain
 \begin{align}\label{defJ12}J(\cdot):=J_1 (\cdot) z  +J_2(\cdot )({\bf 1}- z)
 \end{align}
 is a  Jordan $*$-monomorphism on $\cM$.
Let $0\le a_i\uparrow a\in \cM$.
We have $$J(a_i)z \uparrow \le J(a)z $$ (see e.g. \cite[Eq.(12)]{HSZ} and \cite[p.211]{BR}).
Since $a_i\uparrow a$, it follows that for any $x\in E(\cM,\tau)$,
we have $$a_i x \to_{\sigma(E,E^\times)} a x.$$
By Theorem \ref{weak-continous of T}, we have
$$
T(a_i x )\to_{\sigma(F,F^\times)} T(a x) .
$$
In other words,
\begin{align}\label{Taiax}
T((a_i-a)x)\to_{\sigma(F,F^\times)} 0 .
\end{align}
Consequently,
for all $x\in E(\cM,\tau)$ such that $T(x) z =T(x)$ and  $y\in F(\cM,\tau)^\times$, 
   we have
$$\tau\Big(J(a-a_i)\cdot z \cdot  T(x)  \cdot  y \Big) \stackrel{ \eqref{tlt} } {=}\tau\Big(T\big( \left(a-a_i\right)x\big)  y \Big)\to0 .$$ 
Since $T$ is surjective, it follows that 
for any $\tau$-finite projection $p\in \cM$ with $p\le z $,  we have  $$\tau\left( J(a-a_i)zp \right)\stackrel{\eqref{Taiax} }{\longrightarrow}
0.$$
Therefore,
 $J(a-a_i)z\to0$  in the  local measure topology~\cite[Proposition 20]{DP2014}.
Hence,
$J(a_i)z\uparrow J(a)z$ (see Section~\ref{s:p1}).
The same argument shows that $J(a_i)({\bf 1}-z)\uparrow J(a)z({\bf 1}-z) $.
Therefore, $J(\cM)$ is weakly closed~\cite[Remark 2.16]{HSZ} and $J:\cM\to J(\cM)$ is a surjective (normal) Jordan $*$-isomorphism.

%\textbf{Step 3.}
The same argument used in \cite[Theorem 4.4]{HS} shows that   $J$ is surjective.
Applying \eqref{tlt} to $T(x)$,
 we obtain that
\begin{align*}
T(ax)& ~=~ J_1(a)z T(x) + T(x)({\bf 1} - z) J_2(a )\\
&~= ~J_1(a)z T(x)  + T(x) J_2(a)({\bf 1} - z) \\
&\stackrel{\eqref{defJ12}}{=}J  (a)z  T(x) +  T(x) J (a)({\bf 1} - z)
 \end{align*}
for all $ a=a^*\in \cM ,~x\in E(\cM,\tau)\cap \cM.$
%It shows that $J$ is surjective. Indeed,
%let $M_x$ be an arbitrary factor which is a weakly closed $*$-subalgebra in $\cM$.
%Since $J$ is normal, it follows that $J(M_x)$ is a factor as well.
%Hence, $J(\cM)= \int^\oplus_x J(M_x)d\mu(x)$.
%By the surjectivity of $T$, we obtain that   for almost all %$x'$, there exists $x$ such that $J(M_x)=M_{x'}$
Hence, for any $\tau$-finite projection $e $ in $\cM $, we have
\begin{align}\label{Txfinite}
T(xe )=J(x)T(e) z + T(e)J(x)({\bf 1}-z), ~x\in E(\cM,\tau)\cap \cM . 
\end{align}
Let  $\{e_i \}_{i=1}^{\infty}$ be a sequence of pairwise orthogonal $\tau$-finite projections in $\cM $ such that $\sup_{i\ge1}  e_i ={\bf 1}$~\cite[Corollary 8]{DP2014}.
Note that for any $x\in E(\cM,\tau)$, we have 
$$\sigma(E,E^\times)-\lim\limits_{n\to \infty} \sum_{i=1}^{n} xe_i =x.$$
 Since $T$ is a surjective   isometry, it follows from Theorem \ref{weak-continous of T} that $T$ is $\sigma(E,E^{\times})-\sigma(F,F^{\times})$-continuous. We obain that
\begin{align*}
T(x) &=\sigma(F,F^\times)-\lim\limits_{n\to \infty}  \sum _{i=1}^{n} T(x e_i)\\
& =\sigma(F,F^\times)-\lim\limits_{n\to \infty}  \sum _{i=1}^{n}  J(x)T(e_i) z + T(e_i)J(x)({\bf 1}-z)\\
& =\sigma(F,F^\times)-  \sum _{i=1}^{\infty}  J(x)T(e_i) z + T(e_i)J(x)({\bf 1}-z) ,
\end{align*}
for every $ x\in E(\cM,\tau)\cap \cM.$

Note that
$$T(e_i) \stackrel{\eqref{Txfinite}}{=}  J(e_i )T(e_i) z + T(e_i)J(e_i)({\bf 1}-z) .$$
Recall that Jordan $*$-isomorphisms preserve the disjointness for projections (see e.g. \cite[Proposition 2.14]{HSZ}).
Letting $$ A_i:= T(e_i) z= J(e_i )T(e_i) z $$
and $$B_i :=T(e_i) ({\bf 1}-z) = T(e_i)J(e_i)({\bf 1}-z),
$$ we complete the proof.
\end{proof}

\chapter{A Noncommutative Kalton--Randrianantoanina--Zaidenberg Theorem}\label{S:KRZ}

This chapter  aims to establish a noncommutative   Kalton--Randrianantoanina--Zaidenberg Theorem.
 Although
it was shown in \cite{HS}  that all isometries on a separable  noncommutative  symmetric  space   have elementary form, it is far away from being a sufficient condition.
This chapter  provides a necessary and sufficient condition for an operator on a noncommutative  symmetric   space to be an isometry.
With the results established in Chapter \ref{C6} at hand, we can also treat the case for non-separable symmetric spaces.

Recall the following theorem due to Kalton, Randrianantoanina and Zaidenberg.

\begin{theorem}\label{BR pro3}\cite{KR2}\cite[Proposition 3]{R}\cite{Z77}
Let $X,Y$ be two complex symmetric  function spaces on $(0,1)$, and let $T:X\to Y$ be a surjective isometry.
If $X,Y$ are not equal to $L_p$ up to renorming, then
$$(Tf)(t)=a(t)f(\sigma(t)), ~\forall ~ f\in X  ,  $$
where $a$ is a measurable function with $|a|=1$ and $\sigma$ is a measure-preserving invertible Borel map on $(0,1)$.
\end{theorem}
%According to \cite{KR}, Zaidenberg proved the complex case\cite{Z1}. I checked Zaidenberg but I didn't see where he proved it (I mean, when $X,Y\ne L_p$, $|a|=1$ and $\sigma$ preserves the measure).
We note that Lamperti's result~\cite[Section 4]{Lamperti} 
(see also \cite[p.639 Example]{Z77})
demonstrates that the condition of  the above theorem   can not be relaxed to setting when $\norm{\cdot}_E\ne \lambda\norm{\cdot}_p$ for all $\lambda >0$.

The following theorem shows that if a
 surjective isometry on a symmetric space (over  a finite von Neumann algebra equipped with a faithful normal tracial state) is of  elementary form,
then it is an $L_p$-isometry or is generated by a unitary operator and a trace-preserving Jordan $*$-isomorphism.
\begin{theorem}\label{KRZ}
Let $E(0,1  ), F(0,1)$ be   symmetric function spaces.
Let $\cM_1$ and $ \cM_2 $ be   atomless finite von Neumann algebras  equipped with   faithful normal tracial states  $\tau_1$ and $\tau_2$,  respectively. 
Let  $E(\cM_1,\tau_1)$ and $F(\cM_2,\tau_2)$ be the noncommutative  symmetric spaces corresponding to $E(0,1 )$ and $F(0,1)$,
 respectively.

Let $T:E(\cM_1,\tau_1)\rightarrow F(\cM_2,\tau_2)$ be a surjective isometry. Assume that $T$ is of the form: $$T(x)= J'(x) Az +B J'(x) ({\bf 1}-z),~\forall x\in E(\cM_1,\tau_1) $$
where $A,B\in F(\cM_2,\tau_2)$, $z\in P(Z(\cM_2))$  and  $J':S(\cM_1,\tau_1)\rightarrow S(\cM_2,\tau_2)$ is a   Jordan $*$-isomorphism.
Then,  \begin{enumerate}
        \item If $E(0,1)=L_p(0,1)$ up to an  equivalent norm for some $1\le p\le\infty$, then $T$ is a surjective isometry  from $L_p(\cM_1,\tau_1)$ onto $L_p(\cM_2,\tau_2)$;
        \item If $\norm{\cdot}_E$ is not equivalent to $\norm{\cdot}_p$, then $J'$  is a trace-preserving Jordan $*$-isomorphism and 
there exists a unitary operator $u\in \cM_2$ satisfying  $A+B=u$ such that  
$$T(x)=u\cdot J(x), ~\forall x\in E(\cM_1,\tau_1),$$
where   $J:S(\cM_1,\tau_1)\rightarrow S(\cM_2,\tau_2)$ is a  trace-preserving Jordan $*$-isomorphism such that 
$J (x)=u^*  J' (x)z u  +  J'(x)({\bf 1}-z)$, $x\in S(\cM_1,\tau_1)$.
       \end{enumerate}

\end{theorem}

\section{Proof of Theorem \ref{KRZ} for the  case when $E(0,1)\ne L_p(0,1)$}\label{subs:KRZ}
In this part, we prove the second case of   Theorem \ref{KRZ}.

\begin{proof}
Let $A^*=u_{A^* }|A^*|$ and $B=u_{B }|B |$ be polar decompositions.
Since $T$ is surjective, it follows that $ u_{A^* }  (u_{A^* })^* =z$ and $u_{B }(u_{B })^* ={\bf 1}-z$.
We have $$A=|A^*|u_A  ,~ B^* =|B |(u_{B })^*.$$

Since $u_{A^*}+{\bf 1}-z$ and $z+u_{B}^*$ are unitary operators in $\cM_2$, it follows that  the mapping
$$T'(\cdot):=(u_{B}^*  +z)T(\cdot )(u_{A^*}+{\bf 1}-z  ), ~
\forall x\in E(\cM_1,\tau_2), $$
 is a surjective isometry from $\left(  E(\cM_1,\tau_1), \norm{\cdot}_E \right)$ onto $\left(  F(\cM_2,\tau_2), \norm{\cdot}_F \right).$
Note that $\forall$ $x\in E(\cM_1,\tau_1)$, we have
 \begin{align*}
T'(x)&=(z+u_{B}^*   )\Big(J'(x)Az+B J'(x) ({\bf 1}-z)\Big)(u_{A^*}+{\bf 1}-z  )  \\
&= J'(x) A u_{A^*}z+u_B^* B J'(x) ({\bf 1}-z) \\
&=J'(x)|A^*|u_A u_A^*  z+ u_B^* u_B |B| J'(x) ({\bf 1}-z)\\
&= J'(x)|A^*| z+ |B|J'(x)({\bf 1}-z).
\end{align*}
Without loss of generality, we may assume that $T$ is of the form
\begin{align}\label{formula T positive}
T(x)=J'(x)Az +BJ'(x)({\bf 1}-z),~\forall x\in E(\cM_1,\tau_1), \end{align}
where $A,B\in F(\cM_2,\tau_2)$ are positive, $z\in P(Z(\cM_2))$  and $J'$ is a surjective Jordan $*$-isomorphism from $S(\cM_1,\tau_1)$ onto $S(\cM_2,\tau_2)$.
For simplicity, we may assume that
$$Az=A,~Bz=B. $$

 Let $\cN_2$ be an atomless abelian von Neumann subalgebra of $\cM_2$ containing all spectral projections of $A$ and $B$,  such that   there exist a  trace-measure preserving $*$-isomorphism $I_1 $ from  $\cN_2$ onto  $L_\infty(0,1)$ equipped with the Lebesgue measure \cite[Lemma 1.3]{CKS}.
Note that $(J')^{-1}$ is also a Jordan $*$-isomorphism from $S(\cM_2,\tau_2)$ onto $S(\cM_1,\tau_1)$.  %\cite[Appendix A]{RR}.
Denote  $$\cN_1:=(J')^{-1}(\cN_2),$$
which is also an abelian von Neumann subalgebra of $\cM_1$ (see e.g. \cite[Lemma~4.2]{Weigt} and \cite[Proposition 2.14]{HSZ}), \cite[Proposition 2.9.2(i)]{DPS}).
In particular,
by the measure-continuity of $J'$ (see Corollary \ref{lemma measure-measure} above), we obtain that 
 $J'|_{S(\cN_1)}$ is  a surjective $*$-isomorphism from $S(\cN_1,\tau_1)$  onto $S(\cN_2,\tau_2)$.

We claim that $T$ is a surjective isometry from $E(\cN_1,\tau_1)$ onto $F(\cN_2,\tau_2)$.
For any element $y\in F(\cN_2,\tau_2)$,
we have
$$
 y(A+B)^{-1}z
,\quad  y (A+B)^{-1}({\bf 1}-z)\in S(\cN_2,\tau_2).$$
Since $J'$ is a surjective $*$-isomorphism from $S(\cN_1,\tau)$  onto $S(\cN_2,\tau)$, it follows that
there exists $x\in S(\cN_1,\tau_1) $
such that
\begin{align}\label{pab}
S(\cN_2,\tau_2) \ni y (A+B)^{-1} =J'(x) .
\end{align}
We only need to observe that $x\in E(\cM_1,\tau_1)$. Indeed,
since $T$ is surjective and  $y \in F(\cM_2,\tau_2)$, it follows that there exists $x'\in E(\cM_1,\tau_1)$ such that
\begin{align}\label{Tx'form}
J'(x') Az+B J'(x') ({\bf 1}-z)
\stackrel{\eqref{formula T positive}}{=}
T(x')
=y  \stackrel{\eqref{pab}}{=} J'(x)A+B J'(x). 
\end{align}
 Note that $$A(A+B)^{-1}z=Az(A+B)^{-1}=(A+B)z(A+B)^{-1}=z.$$ 
Multiplying by
$(A+B)^{-1}z$ 
 on both sides of the \eqref{Tx'form}, we obtain that
$$z J'(x')=zJ'(x) .$$ Similarly, we obtain that
$$({\bf 1}-z)  J'(x')=({\bf 1}-z)J'(x) .$$
Hence, $J'(x')=J'(x)$, i.e., $x=x'
$.
This implies that    $$
x\in E(\cM_1,\tau_1)\cap S(\cN_1,\tau_1)=  E(\cN_1,\tau_1).$$
In other words, $T$ is a surjective isometry from $E(\cN_1,\tau_1)$ onto $F(\cN_2,\tau_2)$.

Since $0\le A ,B\in F(\cN_2,\tau_2) $, it follows from \eqref{formula T positive} that for any  $
  x\in E(\cN_1,\tau_1)$, we have
\begin{align*}
T(x) = J'(x) A z + B  J'(x)  ({\bf 1}-z)=A^{1/2}J'(x)A^{1/2} z +B^{1/2}J'(x)B^{1/2}({\bf 1}-z), \end{align*}
which implies  that $T$ is a positive isometry from $E(\cN_1,\tau_1) $ onto $F(\cN_2,\tau_2 )$.

Recall that   $\cN_2$ is $*$-isomorphism to $L_\infty (0,1)$.
Since $\cN_1$ and $\cN_2$ are also $*$-isomorphic, it follows that there exists a
 trace-measure preserving $*$-isomorphism  $I_2$ from $ S( \cN_1,\tau_1)$ onto   $S(0,1)$ equipped with the Lebesgue measure $m$, see e.g. \cite[Theorem 9.3.4]{Bogachev}.
We obtain that
$$I_1\circ T \circ I_2^{-1}$$
is a positive  surjective isometry on $E(0,1)$.
Since $E(0,1)\ne L_p(0,1)$ up to renorming, it follows from Theorem \ref{BR pro3}  that
 $$I_1\circ T \circ I_2^{-1} (x)(t)= x(\sigma(t)),$$
 where $\sigma$ is a measure-preserving Borel map from $(0,1)$ onto $(0,1)$.

Taking $x=\chi_{(0,1)}$, we obtain that
 $$I_1\circ T ({\bf 1}) = \chi_{(0,1)}.$$
 Hence, we have
 \begin{align}\label{AB1}
 A+B=  T ({\bf 1}) = {\bf 1}.
\end{align}
Therefore, $A $ and $B$ are partial isometries  (denote by $u_1$ and $u_2$ respectively)   satisfying $u_1u_1^*=z,u_2^*u_2={\bf 1}-z$. We have
$$T(x )= J'(x ) u_1 +u_2 J'(x) ,~\forall x\in E(\cM_1,\tau_1 ).$$

Defining $J (x)=u_1^*  J' (x)z u_1  +  J'(x)({\bf 1}-z)$, $x\in S(\cM_1,\tau_1)$, we obtain that
$$T(x )=  J'(x ) u_1 + u_2   J'(x )  =(u_1+u_2)J(x  )\stackrel{\eqref{AB1}}{=} J(x ),~\forall x\in E(\cM_1,\tau_1 ). $$

It remains to show that $J$ preserves traces.
Observe that   $J$ is also  a surjective isometry from $E(\cM_1,\tau_1)$ onto $F(\cM_2,\tau_2)$. %{\color{red} Why need to prove $J$ is surjective?}
It suffices to show that for any $0\le x\in E(\cM_1,\tau_1)$, we have $$\tau_2(J(x))=\tau_1(x).$$
By \cite[Lemma 1.3]{CKS}, there exists an abelian atomless von Neumann subalgebra $\cA$ of $\cM_1 $ containing $x$, and a trace-measure $*$-isomorphism $L_1$ from $S(\cA,\tau_1)$ onto $S (0,1)$.
On the other hand,
since
  $\cA$ is   $*$-isomorphic  to $S (0,1)$ and $S(\cA,\tau_1)$,   it follows that there exists a
 trace-measure preserving $*$-isomorphism $L_2$  from $S( J(\cA ) ,\tau_2 )$ onto   $S(0,1)$ with the Lebesgue measure, see e.g. \cite[Theorem 9.3.4]{Bogachev}.
Since
  $J$ is an isometry from $E(\cA,\tau_1)$ onto $F(J(\cA),\tau_2)$,
it follows that
$$L_1\circ J\circ L_2^{-1}$$
is a positive surjective isometry from $E(0,1)$ to $F(0,1)$.
Since $E(0,1)\ne L_p(0,1)$ up to renorming,
it follows from  Theorem \ref{BR pro3} that
\begin{align}\label{toKR}
L_1 \circ J\circ L_2^{-1} (y) (t) = y(\sigma (t)), ~\forall y\in E(0,1),
\end{align}
   where $\sigma$ is a measure-preserving  invertiable Borel map from $(0,1)$ onto $(0,1)$.
 Letting $y:= L_2(x)$,   We have
\begin{align*}\tau_1(x)&= \tau_1(L_2^{-1}(y))=\int_0^1 y(t)dt\\
&=\int_0^1 y(\sigma(t)  )dt  \stackrel{\eqref{toKR}}{=} \int_0^1  L_1\circ  J\circ L_2^{-1}(y)(t)dt  \\
&=     \tau_2( J\circ L_2^{-1}(y) ) = \tau_2(J(x))  . \end{align*}
This completes the proof.
\end{proof}

\section{A noncommutative version of Kalton's theorem}
Before proceeding to the proof of   Theorem \ref{KRZ} (1), we should establish a noncommutative version of a result due to Kalton \cite{Kalton} (see also \cite[Proposition 7.1]{KR2}),
who showed that  the norm of an
 operator on a  symmetric function space $E(0,1)$ can be estimated in terms of the  Boyd indices of $E(0,1)$. 
  The main tool is the noncommutative Radon--Nikodym Theorem established by Segal~\cite{Se}.

Let $(\mathcal{M},\tau_1)$ and $(\cN,\tau_2)$ be noncommutative probability spaces (i.e., atomless von Neumann algebras equipped faithful normal tracial states).
Let $a\in S(\cN,\tau_2)$ and let $J:\mathcal{M}\to\mathcal{N}$ be a Jordan $\ast$-isomorphism (bijective, but not necessarily trace-preserving). Define the mapping $T:S(\mathcal{M},\tau_1)\to S(\mathcal{N},\tau_2)$ by setting  \begin{align}\label{Tform}
T(x)=a\cdot J(x) +J(x)\cdot b, ~\forall x\in S(\mathcal{M},\tau_1),
\end{align}
where $a,b\in S(\cN,\tau_2)$ with mutually disjoint central supports.
To simplify notations, we only consider the case when   $b=0$, i.e.,
\begin{align}
T=T_{a,J}(x):=a\cdot J(x),~ \forall x\in S(\mathcal{M},\tau_1).
\end{align}

By Corollary \ref{lemma measure-measure}, we obtain that $J$ is continuous in the measure topology.
Note  that $J$ can extend to a Jordan $*$-isomorphism from $S(\cM,\tau_1)$ onto $S(\cN,\tau_2)$ (see e.g. Proposition \ref{lga}).
For any Borel function $f$ which are bounded on compact subsets of $\mathbb{R}$ and any $a\in S_h(\cM,\tau_1)$, we have~\cite[Proposition 2.9.2]{DPS}
\begin{align}\label{functional calculus J}
f(J(a)) = J(f(a)).
\end{align}

Let $E(0,1)$ be a symmetric
function space on $(0,1).$
Assume that
$T_{a,J}:E(\mathcal{M},\tau_1)\to F(\mathcal{N},\tau_2)$ is a bounded operator, i.e.,  $\left\|T_{a,J}\right\|_{E(\cM,\tau_1)\to F(\cN,\tau_2)}<\infty$ and
$$\norm{T_{a,J}x}_{F(\mathcal{N},\tau_2)}\leq \norm{T_{a,J}}_{E(\cM,\tau_1)\to F(\cN,\tau_2)} \norm{x}_{E(\mathcal{M},\tau_1)},\quad x\in E(\mathcal{M},\tau_1).$$
In particular, $a =T_{a,J}( {\bf 1 })\in F(\mathcal{N},\tau_2)$.

\begin{lem} \label{segalRN}
Let $(\mathcal{M},\tau_1)$ and $(\cN,\tau_2)$ be two noncommutative probability spaces. Let $J:\cM\to\cN$ be a  (not necessarily trace-preserving) Jordan  $\ast$-isomorphism. There exists $0\leq b\in L_1(Z(\mathcal{M}),\tau_1)$ such that
$$\tau_2(J(x))=\tau_1(bx),\quad \forall  x\in\mathcal{M}.$$
\end{lem}
\begin{proof}
Note that $\tau(J(\cdot))$ is a faithful normal tracial state on $\cM$.
The statement follows from \cite{Se} or \cite[Corollary 4.51]{Hiai}. %by the result in Hiai, we have $\tau_1(ax)=\tau_2(x)$, where $a\in L_1(\cM)$. We only need to prove that $a$ is in the center. Indeed, $\tau_1(yax)=\tau_1(axy)=\tau_2(xy)=\tau_2(yx)=\tau_1(ayx)$, i.e., $\tau_1((ya-ay)x)=0$ for any $x,y$, which is impossible if $[a,y]\ne 0$.
\end{proof} 

\begin{lem} \label{yuyuo}
%Assume that $p_X\le r \le q_X$.
Let $(\mathcal{M},\tau_1)$ and $(\cN,\tau_2)$ be two noncommutative  probability spaces.  Let $J:\cM\to\cN$ be a  (not necessarily trace-preserving) Jordan  $\ast$-isomorphism. There exists a positive element $b\in L_1(Z(\cM),\tau_1)$  such that
for any $0<r<\infty$,
we have
$$\norm{T_{a,J}(x)}_{L_r(\mathcal{N},\tau_2)}\leq \norm{b^{\frac1r}\cdot J^{-1}(a)}_{\infty}\norm{x}_{L_r(\mathcal{M},\tau_1)},\quad\forall  x\in E( \mathcal{M},\tau_1).$$
\end{lem}
\begin{proof}
%Note that $L_r(\cM,\tau)\subset X(\cM,\tau)$.
By Stormer's theorem~\cite[Theorem 3.3]{Stormer} (see also \cite[Proposition 3.1]{BHS} or \cite[Theorem 4.3]{Weigt}), there exists an element $z\in P(Z(\cM))$ such that $J(z\cdot )$ is a $*$-homomorphism and $J(({\bf 1}-z)\cdot )$ is a $*$-anti-homomorphism.

By Lemma \ref{segalRN}, there exists a positive element $b\in L_1(Z(\cM),\tau)$  such that
\begin{eqnarray*}
&\qquad &   \norm{T_{a,J}(x)}_{L_r(\mathcal{N},\tau_2)}^r \\
& = & \norm{aJ(x)}_{L_r (\mathcal{N},\tau_2)}^r
\\
&  =& \norm{aJ(x z) +aJ(x ({\bf 1}- z) )    }_{L_r (\mathcal{N},\tau_2)}^r
\\
&  =& \norm{J(J^{-1}(a)\cdot x z)+J(  x \cdot J^{-1}(a)({\bf 1}-z) ) }_{L_r(\mathcal{N},\tau_2)}^r \\
&  =& \tau_2\left(\left |J\left(J^{-1}(a)\cdot x z+   x \cdot J^{-1}(a)  ({\bf 1}-z) \right) \right |^r\right) \\
&\stackrel{\mbox{\tiny\cite[Prop. 2.19]{HSZ}}}{=}&\tau_2\left(J\left(  \left | J^{-1}(a)\cdot x z \right| +  \left| \left (x \cdot J^{-1}(a)  ({\bf 1}-z) \right)^*  \right |  \right) ^r\right) \\
&=& \tau_2\left(J\left(  \left | J^{-1}(a)\cdot x z \right| ^r  +  \left| \left (x \cdot J^{-1}(a)  ({\bf 1}-z) \right)^*  \right |  ^r \right)\right)  \\
&    \stackrel{\mbox{\tiny Lemma \ref{segalRN}}}{=} & \tau_1\left(b   \left | J^{-1}(a)\cdot x z \right| ^r  + b \left| \left (x \cdot J^{-1}(a)  ({\bf 1}-z) \right)^*  \right |  ^r   \right) \\
  & =& \tau_1\left(b   \left | J^{-1}(a)\cdot x z \right| ^r \right) + \tau_1 \left(b   \left| \left (x \cdot J^{-1}(a)  ({\bf 1}-z) \right)^*  \right |  ^r   \right) \\
&    =& \tau_1\left(b   \left | J^{-1}(a)\cdot x z \right| ^r \right) + \tau_1 \left(  b \left| x \cdot J^{-1}(a)  ({\bf 1}-z)  \right |  ^r   \right) \\
&    \le & \norm{b^{\frac1r}\cdot   J^{-1}(a)}_\infty ^r  \tau_1\left(   \left | x z\right| ^r   \right)     + \norm{b^{\frac1r} \cdot   J^{-1}(a)}_\infty^r    \tau_1\left( \left|     x ({\bf 1}-z) \right |  ^r   \right) \\
& = &     \norm{b^{\frac1r} \cdot  J^{-1}(a)}_\infty ^r \cdot \norm{x}_{L_r (\cM,\tau_1)} ^r .
\end{eqnarray*}
The proof is complete.
\end{proof}

Recall the following   facts   concerning dilation operators and Boyd indices, which are  known to experts, see e.g. \cite[p.323, line 22]{KR2}. Given $0<a\le\infty$, for any measurable function $f$ in a symmetric space $E(0,a)$ and $s>0$, the dilation $D_s$ of $f$ is defined by setting
$$(D_s f)(t)=\begin{cases}
f(\frac{t}{s}), & \mbox{if }t\in(0,a)~and~\frac{t}{s}\in(0,a), \\
0 ,  & \mbox{if }t\in(0,a)~and~\frac{t}{s}\notin(0,a).
\end{cases} $$
\begin{lem}\label{J mu lemma}
Let $(\mathcal{M},\tau_1)$ and $(\cN,\tau_2)$ be two noncommutative probability spaces. 
Let $J:\cM\to \cN$ be a Jordan $*$-isomorphism.  
Let $\beta>0$ and
let $q=\chi_{_{(\beta,\infty)}}(b)$,
 where $0\leq b\in L_1(Z(\mathcal{M}),\tau_1)$ such that
$\tau_2(J(x))=\tau_1(bx), ~x\in\mathcal{M}.$ If $$ q\ge  r(zx) +  l(({\bf 1}-z)x) ,$$ then
$$\mu(J(x))\geq D_{\beta}\left( \mu(x)\right), $$
 where $z\in P(Z(\cM))$ such that $J(z\cdot )$ is a $*$-homomorphism and $J(({\bf 1}-z)\cdot )$ is a $*$-anti-homomorphism.
\end{lem}
\begin{proof}
Note that $J$ is an isometry from $\cM$ onto $\cN$ \cite{K51}. 
By the definition of singular value functions,
 we have
\begin{align*}
&\qquad    \mu(t;J(x))\\
&~=~\inf \Big\{
\left\|J(x)({\bf 1}-p)\right\|_{\infty}:\ p\in P(\cN),~ p\leq  r(J(x)),\ \tau_2(p)\le t\Big\}\\
 &~=~\inf\Big\{
\left\|J(z x({\bf 1} -J^{-1}(p))) +J(( {\bf 1}-z)({\bf 1} -J^{-1}(p))   x) \right\|_{\infty}:\\
&\qquad \qquad \qquad \  p\in P(\cN),~  p\leq r(J(x)),~ \tau_2(p)\le t\Big\}\\
&~= ~\inf\Big\{
\left\| z x({\bf 1} -J^{-1}(p)) + ( {\bf 1}-z)({\bf 1} -J^{-1}(p))   x  \right  \|_{\infty}:\\
&\qquad \qquad \qquad   p\in P(\cN),~  J^{-1}(p)\leq r(zx) + l(({\bf 1}-z)x),~ \tau_2(p)\le t\Big\}\\
&~=~\inf\Big\{ \left\| z x({\bf 1} -p_0 ) + ( {\bf 1}-z)({\bf 1} -p_0)   x    \right \|_{\infty}:\\
&\qquad \qquad \qquad  \  p_0\in P(\cM),~ p_0\leq  r(zx) +  l(({\bf 1}-z)x),\ \tau_2(J(p_0))\le t\Big\}\\
&~=~\inf\Big\{\left\| z x({\bf 1} -p_0 ) + ( {\bf 1}-z)({\bf 1} -p_0)   x  \right \|_{\infty}:\\
&\qquad \qquad \qquad \  p_0\in P( \cM),~ p_0\leq  r(zx) +  l(({\bf 1}-z)x) ,\ \tau_1(bp_0)\le t\Big\}.\end{align*}
By the condition
$q=\chi_{(\beta,\infty ) } (b)$,   for any $p_0$ satisfying $p_0\leq r(zx) +  _l(({\bf 1}-z)x)\le q$ and  $\tau_1(bp_0)\le t$, we have
$$\tau_1(\beta p_0)=\tau_1 (\beta p_0  q p_0) \le  \tau_1 (p_0  bq p_0)=\tau_1(b p_0) \le t,$$
i.e.,
  $\tau(p_0)\le \frac{t}{\beta}.$
  On the other hand, the assumption   $q\ge r(zx) +   l(({\bf 1}-z)x)$ implies that
$$\mu(t;J(x))\geq\inf\left\{\left\|x({\bf 1}-p_0)\right\|_{\infty}:\  p_0 \in \cM,~ p_0\leq q,\ \tau_1(p_0)\le \frac{t}{\beta}\right\}=\mu\left(\frac{t}{\beta};x\right).$$
In other words, $\mu(J(x))\geq D _{\beta}\mu(x).$
\end{proof}

%Due to the lack of suitable references, we include a short proof for completeness.
For a symmetric function space $E(0,a)$, $0<a\le \infty$, 
the Boyd indices of $E(0,a)$ are defined by
$$p_{_{X(0,a)}} = \sup_{s>1 } \frac{\ln s}{\ln \norm{D_s}_{X(0,a) \to X(0,a) }}  $$
and
$$q_{_{X(0,a) }} = \inf_{s<1 } \frac{\ln s}{\ln \norm{D_s}_{X(0,a) \to X(0,a)}} .   $$

Given a symmetric function space $E(0,b)$, $0<b\le \infty$,
 and   $0<a\le b$, we define $E(0, a)$ by setting
\begin{align*}
E(0,a):=\Big\{ &f\in S(0,a): g\in E(0,b),\\
&\mbox {where}~ g(x) =
\begin{cases}
    f(t) & \text{if } t\in (0,a), \\
    0  & \mbox{otherwise},
\end{cases}\mbox{ with }
  \norm{f}_{E(0,a)}:=\norm{g}_{E(0,b)} \Big\}.
\end{align*}
Note that, $E(0, a)$  is  a $1$-complemented  subspace of $E(0,b)$.
%For simplicity, we denote  $\norm{\cdot}_{E(0,a)}$ by  $\norm{\cdot}_E$.

\begin{fact}\label{first kr fact}
Let $X(0,1)$ be a symmetric function space. For any $0<\delta\le 1$, the Boyd indices of $X(0,\delta)$ equal to those of $X(0,1)$.
\end{fact}

\begin{fact}\label{second kr fact}
Let $X(0,1)$ be a symmetric function space.   For every $\beta>0$ and $\delta>0$, we have
\begin{align}\label{KRinequality}\left\|D_{\beta} \right\|_{X(0,\delta) \to X(0,\delta) }\geq\max\left\{
\beta^{\frac1{p_{_X}}},\beta^{\frac1{q_{_X}}}
\right\},\end{align}
where $p_X$ and $q_X$ stands for the Boyd indices of $X(0,1)$.
\end{fact}
\begin{proof}It suffices to prove \eqref{KRinequality} for   the case when $\beta\ne 1$. By Fact \ref{first kr fact}, we have
$$ p_{_X}  =p_{_{X(0,\delta)}}=\lim_{s\to \infty } \frac{\ln s}{\ln \norm{D_s}_{X(0,\delta)\to X(0,\delta)}}\stackrel{\mbox{\tiny \cite[Def.2.b.1]{LT2}}}{=}\sup_{s>1 } \frac{\ln s}{\ln \norm{D_s}_{X(0,\delta)\to X(0,\delta)}} $$
and
$$ q_{_X}=q_{_{X(0,\delta)}} =\lim_{s\to 0+ } \frac{\ln s}{\ln \norm{D_s}_{X(0,\delta) \to X(0,\delta)}}\stackrel{\mbox{\tiny \cite[Def.2.b.1]{LT2}}}{=}\inf_{0<s<1 } \frac{\ln s}{\ln \norm{D_s}_{X( 0,\delta)\to X(0,\delta)}} .$$
Therefore, we have
$$  p_{_X}\ge   \frac{\ln \beta }{\ln \norm{D_\beta}_{X(0,\delta)\to X(0,\delta)}}, \mbox{ if }\beta>1,$$
and
$$  q_{_X}\le  \frac{\ln \beta }{\ln \norm{D_\beta}_{X(0,\delta)\to X(0,\delta)}}, \mbox{ if }0<\beta <1. $$
That is,
 $$ \ln \norm{D_\beta}_{X(0,\delta)\to X(0,\delta)}\ge    \frac{\ln\beta } {p_{_X} } , \mbox{ if }\beta>1,$$
and
$$\ln \norm{D_\beta}_{X(0,\delta)\to X(0,\delta)}\ge   \frac{ \ln  \beta } {q_{_X} }, \mbox{ if }0<\beta <1. $$
We have
$$\|D_{\beta}\|_{X(0,\delta)\to X(0,\delta)}\geq\max\left\{
\beta^{\frac1{p_{_X}}},\beta^{\frac1{q_{_X}}}
\right\}.$$
This completes the proof.
\end{proof}

\begin{lem}\label{lemma JA}
Let $X(0,1)$ and $Y(0,1)$ be  symmetric function spaces.
Assume that $X(0,1)$ has  Boyd indices $p_{_X}  $ and $q_{_X} $.
Let $(\cM,\tau_1)$ and $(\cN,\tau_2)$ be noncommutative probability spaces and let $T$ be of the form
\eqref{Tform}.
For any $p_{_X}\leq r\leq q_{_X}$,
we have
$$\left\| b^{\frac1r}\cdot J^{-1}(a) \right\|_{L_{\infty}(\mathcal{N},\tau_2)}\le \norm{T}_{X(\mathcal{M},\tau_1) \to Y(\mathcal{N},\tau_2)}.   $$
\end{lem}
\begin{proof} Without loss of generality, we may assume that  $$\norm{T}_{X(\mathcal{M},\tau_1) \to Y(\mathcal{N},\tau_2)}=1.$$
	
Assume by contradiction that  there exist numbers $\alpha,\beta>0$ such that $$\alpha\beta^{\frac1r}>1$$ and   $$pq\neq0,$$ where $$p=\chi_{(\alpha,\infty)}(J^{-1}(a)) \mbox{ and } q=\chi_{(\beta,\infty)}(b).$$
In particular, $b\in L_1(Z(\cM),\tau_1) $ implies that $q\in Z(\cM_1)$.	
Note that  $p=J^{-1}(\chi_{(\alpha,\infty)}(a))$  (see \eqref{functional calculus J} above) and therefore,
 $J(p)=\chi_{(\alpha,\infty)}(a)$. For every $x\in X(\mathcal{M},\tau_1) ,$ we have
\begin{align*}
\left\|T_{a,J}(x) \right\|_{Y(\mathcal{N},\tau_2)}&\geq
\norm{ \alpha e^{a}(\alpha,\infty ) J(x)}_{Y(\cN,\tau_2)} \\
& = \norm{ \alpha \cdot  J \left(  J^{-1} (e^{a}(\alpha,\infty )) \right)  J(x)}_{Y(\cN, \tau_2)} \\
& =
\alpha \left\|J(p)\cdot J(x) \right\|_{Y(\mathcal{N},\tau_2)}.
%&=\alpha \left\|J(pxz ) +J(xp({\bf 1}-z)) \right\|_{X(\mathcal{M},\tau)}.
\end{align*}
If $x\in X(\mathcal{M}_{pq},\tau_1)$ (here, 
$\cM_{pq}$ stands for the reduced algebra generated by the projection $pq$),
 then
\begin{eqnarray}
\label{Taj:estimate}
\left\|
T_{a,J}(x)\right\|_{Y(\mathcal{N},\tau_2)}& \geq &  \alpha
\left\|J(p  )  J(x ) \right\|_{Y(\mathcal{N},\tau_2)}\nonumber\\
&\stackrel{\mbox{\tiny \cite[Th.4.3]{Weigt}}}{=}& \alpha   \left\|J(p  x ) \right\|_{Y(\mathcal{N},\tau_2)}\nonumber \\
& =  & \alpha \left\|\mu(J(x))\right\|_{Y(0,1)}\\
& \stackrel{\tiny \mbox{Lemma}~\ref{J mu lemma}}{\geq} &  \alpha\left\|D_{\beta}\mu(x)\right\|_{X(0,1)}.\nonumber
\end{eqnarray}

Fix $\varepsilon>0$ and set $\delta:=\tau(J(p\wedge q)).$
By Fact \ref{second kr fact},
we have
\begin{align} \label{f7}\left\|D_{\beta} \right\|_{X(0,\delta) \to X(0,\delta) }\geq\max\left\{
\beta^{\frac1{p_X}},\beta^{\frac1{q_X}}
\right\}.\end{align}
Therefore,  for any $\varepsilon>0$, there exists
  $f=\mu(f)\in E(0,1)$ with $\norm{f}_{E(0,1)}=1$
  such that $$\left \|D_{\beta}\mu(f)\right\|_{X(0,1)} \geq(1-\varepsilon)\left\|D_{\beta}\right\|_{X(0,1) \to X(0,1)} .$$
Choose $x\in E(\cM_{pq},\tau_1)$ such that $\mu(x)=f$, see e.g. \cite[Lemma 1.3]{CKS} or \cite{Kalton_S}. In particular,  we have
\begin{align}\label{1epsilon}
\left \|D_{\beta}\mu(x)\right\|_{X(0,1)}\geq(1-\varepsilon)\left \|D_{\beta} \right \|_{X(0,1)\to X(0,1)} \quad\mbox{and}\quad  \left\|x\right\|_{X(\mathcal{M},\tau_1)}=1.
\end{align}
We obtain that
\begin{eqnarray*}\left\|T \right \|_{X(\mathcal{M},\tau_1)\to Y(\mathcal{N},\tau_2)}&  \geq &\left\|T_{a,J}(x)\right \|_{Y(\mathcal{N},\tau_2)}\\
&\stackrel{\eqref{Taj:estimate}}{\geq} &\alpha \left \|D_{\beta}\mu(x)\right \|_{  X(0,1)} \\
&\stackrel{\eqref{1epsilon}}{\geq}&\alpha (1-\varepsilon)\left \|D_{\beta}\right \|_{X(0,1)\to X(0,1)} \\
&\stackrel{\eqref{f7}}{\ge} &\alpha(1-\varepsilon) \cdot  \max\left\{ \beta^{\frac1{p_X}},\beta^{\frac1{q_X}} \right\} .
\end{eqnarray*}

Since $\epsilon>0$ is arbitrarily small, it follows from $p_X\leq r\leq q_X$ that
$$\left\|T \right \|_{X(\mathcal{M},\tau_1)\to Y(\mathcal{N},\tau_2)} \ge\alpha  \beta^{\frac1r}>1,$$
which is a contradiction.
\end{proof}

Combining Lemmas \ref{yuyuo} and \ref{lemma JA}, we obtain the following noncommutative version of a result due to Kalton \cite{Kalton} (see also \cite[Proposition 7.1]{KR2}).
\begin{theorem}\label{NCKR}

Let $X(0,1)$ and $Y(0,1)$ be  symmetric function spaces.
Assume that $X(0,1)$ has  Boyd indices $p_{_X} $ and $q_{_X}$.
Let $(\cM,\tau_1)$ and $(\cN,\tau_2)$ be noncommutative probability spaces and let $T$ be of the form
\eqref{Tform}.
For any $p_{_X}\le r\le q_{_X}$ and any atomless  finite von Neumann algebras $\cM$ and $\cN$ equipped with a faithful normal tracial states $\tau_1$ and $\tau_2$ respectively, we have
  $$\norm{T}_{L_r(\cM,\tau_1)\to L_r(\cN,\tau_2)}\le \norm{T}_{X(\cM,\tau_1)\to Y(\cN,\tau_2)} . $$
\end{theorem}
\section{Proof of  Theorem \ref{KRZ} for the  case when $E(0,1)=L_p(0,1)$}

\begin{theorem}
Let $X(0,1  ), Y(0,1)$ be   symmetric function spaces in the sense of Lindenstrauss and Tzafriri.
Let $\cM_1$ and $ \cM_2 $ be   atomless finite von Neumann algebras  equipped with   faithful normal tracial states  $\tau_1$ and $\tau_2$, respectively.
%  and let  $E(\cM_1,\tau_1)$ and $F(\cM_2,\tau_2)$ be the noncommutative  symmetric space corresponding to $E(0,1 )$ and $F(0,1)$ respectively.
Let $T:X(\cM_1,\tau_1) \to Y (\cM_2,\tau_2)$ be a surjective isometry of the form \eqref{Tform}. If $X(0,1)=L_p(0,1)$ up to an  equivalent norm, then $T$ is a surjective isometry  from $L_p(\cM_1,\tau_1)$ onto $L_p(\cM_2,\tau_2)$.
\end{theorem}

\begin{proof}%Assume that $p=2$.  Order continuity  implies that $E(0,1)$ is fully symmetric.
%The space $E(\cM,\tau)$ isometric to $L_2(\cM,\tau)$ implies that $E(0,1)$ is isometric to a complemented subspace of $L_2(\cM,\tau)$,
% which again is a separable infinite-dimensional Hilbert space.
%In particular, $E(0,1)$ is isometric to $L_2(0,1)$.
%By \cite[Theorem 1]{AZ} (see also \cite{Semenov}), $E(0,1)=L_2(0,1)$ up to a proportional norm.

Note that $Y(0,1)$ coincides with $L_p(0,1)$ up to equivalent norms.
Indeed, 
if 
$Y(0,1) $ coincides with $L_q(0,1)$ for some $q\in [1,\infty )$, then 
$L_{p}(\cM_1,\tau_1)$ is isometric to $L_q(\cM_2,\tau_2)$, 
which is impossible when $p\ne q$, see e.g.  \cite[p.267]{Fack87}. 
Therefore,  $Y(0,1)\ne L_q(0,1)$ for all $1\le q<\infty $.
It follows from Theorem \ref{KRZ} (2) that isometries between $Y(\cM_2,\tau_2)$
and $X(\cM_1,\tau_1)$ are generated by unitary operators and trace-preserving Jordan $*$-isomorphisms, which are clearly $L_p$-isometries.
This implies that $Y(0,1)$ coincides with $L_p(0,1)$, which is a contradiction.

By Theorem \ref{NCKR}, we have
$$\norm{T}_{L_p(\cM_1,\tau_1)\to L_p(\cM_2,\tau_2)}\le \norm{T}_{X(\cM_1,\tau_1)\to Y(\cM_2,\tau_2)}= 1 . $$
Applying Theorem \ref{NCKR} to the surjective isometry  $T^{-1}:Y(\cM_2,\tau_2)\to X(\cM_1,\tau_1)$, we obtain
 $$\norm{T^{-1}}_{L_p(\cM_2,\tau_2)\to L_p(\cM_1,\tau_1)}\le \norm{T^{-1} }_{Y(\cM_2,\tau_2)\to X(\cM_1,\tau_1)} =1 . $$
 Hence, $ T$ is an isometry from $L_p(\cM_1,\tau_1)$ onto $L_p(\cM_2,\tau_2)$.
\end{proof}

\begin{remark}
If a symmetric function space $E(0,1)=L_p(0,1)$ up to equivalent norms, then isometries on $E(0,1)$ must be isometries on $L_p(0,1)$ (see Theorem~\ref{KRZ1}) but the converse fails.

Indeed,  let $L_1(0,1)$ be equipped with a norm $\norm{f}=\int_0^\frac12 \mu(s;f)ds$, $f\in L_1(0,1 )$.
 Let $T$ be an arbitrary  surjective isometry on $(L_1(0,1),\norm{\cdot})$.
 By \cite{KR2,R} or \cite{Z77}, we have
  $$(Tf)(t)= a(t) f(\sigma(t)) , ~t\in (0,1), $$
where $0\le a\in L_1(0,1)$ and $\sigma$ is an invertible Borel mapping from $(0,1)$ onto $(0,1)$.
  Assume that $a\ne \chi_{(0,1)}.$
  Then there exists $A,B,C\subset (0,1)$ such that
  $a>1$ on $A$ and $a<1$ on $B$, $a=1$ on $C$.
  For any subset $C'$ of $C$ with $m(C')<\frac12$, we have
    $$\norm{\chi_{_{C'}}(\sigma(\cdot )) } =\norm{a(\cdot) \chi_{_{C'}}(\sigma(\cdot ))  }=\norm{T(\chi_{_{C'}}) (\cdot )} =\norm{\chi_{_{C'}} (\cdot)} .$$
   In particular,  this implies that $\sigma$ preserves the measure on $C$.

Assume by contradiction that $m(A)\ne 0$. Assume that $m(A\cup C) \ge \frac12$. 
\begin{enumerate}
     \item  If $ m(A)\ge \frac12$, then there exists a subset $A'$ of $A$ such that $m(A')=\frac12$;
       \item or if $ 0< m(A)< \frac12$, then there exists a subset $C'$ of $C$ such that $m(A\cup C')=\frac12$. Set $A':=A\cup C'$.
\end{enumerate}
Therefore,
\begin{align*}%\label{A'12}
\norm{\chi_{_{A'}}(\sigma(\cdot )) } <\norm{a(\cdot ) \chi_{_{A'}}(\sigma(\cdot ))  }=\norm{T(\chi_{_{A'}} )(\cdot)} =\norm{\chi_{_{A'}}(\cdot)}=\frac12.
\end{align*}
This implies that  $m(\sigma^{-1}(A')) <m(A')=\frac12$.
Now, define $x=\chi_{_{A'}} +  %\frac{1}{\norm{T \chi_{_{B'}}}_\infty }
 \chi_{_{B'}}$, where $B'\subset (0,1)\setminus A'$.
% such that $T \chi_{_{B'}}$ is uniformly bounded.
%Observe that $ \left| \frac{1}{\norm{T \chi_{_{B'}}}_\infty } a  \right| \le 1 $.
We have
$$\frac12=  \norm{x} =\norm{Tx}=\norm{ a \chi_{_{\sigma^{-1}(A') }}  + %\frac{1}{\norm{T \chi_{_{B'}} }_\infty }
a  \chi_{_{\sigma^{-1}(B')} }  } >
\norm{ a \chi_{_{\sigma^{-1}(A') }} } =\norm{T \chi_{_{A'}}}=\norm{  \chi_{_{A'}}  }=   \frac12,$$
which is a contradiction.

  The case when $m(A)<\frac12$ and $m(A\cup C)<\frac12$ follows from the above result  by taking the isometry $T^{-1}$.

 We conclude that $a=1$, $a.e.$.
 This shows that isometry on $(
 L_p(0,1),\norm{\cdot}_p)$ may not be an isometry on $L_p(0,1)$ equipped with an equivalent norm.
\end{remark}
\section{Proof of Theorem \ref{K-R-Z} for the case when $E(0,1)$ is separable}

Theorem \ref{K-R-Z} for the case when $E(0,1)$ is separable (or minimal) is a   corollary  of Theorem \ref{KRZ}.

\begin{cor}\label{order continuous}
Let $E(0,1  )$ and $ F(0,1)$ be separable (or minimal)    symmetric function spaces.
Let $\cM_1$ and $ \cM_2 $ be   atomless finite von Neumann algebras  equipped with   faithful normal tracial states  $\tau_1$ and $\tau_2$, respectively.
Let $T:E(\cM_1,\tau_1) \to F(\cM_2,\tau_2)$ be a surjective isometry.
Then,  \begin{enumerate}
        \item if $E(0,1)=L_p(0,1)$ up to an  equivalent norm, then $T$ is a surjective isometry  from $L_p(\cM_1,\tau_1)$ onto $L_p(\cM_2,\tau_2)$;
        \item if $\norm{\cdot}_E$ is not equivalent to $\norm{\cdot}_p$, then
$$T(x)=u\cdot J(x), ~\forall x\in E(\cM_1,\tau_1),$$
where $u$ is a unitary operator in $\cM_2$ and $J:S(\cM_1,\tau_1)\rightarrow S(\cM_2,\tau_2)$ is a trace-preserving Jordan $*$-isomorphism.
\end{enumerate}
\end{cor}
\begin{proof}
By \cite[Theorem 1.2]{HS}, any surjective isometry $ T:E(\cM_1,\tau_1)\rightarrow F(\cM_2,\tau_2)$ is of the form
\begin{align}\label{we}
T(x) = J(x)Az + BJ(x) ({\bf 1} - z), \quad \forall x \in \cM_1,
\end{align}
where $ A, B \in F(\cM_2, \tau_2) $, $ z \in P(Z(\cM_2)) $, and $ J:\cM_1\rightarrow \cM_2 $ is a surjective Jordan $*$-isomorphism.
 %As before, for simplicity, we may assume that $B=0$.

Since $J$ is surjective, %it follows that $J$ is normal \cite[Appendix A]{RR}. Hence, 
it follows form Corollary \ref{lemma measure-measure} that $J:\cM_1\rightarrow \cM_2$ is measure-measure continuous.
By Proposition \ref{lga},
 $J$ can be uniquely extended to a Jordan $*$-isomorphism $S(\cM_1,\tau_1)$ onto $S(\cM_2,\tau_2)$, which is continuous in measure topology.

Since $E(\cM_1,\tau_1)$ has order continuous norm (or is minimal), it follows that
for any $x\in E(\cM_1,\tau_2)$, there exists a sequence $\{x_n\}_{n\ge 1}\subset \cM_1$ such that $$
\mbox{$\norm{x_n-x}_E \to 0$ as $n\to \infty $}$$ (therefore, $x_n\to x$ in measure as $n\to \infty $~\cite[Proposition 11]{DP2014}).
In particular, we have $T(x_n)\to T(x)$ in $\norm{\cdot}_E$  as $n\to \infty$. In particular, $$T(x_n)\to T(x) \mbox{  as $n\to \infty $}$$ 
in measure~\cite[Proposition 11]{DP2014}.
Moreover, since $J$ is continuous with respect to the measure topology, it follows that $J(x_n)\to J(x)$ in measure as $n\to \infty $.
By \eqref{we} and~\cite[Proposition 2.6.11]{DPS}, we obtain that
\begin{align*}
T(x) =J(x)Az +B J(x) ({\bf 1}-z),~\forall x\in E(\cM_1,\tau_1), \end{align*}
where $ A, B \in F(\cM_2, \tau_2) $, $ z \in P(Z(\cM)) $, and $ J $ is a surjective Jordan $*$-isomorphism from $S(\cM_1,\tau_1)$ onto $S(\cM_2,\tau_2)$.
Applying Theorem \ref{KRZ}, we   complete the proof.
\end{proof}

\section{Proof of Theorem \ref{K-R-Z} for the case when $E(0,1)$ has the Fatou property}

The following corollary  provides a proof for Theorem \ref{K-R-Z}  in the setting  when $E(\cM,\tau)$ has the Fatou property.

\begin{cor}\label{cor KRZ Fatou}
Let $E(0,1  )$ and $ F(0,1)$ be   symmetric function spaces having the Fatou property.
Let $\cM_1$ and $ \cM_2 $ be   atomless finite von Neumann algebras  equipped with   faithful normal tracial states  $\tau_1$ and $\tau_2$, respectively.
Let~$T:E(\cM_1,\tau_1) \to F(\cM_2,\tau_2)$ be a surjective isometry.
Then,  \begin{enumerate}
        \item if $E(0,1)=L_p(0,1)$ up to an  equivalent norm, then $T$ is a surjective isometry  from $L_p(\cM_1,\tau_1)$ onto $L_p(\cM_2,\tau_2)$;
        \item if $\norm{\cdot}_E$ is not equivalent to $\norm{\cdot}_p$, then
$$T(x)=u\cdot J(x), ~\forall x\in E(\cM_1,\tau_1),$$
where $u$ is a unitary operator in $\cM_2 $ and $J:S(\cM_1,\tau_1)\rightarrow S(\cM_2,\tau_2)$ is a trace-preserving Jordan $*$-isomorphism.
\end{enumerate}
\end{cor}

\begin{proof}
By Theorem \ref{th:iso}, any surjective isometry $ T:E(\cM_1,\tau_1)\rightarrow F(\cM_2,\tau_2)$ is of the form
\begin{align}\label{we'}
T(x) =  J(x)A z + BJ(x) ({\bf 1} - z), \quad \forall x \in \cM_1,
\end{align}
where $ A, B \in F(\cM_2, \tau_2) $, $ z \in P(Z(\cM_2)) $, and $ J:\cM_1\rightarrow \cM_2 $ is a surjective Jordan $*$-isomorphism.

For any $x\in E(\cM_1,\tau_1)^+$,
defining $x_n=x e^{x}(0,n)$, $n\ge 1$.
We have $x-x_n\downarrow 0$ as $n\to \infty $.
Note  that $$
x_n\stackrel{\sigma(E,E^{\times})}{\longrightarrow}x$$
as $n\to \infty$.
Since $T$ is $\sigma(E,E^\times)-\sigma(F,F^\times)$-continuous (see Theorem \ref{weak-continous of T}), it follows that $$T(x_n)\stackrel{\sigma(E,E^{\times})}{\longrightarrow}T(x)$$
as $n\to \infty$. 
Note that $J$  can be extended to a Jordan $*$-isomorphism from $S(\cM_1,\tau_1)$ onto $S(\cM_2,\tau_2)$ (see Proposition \ref{lga}), which is continuous with respect to the measure topology.
%Therefore, 
% Since $J$ is normal \cite[Appendix A]{RR}, it follows that 
%$J(x_n)\to  J(x)$ in the measure topology as $n\to\infty$.
Moreove, we have
$$J(x_n) =J(x )J(e^{x}(0,n) ) \to J(x)$$ as $n\to \infty $ in the measure topology
and $$J(e^{x}(0,n) ) \uparrow  J(e^{x}(0,\infty ))  . $$
We have 
$$ J(x_n)A z +BJ(x_n ) ({\bf 1}-z )  \to J(x )A z +BJ(x  ) ({\bf 1}-z )   $$
as $n\to \infty$ in measure. 
Moreover, we have 
$$\norm{J(x_n)A z +BJ(x_n ) ({\bf 1}-z ) }_{F(\cM_2,\tau_2)} =\norm{T(x_n)}_{F(\cM_2,\tau_2)} =\norm{x_n}_{E(\cM_1,\tau_1)} \le \norm{x}_{E(\cM_1,\tau_1)}.$$
By the Fatou property, we obtain that \cite[Theorem 32]{DP2014}
$$\norm{  J(x )A z +BJ(x  ) ({\bf 1}-z ) }_{F(\cM_2,\tau_2)}  \le \norm{x }_{E(\cM_1,\tau_1)}.$$
For any $y\in F(\cM_2,\tau_2)^\times$, we have 
\begin{eqnarray*}
& & \tau\left(   \left( J(x_n)A z +BJ(x_n ) ({\bf 1}-z )\right)  y \right )\\
& = &
\tau\Big(   \Big( J(x )J(e^{x}(0,n) ) A z +BJ(x )J(e^{x}(0,n) )  ({\bf 1}-z )   \Big)  y \Big )\\
& \to &
\tau\left(   \left( J(x )  A z +BJ(x )  ({\bf 1}-z )   \right)  y \right )
\end{eqnarray*}
as $n\to \infty$.   
%By the Fatou property, 
 %we know that $F(\cM_2,\tau_2)$ is $\sigma(F,F^\times)$-sequentially complete~\cite[Proposition 3.1]{DK}.
Hence, $$  J(x_n)A z +BJ(x_n ) ({\bf 1}-z ) \to 
  J(x )  A z +BJ(x )   ({\bf 1}-z )   \in F(\cM_2,\tau_2)$$
as $n\to \infty $ in the $\sigma(F,F^\times)$-topology.
 %and therefore,  we have 
 % $$J(x_n)\stackrel{\sigma(F,F^{\times})}{\longrightarrow}J(x) \mbox{ as $n\to \infty$}.$$ 
Since any element in $  E(\cM_1,\tau_1)$ is a linear combination of 
four elements in $E(\cM_1,\tau_1)_+$ and $ J,T$ are linear operators, it follows from \eqref{we'} that
\begin{align*}
T(x) = J(x) A z +B J(x)  ({\bf 1}-z),~\forall x\in E(\cM_1,\tau_1), \end{align*}
where $ A, B \in F(M_2, \tau_2) $, $ z \in P(Z(\cM_2)) $, and $ J :S(\cM_1,\tau_1)\rightarrow S(\cM_2,\tau_2)$ is a  Jordan $*$-isomorphism. The statement of the theorem follows from Theorem~\ref{KRZ}.
\end{proof}

\begin{remark}
Let $E(0,1  ) $ and $ F(0,1)$ be   symmetric function spaces in the sense of Lindenstrauss and Tzafriri.
Let $\cM_1$ and $ \cM_2 $ be   finite von Neumann algebras  equipped with   faithful normal tracial states  $\tau_1$ and $\tau_2$, respectively,  and let  $E(\cM_1,\tau_1)$ and $F(\cM_2,\tau_2)$ be the noncommutative  symmetric spaces corresponding to $E(0,1 )$ and $F(0,1)$, respectively.
Theorem \ref{K-R-Z}  shows that if $ E(\cM_1,\tau_1)  $ is isometric to $ F(\cM_2,\tau_2)$, then 
$E(0,1)$ coincides with $F(0,1)$ (as sets), which is a noncommutative analogue of the result obtained in \cite{Potepun71} (see also \cite[Corollary 4]{Ab91}). 

\end{remark}

\chapter{Mityagin's question and its noncommutative counterpart}\label{S:M}
In this chapter, we give a proof of Theorem \ref{Mityagin} and establish a noncommutative counterpart of it.

\section{Proof of Theorem \ref{Mityagin}}

\begin{proof}[Proof of Theorem \ref{Mityagin}]
Assume that there exists    a surjective isometry  $T$ from $E(0,1)$ onto $F(0,\infty)$.
By Zaidenberg's theorem \cite[Theorem 1]{Z77}, $T$ is of the form
\begin{align}\label{abc}
T(x)(t)=a(t)x(\sigma(t)),~x\in E(0,1), 
\end{align}
where $a\in F(0,\infty)$ and $\sigma$ is an invertible measurable transformation from $(0,\infty)$ onto $(0,1)$. In particular, $T$ is disjointness-preserving.

Denote $$I_n=\sigma(0,n).$$
Consider $T|_{E(I_n)}$. Since $T$ is disjointness-preserving, it follows that $T|_{E(I_n)}$ is a surjective isometry 
from $E(I_n)$ onto $F(0, n)$.

Let $$i_n:I_n\to (0,m(I_n))$$
be a measure-preserving transformation.
Define $$T_0f=f\circ i_n,~f\in E(0,m(I_n)).$$
 Note that $T_0$ is a surjective isometry from $E(0,m(I_n))$ onto $E(I_n)$.
Define
\begin{align}\label{E' label 2}
\begin{array}{ll}
&E'(0,1) := \\
 &\qquad \left\{f\in S(0,1)|   D_{m(I_n)}f \in E(0,m(I_n))
~ \mbox{and}~\norm{f}_{E'}=
\frac{\norm{   D_{m(I_n)}f }_E}{\norm{\chi_{_{I_n}}}_E}\right\} \\
\end{array}
\end{align}
and
\begin{align}\label{F' label 2}
\begin{array}{ll}
&
F'(0,1):=\\
& \qquad \left\{f\in F(0,1)|D_n f\in F(0,n)~ \mbox{and}~ \norm{f}_{F'}=\frac{1}{\norm{\chi_{_{(0,n)}}}_F}\norm{D_nf}_F  \right\}.\\
\end{array}
\end{align}
Clearly,
 $E'(0,1)$ and $F'(0,1)$ are symmetric function spaces.

Observe that the mapping
$$f\mapsto \frac{\norm{\chi_{_{(0,n)}}}_F}{\norm{\chi_{_{I_n}}}_E}D_{\frac1n}\circ T\circ T_0\circ D_{m(I_n)}(f)$$
is an isometry from $E'(0,1)$ onto $F'(0,1)$.
Indeed,  for any $f\in E'(0,1)$, we have 
\begin{align*}
&\quad \qquad \norm{D_{\frac1n}\circ T\circ T_0\circ D_{m(I_n)}(f)}_{F'(0,1)}
\\&~\stackrel{\eqref{F' label 2}}{=}\frac{1}{\norm{\chi_{_{(0,n)}}}_{F(0,\infty)} }\norm{T\circ T_0\circ D_{m(I_n)}(f)}_{F(0,\infty)}\\
&\quad =\quad \frac{1}{\norm{\chi_{_{(0,n)}}}_{F(0,\infty )}}\norm{T_0\circ D_{m(I_n)}(f)}_{E(0,1)}\\
&~\stackrel{\eqref{E' label 2}}{=}\frac{\norm{\chi_{_{I_n}}}_{E(0,1)} }{\norm{\chi_{_{(0,n)}}}_{F(0,\infty)}}\norm{f}_{E'(0,1)}.
\end{align*}
Since $E(0, 1)\ne L_p(0, 1)$ as sets, it follows that $E'(0,1)\ne L_p(0,1)$ as sets. By Theorem~\ref{KRZ1}, we have

 $$\frac{\norm{\chi_{_{(0,n)}}}_F}{\norm{\chi_{_{I_n}}}_E}D_{1/n}T\circ T_0\circ D_{m(I_n)}(f)=c_nf\circ \sigma_n, ~\forall f\in E'(0,1).$$
%Thus,
%\begin{align}\label{sigma}
%\frac{n}{m(I_n)}a\sigma\circ i_n(D_{\frac{m(I_n)}{n}}f)=c_nf\circ \sigma_n
%\end{align}
where $c_n:(0,1)\to \mathbb{C}$  is a non-vanishing Borel function with $|c_n|$ = 1 and
 $\sigma_n:(0,1)\to (0,1)$ is an 
 invertible measure-preserving Borel transformation.
Taking $f=\chi_{_{(0,1)}}$, we have
\begin{align}\label{c}
T(\chi_{_{I_n}})=D_n\circ\frac{\norm{\chi_{_{I_n}}}_E}{\norm{\chi_{_{(0,n)}}}_F}c_n f\circ \sigma_n=\frac{\norm{\chi_{_{I_n}}}_E}{\norm{\chi_{_{(0,n)}}}_F}\tilde{c}_n\chi_{_{(0,n)}}
\end{align}
where $\tilde{c}_n:(0,n)\to \mathbb{C}$  is a non-vanishing Borel function with $|\tilde{c}_n|$ = 1.
For any  $m\ge n$,  we have
\begin{align*}
T(\chi_{_{I_m}} )=\frac{\norm{\chi_{_{I_m}}}_E}{\norm{\chi_{_{(0,m)}}}_F}\tilde{c}_m\chi_{_{(0,m)}},
\end{align*}
where $\tilde{c}_m:(0,m) \to \mathbb{C}$  is a non-vanishing Borel function with $|\tilde{c}_m|$ = 1.
By the disjointness-preserving property of $T$, we obtain that $T\left(\chi_{_{I_n}}\right)$ and $T\left(\chi_{_{I_m\setminus I_n}}\right)$
have mutually disjoint supports. Hence, $T\left(\chi_{_{I_m\setminus I_n}}\right)$ is supported on $[n, m)$. We
have
\begin{align*}
\tilde{c}_m \frac{\norm{\chi_{_{I_m}}}_E}{\norm{\chi_{_{(0,m)}}}_F}  \chi_{(0,n)}
&=T(\chi_{_{I_m}})\chi_{_{(0,n)}}\\
&=T(\chi_{_{I_n}})\chi_{_{(0,n)}}+T(\chi_{_{I_m\setminus I_n}})\chi_{_{(0,n)}}\\
&=T(\chi_{_{I_n}})\chi_{_{(0,n)}}\\
&=
\tilde{c}_n  \frac{\norm{\chi_{_{I_n}}}_E}{\norm{\chi_{_{(0,n)}}}_F}   \chi_{_{(0,n)}}.
\end{align*}
Hence, we have $\frac{\norm{\chi_{_{I_m}}}_E}{\norm{\chi_{_{(0,m)}}}_F} = \frac{\norm{\chi_{_{I_n}}}_E}{\norm{\chi_{_{(0,n)}}}_F}$ and
$$\tilde{c}_n=\tilde{c}_m \quad \mbox{a.e. on}\quad (0,n) $$
for all $m\ge n$. By (\ref{abc}), we have
\begin{align}\label{a}
T(\chi_{_{I_n}})=a\cdot \chi_{_{(0,n)}}.\end{align}
Compare the equality (\ref{c}) and (\ref{a}), we  conclude that
$$|a|=\mbox{Const}~ a.e.. $$

Consider $\chi_{_{I_n/I_{n-1}}}\in E(0,1)$. Since $T$ is an isometry, it follows that $$\norm{\chi_{_{I_n\setminus I_{n-1}}}}_E={\rm Const} \cdot\norm{\chi_{_{(n-1,n)}}}_F ={\rm Const} \cdot \norm{\chi_{_{(0,1)}}}_F .$$ 
However, since $(0, 1)$ is a finite interval, it follows that $$\mbox{
$m(I_n \setminus I_{n-1})\to 0$ as $n\to \infty$.}$$
 By our assumption that $E(0,1)\ne L_\infty(0,1)$, we obtain that~\cite[p.118]{LT2} $$\norm{\chi_{_{I_n \setminus I_{n-1}}}}_E\to0$$ as $n\to \infty,$ which shows that $a = 0$. That implies that $T = 0$, which 
contracts the assumption that $T$ is an   isometry.
\end{proof}

\section{The noncommutative counterpart of Mityagin's question}
Below, we establish a noncommutative version of Theorem \ref{Mityagin}.
  \begin{theorem}\label{th5}
    If $E(0,1)$ is a separable  symmetric function space on $(0,1)$ and $E(0,1)\ne L_p(0,1)$ (as sets), $1\le p\le \infty $.
Let $\cM_1$ be an atomless von Neumann algebra equipped with a finite faithful normal trace $\tau_1$.
Then, $E(\cM_1,\tau_1)$ is not isometric to any noncommutative  symmetric   space $F(\cM_2,\tau_2)$ associated with an arbitary symmetric function space $F(0,\infty)$ and    any atomless  von Neumann algebra $\cM_2$ equipped with a semifinite infinite faithful normal trace $\tau_2$.
  \end{theorem}

  \begin{proof}

Let $T$ be a surjective isometry from $E(\cM_1,\tau_1)$ onto $F(\cM_2,\tau_2)$.
By~\cite[Theorem 1.2]{HS},  we have
\begin{align}\label{Txform}
T(x)=J(x)Az+BJ(x)(\mathbf{1}-z),\quad x\in \cM_1,
\end{align}
where $A,B\in F(\cM_2,\tau_2)$, $z\in P(Z(\cM_2 ))$,  and $J:\cM_1 \to \cM_2$ is a Jordan $*$-isomorphism\footnote{Note that $J$ can not be extended to a Jordan $*$-isomorphism from $S(\cM_1,\tau_1)$ onto $S(\cM_2,\tau_2)$, see \cite[Theorem 3.10 and 3.13]{Weigt}.}.

Without loss of  generality, we may assume that $A=Az, B=B(\mathbf{1}-z)$. Let $A=u_A|A|,B=u_B|B|$ be the polar decomposition. Since $T$ is surjective, it follows that $$\mbox{$u_Au_A^*=z$ and $u_Bu_B^*=\mathbf{1}-z$.}$$
Noting  that $z+u_B^*$ and $u_A^*+\mathbf{1}-z$ are unitary operators, we obtain  that 
\[
T'(\cdot):=(u_B^*+z)T(\cdot)(1-z+u_A^*)=J(\cdot)|A|+|B| J(\cdot)
\]
is a surjective isometry from $E(\cM_1,\tau_1)$ onto $F(\cM_2,\tau_2)$. Hence, without lossing generality, we may assume that both of $A$ and $B$ are   positive.

Since $E(0,1)$ is separable, it follows that $E(\cM_1,\tau_1)$
does not contain a copy of $\ell_\infty$~\cite[Theorem 5.6.30]{DPS},
 and therefore, $F(0,\infty)$ contains no copy of $\ell_\infty$, which implies that 
$F(0,\infty)$ is also separable~\cite[Theorem 5.6.30]{DPS}.
In particular, we have $$A,B\subset F(\cM_2,\tau_2)\subset S_0(\cM_2,\tau_2).$$ 
%Then, for any $x\in E(\cM_1,\tau_1)$, there exists  a sequence $\{x_n\}_{n\ge 1}\subset\cM_1$ such that $\left\|x-x_n\right\|_E \to 0$. Note that
%\[ T(x_{n})=J(x_{n})Az+BJ(x_{n})(\mathbf{1}-z). \]
%Since $T$ is continuous with respect to $\left\|\cdot\right\|_E$, it follows that $T(x_n)\to T(x)$ in $E(\cM_1,\tau_1)$ (in particular,  is  the measure topology).
% By the measure-continuity of  $J$, we have $  J(x_n)\to J(x). $
%Therefore,
%\begin{align}\label{Txform}
%T(x)=J(x)Az+BJ(x)(\mathbf{1}-z),\quad \forall  x\in E(\cM_1,\tau_1).
% \end{align} 

By \cite[Lemma 1.3]{CKS}, there exists    an atomless abelian von Neumann subalgebra  $\cN_2$  of $\cM_2$ containing all spectral projections  of $A$ and $B$,
such that $ \cN_2 $  is (trace-measure preserving) $*$-isomorphic to $L_\infty (0,\infty )$.

Note that $J^{-1}$ is a Jordan $*$-isomorphism from $\cM_2$ onto $\cM_1$. Denote
\[
\cN_1:=J^{-1}(\cN_2),
\]
which is an atomless abelian von Neumann subalgebra of $\cM_1$.
Since both $\cN_1$ and $\cN_2$ are abelian, it follows that $J:\cN_1\to \cN_2 $ is a $*$-isomorphism.
%By the uniqueness of the extension of $J\big|_{\cN_1}$ (see Lemma  \ref{lga}),  the restriction
%$$J\big|_{S(\cN_1,\tau_1)}:S(\cN_1,\tau_1)\to S(\cN_2,\tau_2)$$ is a *-isomorphism.

Let $\{q_n^a\}$, $\{q_n^b\}\subset P(\cN_2)$ such that $q_n^a\uparrow z$ and $q_n^b\uparrow {\bf 1}-z$ with $\tau_2(q_n^a),\tau_2(q_n^b)<\infty $.
Define 
$$p^a_n:=  J^{-1}(q^a_n) \mbox{ and }p^b_n:=  J^{-1}(q^b_n).$$ 
By \eqref{Txform}, we obtain that $T$ maps $(\cN_1)_{p^a_n}$ into $F( (\cN_2)_{q^a_n},\tau_2)$ (respectively,
maps $(\cN_1)_{p^b_n}$ into $F( (\cN_2)_{q^b_n},\tau_2)$). 
Moreover, since 
 $ (\cN_1)_{p^a_n +p^b_n}  $ and $( \cN_2)_{q^a_n+q^b_n} $ 
are equipped with finite faithful normal traces, it follows from 
 Corollary \ref{lemma measure-measure}
that $J|_{(\cN_1)_{p^a_n+p^b_n}} $ is continuous with respect to the measure topology.

We claim that $T$ is   a surjective isometry from
 $E\left((\cN_1)_{p^a_n+p^b_n} ,\tau_1\right)$ onto $F
\left( (\cN_2)_{q^a_n+q^b_m},\tau_2
\right)$ for all $n\ge 1$. 
Indeed, 
by Proposition~\ref{lga}, 
$J|_{(\cN_1)_{p^a_n+p^b_n}} $ extends to a   $*$-isomorphism from
 $S\left((\cN_1)_{p^a_n+p^b_n} ,\tau_1\right) $ onto  $S\left( ( \cN_2)_{q^a_n+q^b_n},\tau_2\right)  $. 
For   any $0\le y\in F\left( (\cN_2)_{q^a_n+q^b_n } ,\tau_2\right)$,  we have $$y\left( 
 ( A+B )  ( q^a_n+q^b_n)\right)^{-1}\in  S
\left(
 (\cN_2)_{q^a_n+q^b_n} ,\tau_2
\right) $$
and
   there exists
 $x:=J^{-1}\left(
y\left( 
 ( A+B )  ( q^a_n+q^b_n)\right)^{-1} \right)\in S\left(  (\cN_1)_{p^a_n+p^b_n},\tau_ 1\right)$
(note that 
 $x\ge 0$)
 such that
%\[
%J(x)A=y(A+B)^{-1}A=yz,
%\]
%\[
%BJ(x)=By(A+B)^{-1}=y(\mathbf{1}-z).
%\]
%Therefore, we have
\[
J(x)A+BJ(x)=y.
\]
We only need to observe that $x\in E(\cM_1,\tau_1)$. Indeed, 
since $T$ is a surjective isometry from
 $E(\cM_1,\tau_1)$ onto $F(\cM_2,\tau_2)$, it follows  
  that there exists $x'\in E(  \cM_1 ,\tau_1)$ such that
\[
T(x')= y=J(x)A+BJ(x).
\]

%The same argument used in Corollary \ref{order continuous} shows that 
%\begin{align}\label{Txformreduced}
%T(w)= J(w) A z+ BJ(w)  ({\bf 1}-z)
%\end{align}
%for all $w \in E\left(
%(\cN_1)_{p^a_n+p^b_n} ,\tau_1\right) $, $n\ge 1$. 
Let $x_m : =x 'e^{|x'| } (0,m  )  
 \in     \cN_1 $, $m \ge 1$. 
In particular,   $$|x_m|   \uparrow |x'|$$ and $x_m\to x'$ in $E(\cM,\tau)$ (in particular, in the measure topology) as $m\to \infty$. 
We have $$\norm{T(x_m) -T(x')}_F \to 0$$
as $m\to \infty$. 
Therefore, $$T(x_m) = J(x_m)A +BJ(x_m) \to J(x)A+BJ(x)$$
in the measure topology as $m\to \infty$. 
In particular, for any fixed $i\ge n$, we have 
$$ J\Big((p^a_i +p^b_i) 
x_m  (p^a_i +p^b_i) \Big) A
 = (q^a_i +q^b_i) 
\big(J(x_m)A \big) (q^a_i +q^b_i) \to J(x) A$$
and 
$$ B J\Big((p^a_i +p^b_i) 
x_m  (p^a_i +p^b_i) \Big) 
 = (q^a_i +q^b_i) 
\big(B J(x_m)  \big) (q^a_i +q^b_i) \to BJ(x) $$
in the measure topology
as $m\to \infty$. 
We have $$ J\Big((p^a_i +p^b_i) 
x_m  (p^a_i +p^b_i) \Big)  
\to J(x) $$
in the measure topology
as $m\to \infty$. 
Hence, 
$$  (p^a_i +p^b_i) 
x_m  (p^a_i +p^b_i)  
\to  x  $$
in the measure topology
as $m\to \infty$. 
This shows that $$x =  (p^a_i +p^b_i) 
x'  (p^a_i +p^b_i)  .$$ 
Note that this equality holds for any sufficiently large $i$ and 
$p^a_i +p^b_i\to {\bf1} $ as $i \to \infty $.
We obtain that for any $\tau$-finite projection $f\in \cM_2$, 
\begin{eqnarray*}& & \mu\left(4t;f(x'-x)f \right)\\
&=  &
\mu\left(4t; f  (x'- (p^a_i +p^b_i) 
x'  (p^a_i +p^b_i) )  f\right)\\
  &\stackrel{\tiny\mbox{\cite[Proposition 3.2.7]{DPS}}}{\le}& \mu\left(2t; f  (x'- (p^a_i +p^b_i) 
x' )  f\right)  +
\mu\left(2t;      
x'  ({\bf 1}- p^a_i -p^b_i) )  f\right)  \\
  &\stackrel{\tiny\mbox{\cite[Proposition 3.2.7]{DPS}}}{\le}& 
\mu\left(t;   f({\bf 1}- p^a_i -p^b_i) )   \right)\mu\left(t;x'\right)
   +
\mu\left(t;  x'\right) \mu\left(  t;  
  ({\bf 1}- p^a_i -p^b_i) )  f\right)  \\
&\stackrel{\tiny\mbox{\cite[Proposition 2.6.4]{DPS}}}{\longrightarrow}& 0 
\end{eqnarray*}
as $i \to \infty$, 
i.e.,  $x=x'$. 
Therefore, $$x\in
S\left(
(\cN_1)_{p^a_n+p^b_n} ,\tau_1\right) \cap E(\cM_1,\tau_1)=E\left(
(\cN_1)_{p^a_n+p^b_n} ,\tau_1
\right).$$
 This proves the claim for elements in 
$F\left(
(\cN_2)_{q^a_n+q^b_n} ,\tau_2\right) $. 
For an arbitrary element  $y\in F(\cN_2,\tau_2)$, 
there exists a sequence of elements $y_n\in F\left(
(\cN_2)_{q^a_n+q^b_n} ,\tau_2\right)$  such that
$$\mbox{ $\norm{y_n-y}_F\to 0$ as $n\to \infty $.}$$
There exists $x_n\in E(\cN_1,\tau_1)$ such that $T(x_n)=y_n$.
In particular, $$\norm{x_m-x_n}_E =\norm{y_m-y_n}_F$$ for any
$m,n>0$.
Since $E(\cN_1,\tau_1)$ is a Banach space (see e.g. \cite{Kalton_S} or \cite{LSZ}), it follows that there exists $x$ such that $x_n\to x$ in $E(\cN_1,\tau_1)$ as $n\to \infty$. In particular, we have $T(x)=y$.
In particular, $T$ is a surjective isometry from $E(\cN_1,\tau_1)$ onto $F(\cN_2,\tau_2)$.

Note that $\tau_1|_{\cN_1}$  is a tracial state  and
$\tau_2|_{\cN_2}$  is infinite. There exists a trace-measure
preserving $*$-isomorphism $I_1$ between $S(\cN_1,\tau_1)$ and $S(0,1)$ with the Lebesgue measure and 
there exists  a trace-measure preserving $*$-isomorphism $I_2$ between $S(\cN_2,\tau_2)$ and $S_\infty(0,\infty)$ with the Lebesgue measure~\cite{CKS}.

%Note that $I_1$ and $I_2$ are continuous with respect to measure topology. By [DPS Proposition 2.9.2], $I_1$ can be uniquely extended to a *-isomorphism from $S(\cN_1,\tau_1)$ onto measurable function space $S[0,1]$. $I_2$ can be uniquely extended to a *-isomorphism from $S(\cN_2,\tau_2)$ onto measurable function space $S[0,\infty]$.

We obtain that $I_2\circ T\circ I_1^{-1} $ is a surjective isometry from $E(0,1)$ to $F(0,\infty)$, which contradicts  Theorem \ref{Mityagin}.
%Indeed, if $E(\cM_1,\tau_1)$ is isometric to $F(\cM_2,\tau_2)$, we can construct a surjective isometry from $E(0,1)$ to $F(0,\infty)$. However, by Theorem 0.3, $E(0,1)$ is not isometric to any symmetric function space $F(0,\infty )$. This yields a contradiction. Therefore, $E(\cM_1,\tau_1)$ is not isometric to symmetric operator space $F(\cM_2,\tau_2)$.
  \end{proof}

Below, we establish a non-separable version of the above theorem.

 \begin{theorem}\label{th:F:M}
    If $E(0,1)$ is a   symmetric function space on $(0,1)$
having the Fatou property
 and $E(0,1)\ne L_p(0,1)$ (as sets), $1\le p\le \infty $.
Let $\cM_1$ be an atomless  von Neumann algebra equipped with a  faithful normal tracial state $\tau_1$.
Then, $E(\cM_1,\tau_1)$ is not isometric to any noncommutative  symmetric   space $F(\cM_2,\tau_2)$ associated with an arbitary symmetric function space $F(0,\infty)$ and    any atomless $\sigma$-finite   von Neumann algebra $\cM_2$ equipped with a semifinite infinite faithful normal trace $\tau_2$.
  \end{theorem}

  \begin{proof}

Let $T$ be a surjective isometry from $E(\cM_1,\tau_1)$ onto $F(\cM_2,\tau_2)$. 
Since   $E(\cM_1,\tau_1)$ has the Fatou property, it follows that $F(\cM_2,\tau_2)$ 
 also has the Fatou property \cite[Theorem 4.1]{DDST}.
By Theorem \ref{th:iso}, we have
\begin{align}\label{TformFatou}
T(x)=J(x)Az+BJ(x)(\mathbf{1}-z),\quad x\in \cM_1,
\end{align}
where $A,B\in F(\cM_2,\tau_2)$, $z\in P(Z(\cM_2 ))$ and $J:\cM_1 \to \cM_2$ is a Jordan $*$-isomorphism.

Without losing generality, we may assume that $A=Az, B=B(\mathbf{1}-z)$
and $A,B$ are positive (see the proof of Theorem \ref{th5}).

%In particular, we have $$A,B\subset F(\cM_2,\tau_2)\subset S_0(\cM_2,\tau_2).$$ 
%Then, for any $x\in E(\cM_1,\tau_1)$, there exists  a sequence $\{x_n\}_{n\ge 1}\subset\cM_1$ such that $\left\|x-x_n\right\|_E \to 0$. Note that
%\[ T(x_{n})=J(x_{n})Az+BJ(x_{n})(\mathbf{1}-z). \]
%Since $T$ is continuous with respect to $\left\|\cdot\right\|_E$, it follows that $T(x_n)\to T(x)$ in $E(\cM_1,\tau_1)$ (in particular,  is  the measure topology).
% By the measure-continuity of  $J$, we have $  J(x_n)\to J(x). $
%Therefore,
%\begin{align}\label{Txform}
%T(x)=J(x)Az+BJ(x)(\mathbf{1}-z),\quad \forall  x\in E(\cM_1,\tau_1).
% \end{align} 

By \cite[Lemma 1.3]{CKS}, there exists    an atomless abelian von Neumann subalgebra  $\cN_2$  of $\cM_2$ containing all spectral projections  of $A$ and $B$,
such that $ \cN_2 $  is (trace-measure preserving) $*$-isomorphic to $L_\infty (0,\infty )$. 
Note that $J^{-1}$ is a Jordan $*$-isomorphism from $\cM_2$ onto $\cM_1$. Denote
\[
\cN_1:=J^{-1}(\cN_2),
\]
which is an atomless abelian von Neumann subalgebra of $\cM_1$.
Since both $\cN_1$ and $\cN_2$ are abelian, it follows that $J:\cN_1\to \cN_2 $ is a *-isomorphism.
%By the uniqueness of the extension of $J\big|_{\cN_1}$ (see Lemma  \ref{lga}),  the restriction
%$$J\big|_{S(\cN_1,\tau_1)}:S(\cN_1,\tau_1)\to S(\cN_2,\tau_2)$$ is a *-isomorphism.

Recall that $\cM_2$ is assumed to be $\sigma$-finite. 
Let $\{q_n^a\}_{n\ge 1}$, $\{q_n^b\}_{n\ge 1} \subset P(\cN_2)$ such that $q_n^a\uparrow z$ and $q_n^b\uparrow {\bf 1}-z$ with $\tau_2(q_n^a),\tau_2(q_n^b)<\infty $
and $q_n^a A+q_n^b B$ is bounded and invertible in $\left(\cN_2\right)_{q_n^a+q_n^b} $.
Define 
$$p^a_n:=  J^{-1}(q^a_n) \mbox{ and }p^b_n:=  J^{-1}(q^b_n).$$ 
By \eqref{TformFatou}, we obtain that $T$ maps $(\cN_1)_{p^a_n}$ into $F \left(
 (\cN_2)_{q^a_n},\tau_2\right)$ (respectively,
maps $(\cN_1)_{p^b_n}$ into $F \left(
 (\cN_2)_{q^b_n},\tau_2 \right)$).

We claim that $T$ is   a surjective isometry from
 $E\left((\cN_1)_{p^a_n+p^b_n} ,\tau_1\right)$ onto $F
\left( (\cN_2)_{q^a_n+q^b_n},\tau_2
\right)$ for all $n\ge 1$. 
Indeed, 
since 
 $ (\cN_1)_{p^a_n +p^b_n}  $ and $( \cN_2)_{q^a_n+q^b_n} $ 
are equipped with finite faithful normal traces, it follows from 
 Corollary \ref{lemma measure-measure}
that $J|_{(\cN_1)_{p^a_n+p^b_n}} $ is continuous with respect to the measure topology.
By Proposition~\ref{lga}, 
$J|_{(\cN_1)_{p^a_n+p^b_n}} $ extends to a Jordan $*$-isomorphism from
 $S\left((\cN_1)_{p^a_n+p^b_n} ,\tau_1\right) $ into  $S\left( ( \cN_2)_{q^a_n+q^b_n},\tau_2\right)  $. 
For   any  positive element $  y\in F\left( (\cN_2)_{q^a_n+q^b_n } ,\tau_2\right)$,  we have $$y\left(
 ( A+B )  ( q^a_n+q ^b_n)\right)^{-1}\in  S
\left(
 (\cN_2)_{q^a_n+q^b_n} ,\tau_2
\right) $$
and
   there exists
 $x:=J^{-1}\left(
y\left(
 ( A+B )  ( q^a_n+q^b_n)\right)^{-1} \right)\in S\left(  (\cN_1)_{q^a_n+q^b_n},\tau_ 1\right)$
(note that 
 $x\ge 0$)
 such that
%\[
%J(x)A=y(A+B)^{-1}A=yz,
%\]
%\[
%BJ(x)=By(A+B)^{-1}=y(\mathbf{1}-z).
%\]
%Therefore, we have
\[
J(x)A+BJ(x)=y.
\]
We only need to observe that $x\in E(\cM_1,\tau_1)$ and $T(x)=y'$. Indeed, 
since $T$ is a surjective isometry from
 $E(\cM_1,\tau_1)$ onto $F(\cM_2,\tau_2)$, it follows  
  that there exists $x'\in E(  \cM_1 ,\tau_1)$ such that
\[
T(x')= y=J(x)A+BJ(x).
\]

%The same argument used in Corollary \ref{order continuous} shows that 
%\begin{align}\label{Txformreduced}
%T(w)= J(w) A z+ BJ(w)  ({\bf 1}-z)
%\end{align}
%for all $w \in E\left(
%(\cN_1)_{p^a_n+p^b_n} ,\tau_1\right) $, $n\ge 1$. 
Let $x_m : =x 'e^{|x'| } (0,m  )  
 \in     \cM_1 $, $m \ge 1$. 
In particular,   $$|x_m|   \uparrow |x'|$$ and $x_m\to x'$ 
in the $\sigma(E,E ^\times)$-topology as $m\to \infty$. 
We have $$ T(x_m) \to T(x') $$
in the $\sigma(F,F^\times)$-topology
as $m \to \infty$  (see Theorem \ref{weak-continous of T}). 
Therefore, $$T(x_m) = J(x_m)A +BJ(x_m) \to J(x)A+BJ(x)$$
in the $\sigma(F,F^\times)$-topology as $m\to \infty$. 
In particular,  we have 
$$ J\Big((p^a_n +p^b_n) 
x_m  (p^a_n +p^b_n) \Big) A
 = (q^a_n +q^b_n) 
\big(J(x_m)A \big) (q^a_n +q^b_n) \to J(x) A$$
and 
$$ B J\Big((p^a_n +p^b_n) 
x_m  (p^a_n +p^b_n) \Big) 
 = (q^a_n +q^b_n) 
\big(B J(x_m)  \big) (q^a_n +q^b_n) \to BJ(x) $$
in the $\sigma(F,F^\times)$-topology
as $m\to \infty$. 
We have $$ J\Big((p^a_n +p^b_n) 
x_m  (p^a_n +p^b_n) \Big)  
\to J(x) $$
in the $\sigma(F,F^\times)$-topology as $m\to \infty$. 
Hence, 
 $$ J\Big((p^a_n +p^b_n) 
x '   (p^a_n +p^b_n) \Big)  = J(x) $$
%in the $\sigma(F,F^\times)$-topology
%as $m\to \infty$. 
This shows that $$x =  (p^a_n +p^b_n) 
x'  (p^a_n +p^b_n)  .$$ 
Note that this equality holds for any sufficiently large $n$ and 
$p^a_n +p^b_n\to {\bf1} $ as $n\to \infty $.
We obtain that $x=x'$. 
Therefore, $$x\in
S\left(
(\cN_1)_{p^a_n+p^b_n} ,\tau_1\right) \cap E(\cM_1,\tau_1)=E\left(
(\cN_1)_{p^a_n+p^b_n} ,\tau_1
\right).$$
 This proves the claim for elements in 
$F\left(
(\cN_2)_{q^a_n+q^b_n} ,\tau_2\right) $. 
For general element  $y\in F(\cN_2,\tau_2)$, 
there exists a sequence of elements $y_n\in F\left(
(\cN_2)_{q^a_n+q^b_n} ,\tau_2\right)$  such that
$$ y_n\to y $$
in the $\sigma(F,F^\times)$-topology as $n\to \infty $. 
There exists $x_n\in E(\cN_1,\tau_1)$ such that $T(x_n)=y_n$.
By the $\sigma(F,F^\times)-\sigma(E,E^\times)$ continuity of $T^{-1}$ (see Theorem \ref{weak-continous of T}), 
the sequence 
$$ \{x_n\} _{n\ge 1}  $$
converges
in the $\sigma(E,E^\times)$-topology as $n\to \infty $. 
  By \cite[Proposition 3.1]{DK},
 we know that $E(\cN_1,\tau_1)$ is $\sigma(E,E^\times)$-sequentially complete.
Hence, there exists $x\in E(\cN_1,\tau_1)$ such that $$x_n\to x$$
 in the  $\sigma(E,E^\times)$-topology as $n\to \infty $. 
 In particular, we have $T(x)=y$.

Note that $\tau_1|_{\cN_1}$  is a tracial state  and
$\tau_2|_{\cN_2}$  is infinite. There exists a trace-measure
preserving $*$-isomorphism $I_1$ between $S(\cN_1,\tau_1)$ and $S(0,1)$ with the Lebesgue measure and 
there exists a trace-measure preserving $*$-isomorphism $I_2$ between $\cN_2$ and $L_\infty(0,\infty)$ with the Lebesgue measure~\cite{CKS}.
%Similarly, there exist a trace-measure preserving *-isomorphism $I_2$ between $\cN_2$ and $L_\infty(0,\infty)$ with the Lebesgue measure.

%Note that $I_1$ and $I_2$ are continuous with respect to measure topology. By [DPS Proposition 2.9.2], $I_1$ can be uniquely extended to a *-isomorphism from $S(\cN_1,\tau_1)$ onto measurable function space $S[0,1]$. $I_2$ can be uniquely extended to a *-isomorphism from $S(\cN_2,\tau_2)$ onto measurable function space $S[0,\infty]$.

We obtain that $I_2\circ T\circ I_1^{-1} $ is a surjective isometry from $E(0,1)$ to $F(0,\infty)$, which   contradicts Theorem \ref{Mityagin}.
%Indeed, if $E(\cM_1,\tau_1)$ is isometric to $F(\cM_2,\tau_2)$, we can construct a surjective isometry from $E(0,1)$ to $F(0,\infty)$. However, by Theorem 0.3, $E(0,1)$ is not isometric to any symmetric function space $F(0,\infty )$. This yields a contradiction. Therefore, $E(\cM_1,\tau_1)$ is not isometric to symmetric operator space $F(\cM_2,\tau_2)$.
  \end{proof}

\chapter{Description of surjective  isometries on noncommutative symmetric spaces over a von Neumann algebra equipped with a semifinite infinite faithful normal trace}\label{Sec:inf}
The main result of this chapter  extends Theorem \ref{KRZ1} to the infinite setting.
Note that the condition that $E(0,\infty )\ne L_p(0,\infty )$  up to equivalent norms    does not guarantee that an isometry on $E(0,\infty)$ is generated by a unimodular measurable function and an invertible Borel measure-preserving map,
%\begin{align}\label{qwer}
%a(s)x(\sigma(s)) \quad a.e.,~\forall  x\in E(0,1),
%\end{align}
%where $a:(0,1)\to \mathbb{C}$ is a non-vanishing Borel function such that $|a| = 1$ a.e., and $\sigma : (0, 1) \to (0, 1)$ is an invertible Borel measure-preserving map,
 e.g., isometries on $L_{p,q}(0,\infty)$ can be generated by   dilation operators  $D_t$, $t>0$.
Indeed,
$$t^{- 1/p} \norm{D_t f}_{L_{p,q}(0,\infty )}=\norm{f}_{L_{p,q}(0,\infty )},~\forall f\in L_{p,q}(0,\infty). $$

%The main result of this section  extends Theorem \ref{KRZ1} to the infinite setting.
%Note that the condition that $E(0,a)\ne L_p(0,a)$, $\forall a>0,$ up to equivalent norms in  Theorem \ref{KRZ1}  does not guarantee that an isometry on $E(0,\infty)$ is of the form
%\begin{align}\label{qwer}
%a(s)x(\sigma(s)) \quad a.e.,
%\end{align}
%for any $x\in E(0,1)$ where $a:[0,1]\to \mathbb{C}$ is a non-vanishing Borel function such that $|a| = 1$ a.e., and $\sigma : [0, 1] \to [0, 1]$ is an invertible Borel measure-preserving map, e.g., isometries on $L_{p,q}(0,\infty)$ can be generated by the dilation operator $\sigma_t$, $t>0$.

%{\color{red}
%Given a symmetric function space $E(0,\infty)$, for every $0<s<\infty$, we define the dilation operator $D_s$ on $E(0,\infty)$ by
%$$(D_s f)(t)=f(t/ s),\quad \mbox{for any}\quad f\in E(0,\infty).$$
%}

\begin{rem}
Let $E(0,\infty)$ be a symmetric function space.
Note that  $$f\in E(0,a) \mbox{~for any $a>0$ if and only if } D_{1/ a}f \in E(0,1)$$
and
 $$ f\in L_p(0,a) \mbox{~for any $a>0$ if and only if }  D_{1/ a} f \in L_p(0,1).$$
Therefore, if the  symmetric function space $E(0,\infty)$ satisfies that
 $E(0,1)\ne L_p(0,1)$, then
 $E(\supp\{f\})\ne L_p(\supp\{f\})$ for any non-trivial $f\in E(0,\infty )$ with finite support.
\end{rem}

The following auxiliary tool is a corollary of Theorem \ref{Mityagin}.
\begin{lemma}\label{finite support}
Let $E(0,\infty)$ and $F(0,\infty)$ be two (complex) symmetric spaces.
Assume that $E(0,1)\ne L_p(0,1)$ (as sets) for all $1\le p\le \infty$. Let $T:E(0,\infty)\to F(0,\infty)$ be a surjective isometry. Then for any $0\ne f\in E(0,\infty)$ with $m( \mbox{supp}\{f\})<\infty$,
we have  $$m( \mbox{supp} \{Tf\})<\infty.$$
\end{lemma}

\begin{proof}
Suppose that $0\ne f\in E(0,\infty)$, with $m(\mbox{supp} \{f\})<\infty$. By  Zaidenberg's
result~\cite[Theorem 1]{Zaidenberg}, 
we have 
$$T(x)(t)=a(t)x(\sigma(t)),\quad \forall x\in E(0,\infty),\quad t\in(0,\infty),$$
where $a\in F(0,\infty)$ and $\sigma$ is an invertible measurable transformation from $(0,\infty)$
onto $(0,\infty)$. Then,
 $T$ is a surjective isometry from $E(\mbox{supp} \{f\})$ onto $F(\sigma^{-1}(\mbox{supp} \{f\}))$. 
By Theorem \ref{Mityagin}, we obtain that
 $$m\left(\sigma^{-1}(\supp \{f\}) \right)<\infty.$$ This complete the proof.
\end{proof}

\begin{theorem}\label{infinite}%\footnote{The same argument applies to surjective isometries on real symmetric spaces on $(0,\infty)$ in the sense of Lindenstrauss and Tzafriri. }
Assume that $E(0, \infty)$ and $F(0,\infty)$ are   complex symmetric function spaces with
$\norm{\chi_{_{(0,1)}}}_E=\norm{\chi_{_{(0,1)}}}_F=1$.
Let $T:E(0,\infty)\to F(0,\infty )$ be a surjecitve isometry.
Assume that $E(0,1)\ne  L_p(0,1)$, $1\le p\le \infty$ (as sets).
Then,  we have \begin{align}\label{supinf}
\sup_{0\ne f\in E,m(\supp\{f\})<\infty}\frac{\norm{D_t f}_E}{\norm{  f}_F }=\inf_{0\ne f\in E,m(\supp\{f\})<\infty}\frac{\norm{ D_t f}_E}{\norm{   f}_F }
\end{align}
 for some $t\in(0,\infty)$, and  there exists $t= \alpha>0$ satisfying \eqref{supinf} and
$$
      (Tf)(s)  = a (s)\cdot f\left(
j (
\alpha s )
 \right), \quad \forall  f\in E(0,\infty), $$
 where
$j:(0,\infty)\to (0,\infty)$ is an invertible measure-preserving Borel transformation on $(0,\infty)$ and $a$ is a non-vanishing Borel  complex  function on $(0,\infty )$ such that  $|a|$ is a constant a.e..

If, in addition, that
\eqref{supinf} does not hold for any $0<t\neq1$,
 then
$$
Tf=af\circ \sigma , ~\forall  f\in E(0,\infty), $$ where  $a$ is a non-vanishing Borel   complex  function on $(0,\infty )$  such that  $|a| = 1$ a.e., and $\sigma : (0, 1) \to (0, 1)$ is an invertible
    Borel measure-preserving map.

\end{theorem}
\begin{proof}
Assume that $E(0,1)\ne L_p (0,1)$ up to an equivalent norm.
Let $T$ be a surjective isometry from $E(0,\infty)$ onto $F(0,\infty)$. By \cite[Theorem 1]{Z77},
we have
\begin{align}\label{123}
T(x)(t)=a(t)x(\sigma(t))\end{align}
where $a\in F(0,\infty)$ and $\sigma$ is an invertible measurable transformation from $(0,\infty)$ onto $(0,\infty)$. In particular, $T$ is disjointness-preserving.

Let $s$ be a positive integer. Denote $$I_s=\sigma(0,s).$$ 
By Lemma \ref{finite support}, we have 
$m(I_s)<\infty$. Since $T$ is an isometry, it follows that $m(I_s)>0$.
Therefore, $T$ is a surjective isometry from $E(I_s)$ onto $F(0,s)$.

Let $$
i_s:I_s\to (0,m(I_s)) 
$$
be an invertible Borel measure-preserving map such that $i_s = i_t$ on $I_s$ whenever $t>s$.

Define $T_0f=f\circ i_s$.
Then, $T_0$ is a surjective isometry from $E(0,m(I_s))$ onto $E(I_s)$. Define
\begin{align}\label{define E' again}
&~ \quad  E'(0,s) \nonumber \\
&:= \left\{f\in S(0,s)\Big| T_0\circ D_{\frac{m(I_s)}{s}}f\in E(I_s)
~ \mbox{and}~ \norm{f}_{E'}=
\frac{1}{\norm{\chi_{_{I_s}}}_E}\norm{D_{\frac{m(I_s)}{s}}f}_E\right\}  \end{align}
and
\begin{align}\label{define F' again}
F'(0,s):=
  F(0,s)  \mbox{ with } \norm{f}_{F'}=\frac{1}{\norm{\chi_{_{(0,s)}}}_F}\norm{f}_F ,~\forall
f\in F(0,s)  .
\end{align}
Clearly $E'(0,s)$ and $F'(0,s)$ are symmetric function spaces.

Since for any $f\in E'(0,s)$, we have
\begin{align*}
\norm{T\circ T_0\circ D_{\frac{m(I_s)}{s}}(f)}_{F'}&=\frac{1}{\norm{\chi_{(0,s)}}_F}\norm{T\circ T_0\circ D_{\frac{m(I_s)}{s}}(f)}_{F}\\
&=\frac{1}{\norm{\chi_{_{(0,s)}}}_F}\norm{ T_0 \left(
D_{\frac{m(I_s)}{s}}(f) \right)}_{E}\\
&=\frac{1}{\norm{\chi_{_{(0,s)}}}_F}\norm{D_{\frac{m(I_s)}{s}}(f)}_{E}\\
&=\frac{\norm{\chi_{_{I_s}}}_F}{\norm{\chi_{_{(0,s)}}}_F}\norm{ f}_{E'},
\end{align*}
it follows that the mapping

$$T':=\frac{\norm{\chi_{_{(0,s)}}}_F}{\norm{\chi_{_{I_s}}}_E}T\circ T_0\circ D_{\frac{m(I_s)}{s}}(f)$$
is a surjective isometry form $E'(0,s)$ to $F'(0,s)$.

By Theorem \ref{KRZ1}, we have 
$$T'f=\frac{\norm{\chi_{_{(0,s)}}}_F}{\norm{\chi_{_{I_s}}}_E}T\circ T_0\circ D_{\frac{m(I_s)}{s}}(f)=c_sf\circ \sigma_s,\quad f\in E'(0,s),$$
where $c_s:(0,s)\to \mathbb{C}$  is a non-vanishing Borel function with $|c_s|$ = 1 and  $\sigma_s$ is an invertible measure-preserving Borel transformation from $(0,s)$ onto $(0,s)$, i.e.,
\begin{align}\label{sigma}
\frac{\norm{\chi_{_{(0,s)}}}_F}{\norm{\chi_{_{I_s}}}_E}a\cdot \left(
D_{\frac{m(I_s)}{s}}f
\right)\circ i_s\circ\sigma\stackrel{\eqref{123}}{=}c_sf\circ \sigma_s,\quad \forall f\in E'(0,s).
\end{align}
Taking $f=\chi_{(0,s)}$, we have
\begin{align}\label{c1}
T(\chi_{_{I_s}})=\frac{\norm{\chi_{_{I_s}}}_E}{\norm{\chi_{_{(0,s)}}}_F}c_s\chi_{_{(0,s)}}.
\end{align}
Let $t\ge s$.
We have
\begin{align*}
T(\chi_{_{I_t}})=\frac{\norm{\chi_{_{I_t}}}_E}{\norm{\chi_{_{(0,t)}}}_F}c_t\chi_{_{(0,t)}}.
\end{align*}
By the disjointness-preserving property of $T$, we obtain that $T(\chi_{_{I_t/ I_s}})$ and $T(\chi_{_{I_s}})$
have mutually disjoint supports. Hence, $T(\chi_{_{I_t/I_s}})$ is supported on $(s,t)$. We
have
\begin{align*}
\frac{\norm{\chi_{_{I_t}}}_E}{\norm{\chi_{_{(0,t)}}}_F}c_t\chi_{_{(0,s)}}=T(\chi_{_{I_t}})\chi_{(0,s)}&=T(\chi_{_{I_s}})\chi_{_{(0,s)}}+T(\chi_{_{I_t/I_s}})\chi_{_{(0,s)}}\\
&=T(\chi_{_{I_s}})\chi_{_{(0,s)}}=
\frac{\norm{\chi_{_{I_s}}}_E}{\norm{\chi_{_{(0,s)}}}_F}c_s \chi_{_{(0,s)}}.
\end{align*}
Hence, for any $t\ge s$, we have 
$$\frac{\norm{\chi_{_{I_s}}}_E}{\norm{\chi_{_{(0,s)}}}_F}c_s=\frac{\norm{\chi_{_{I_t}}}_E}{\norm{\chi_{_{(0,t)}}}_F}c_t \quad \mbox{on}\quad (0,s).$$
This implies that $\frac{\norm{\chi_{_{I_s}}}_E}{\norm{\chi_{_{(0,s)}}}_F}$ is a constant independent of $s$ and $c_t\mid_{(0,s)}=c_s$. By (\ref{123}), we have
\begin{align}\label{a1}
T\left(\chi_{_{I_s}}\right)=a\cdot \chi_{_{(0,s)}}.\end{align}
Comparing  (\ref{c1}) and (\ref {a1}), we obtain
\begin{align}\label{d1}
\frac{\norm{\chi_{_{I_s}}}_E}{\norm{\chi_{_{(0,s)}}}_F}c_s=a  ,~a.e. \quad \mbox{on}\quad (0,s).
\end{align}
This implies that
 $|a|$ is a constant on $(0,\infty)$, and, by  \eqref{sigma}, we have
\begin{align}\label{opa}
 \left(D_{\frac{m(I_s)}{s}}f\right)\circ i_s\circ\sigma=f\circ \sigma_s, \quad \forall  f\in E'(0,s).
\end{align}
Let $g=\left(D_{\frac{m(I_s)}{s}}f\right) \circ i_s \in E(I_s)$. In particular,  $f=D_{\frac{s}{m(I_s)}}(g\circ i_s^{-1})$.
Thus, by \eqref{opa}, we have
\begin{align}
\label{a3}g\circ \sigma=D_{\frac{s}{m(I_s)}}(g\circ i_s^{-1})\circ \sigma_s
= (g\circ i_s^{-1 })\left ( \frac{m(I_s)}{s } \cdot \sigma_s(\cdot ) \right)
,\quad g\in E(I_s).
\end{align}
Since $I_s\subset I_t$ when $s\le t$% and  $i_t^{-1}\mid_{(0,m(I_s))}=i_s^{-1}$
, it follows
from  \eqref{a3} (by letting $g=\chi_{_{I_s }}$)  that
\begin{align}\label{uij}\sigma (\cdot )=i_s^{-1}\left(\frac{m(I_s)}{s}\sigma_s(\cdot ) \right)=i_t^{-1}\left(\frac{m(I_t)}{t}\sigma_t(\cdot )\right )\mbox{ on $(0,s)$}.\end{align}
 Since $i_s$ and $\sigma_s$ are  measure-preserving, it follows that,
for any $J_s\subset (0,s)$,
 we have
$$m(\sigma(J_s))=m\left(
i_s^{-1}(\frac{m(I_s)}{s}\sigma_s(J_s))\right)=
\frac{m(I_s)}{s}m\left(\sigma_s(J_s)\right)=\frac{m(I_s)}{s}m(J_s).$$
Therefore, for any $t\ge s$, we have 
$$m(\sigma(J_s))
%=m\left(i_s^{-1} \left(\frac{m(I_s)}{s}\sigma_s(J_s) \right)\right)=m\left( i_t^{-1}
%\left(\frac{m(I_t)}{t}\sigma_t(J_s)
%\right)\right)
=\frac{m(I_t)}{t}m(J_s).$$
This implies that $\frac{m(I_s)}{s}$ is a constant for any positive integer $s $.
 Let $\alpha=\frac{m(I_s)}{s}$, $s\ne 0$.
Since $\sigma_s$ is a measure-preserving Borel transformation from $(0,s)$ onto $(0,s)$, it follows from \eqref{uij} that
$$\sigma(\cdot)=j(\alpha \cdot)$$ where $j$ is an invertible measure-preserving Borel transformation on $(0,\infty)$.
Therefore, for any $ f\in E(0,\infty)$, we have $$T(f)(t)
=
 a (t)\cdot f(\sigma (t) )= a (t)\cdot f\left(
j (
\alpha t)
 \right) , $$
where  $a:(0,\infty)\to \mathbb{C}$ is a non-vanishing Borel function such that  $|a|$ is a constant a.e..
In particular, this proves
\eqref{supinf}.
% For any $f\in E(0,\infty)$ with $m(\mbox{supp}\{f \})=\infty$, $f=f\chi_{(0,a)}+f\chi$

Next, we consider the case when
\eqref{supinf} does not hold
 for any $0<t\ne1$.
  Since $\alpha=\frac{m(I_s)}{s}$ is a constant independent of $s$, it follows that, for any $f\in E(0,s)$, 
we have
\begin{eqnarray*}
%\norm{T D_{\frac1\alpha}f}_E=\norm{T D_{\frac{m(I_s)}{s}}f}_E &~=~
 \norm{D_{\frac{m(I_s)}{s}}f}_E&  \stackrel{\eqref{define E' again}}{=} &\norm{\chi_{_{I_s}}}_E\norm{f}_{E'}\\
&  =&\norm{\chi_{_{I_s}}}_E\norm{T'f}_{F'}\\
&~ =~&\norm{\chi_{_{I_s}}}_E\norm{c_s f\circ \sigma_s}_{F'}\\
&\stackrel{\eqref{define F' again}}{=}&\frac{\norm{\chi_{_{I_s}}}_E}{\norm{\chi_{_{(0,s)}}}_F}\norm{c_s f\circ \sigma_s}_{F}\\
&~ =~&\frac{\norm{\chi_{_{I_s}}}_E}{\norm{\chi_{_{(0,s)}}}_F}\norm{f}_{F}.
\end{eqnarray*}
Then,
 we have
$$\inf_{f\in E, f\ne 0, m(\supp\{f\})<\infty}\frac{\norm{D_\alpha f}_E }{\norm{f}_F} =\sup_{f\in E, f\ne 0,m(\supp\{f\})<\infty}\frac{\norm{D_\alpha f}_E }{\norm{f}_F}.$$
Therefore, $\alpha=1$, which implies that  $s=m(I_s)$ for any $s>0$. This implies that $\sigma$ is measure-preserving. Hence, for any $s>0$, $$|a|=\frac{\norm{\chi_{_{I_s}}}_E}{\norm{\chi_{_{(0,s)}}}_F}=\frac{\norm{\chi_{_{I_1}}}_E}{\norm{\chi_{_{(0,1)}}}_F}=\frac{\norm{\chi_{_{(0,1)}}}_E}{\norm{\chi_{_{(0,1)}}}_F}=1\quad \mbox{on}\quad (0,s).$$
This completes the proof.
\end{proof}

\begin{remark}Let $1\le p\le q<\infty$.
Note that $(L_p+L_q)(0,\infty)$ coincides with $L_p$ (as sets) on any finite interval.
It seems that the
description of surjective isometries on $ L_p+L_q $ remains unknown even in the commutative setting.
\end{remark}

\begin{theorem}\label{th:infiniteKRZ}
Let $E(0,\infty), F(0,\infty)$ be  symmetric function spaces  with  $\norm{\chi_{_{(0,1)}}}_E=\norm{\chi_{_{(0,1)}}}_F=1$.

Let $\cM_1$ and $ \cM_2 $ be   atomless   von Neumann algebras  equipped with  semifinite  faithful  normal traces  $\tau_1$ and $\tau_2$, respectively.

Assume that
\begin{enumerate}\item  $\cM_1$ and $\cM_2$ are $\sigma$-finite, and $E(0,\infty)$ and $F(0,\infty)$ have the Fatou property;
\item or  $E(0,\infty)$ and $F(0,\infty)$ are minimal.
\end{enumerate}
Assume that $E(0,1)\ne  L_p(0,1)$, $1\le p\le \infty$ (as sets).
Let $T:E(\cM_1,\tau_1)\rightarrow F(\cM_2,\tau_2)$ be a surjecitve isometry.
Then,  we have
\begin{align}\label{nonsupinf}
\sup_{0\ne f\in E(0,\infty),m(\supp\{f\})<\infty}\frac{\norm{D_t f}_E}{\norm{f}_F }=\inf_{0\ne f\in E(0,\infty),m(\supp\{f\})<\infty}\frac{\norm{D_t f}_E}{\norm{f}_F }
\end{align}
 for some $t\in(0,\infty)$,   and there exists a positive number $t=\alpha$ satisfying \eqref{nonsupinf} and
\begin{align*}
T(x)=au\cdot J(x) , ~\forall  x\in E(\cM_1,\tau_1), \end{align*} where $a>0$ is a constant, $u$ is a unitary operator in $\cM_2$ and $J:S(\cM_1,\tau_1)\to S(\cM_2,\tau_2)$ is a   Jordan $*$-isomorphism such that 
$ \alpha \cdot   \tau_2(J(x))=\tau_1(x)$, $\forall x\in \cM_1$.

If, in addition, that
\eqref{nonsupinf} does not hold for any $0<t\neq1$,
 then
 \begin{align}\label{dfg}
 T(x)=u\cdot J(x), ~\forall  x\in E(\cM_1,\tau_1 ), \end{align} where $u$ is  a unitary operator in $\cM_2$ and $J:S(\cM_1,\tau_1)\to S(\cM_2,\tau_2)$ is a trace-preserving Jordan $*$-isomorphism.
\end{theorem}

\begin{proof}
Without loss of generality, we may  assume that $\tau_1$ and $\tau_2$ are infinite, see Chapters \ref{S:KRZ} and \ref{S:M}.

By Theorem \ref{th:iso} and \cite[Theorem 4.4]{HS}, $T$ is of the form:
\begin{align}\label{kl}
T(x) =\sigma(F,F^\times)-\sum_{i }     J(x)A_i z + B_i J(x) ({\bf 1}-z),~x\in E(\cM_1,\tau_1)\cap \cM_1, 
\end{align}
where $A_i \in F(\cM_2,\tau_2)$ are   disjointly supported from the left, and $B_i \in F(\cM_2,\tau_2)$ are  disjointly supported from the right,  $J:\cM_1\to \cM_2$ is a surjective Jordan $*$-isomorphism   and $z$ is  a central projection in
$  \cM_2$.

We only prove the case for symetric spaces with the Fatou property. 
The minimal case follows from the same argument.
For simplicity, we may assume that all $A_i =0$. That is, $z=0$.  

We claim the there exists 
   an increasing net  $\{e_i\}_{i\in I}$  of $\tau$-finite projections in $\cM_1$ such that $$\mbox{$e_i\uparrow {\bf 1}$
and $\tau_2(J(e_i))<\infty $ for all $i$.}$$
Indeed, the Zorn lemma implies that there exists a maximal family of non-zero pairwise
disjoint $\tau$-finite projections $\{e_i\}_{i\in I}$ in $\cM_1$ such that 
$\tau_2(J(e_i))<\infty $ for all $i$.
Let $p=\vee _ {i\in I} e_i$. Let $q\le {\bf 1} -J(p)$ be a $\tau$-finite projection in $\cM_2$. 
For any   $\tau$-finite projection $f\in \cM_1$ with  
$f\le J^{-1}(q)$, we have $\tau_2(J(f))\le \tau_2 (q)<\infty$. By the maximality of $\{e_i\}_{i\in I}$, we obtain that $f=0$ and $q=0$.
Hence, $p={\bf 1}$. This proves the claim.

Since $\cM_1$ is $\sigma$-finite (or $E(\cM_1,\tau_1)$ is minimal), it follows that $\{e_i\}_{i\in I}$ is at most countable, denoted by $\{e_n\}_{n\ge 1}$.

For any $x\in e_n \cM_1  e_n$, we have  (see the proof of  Theorem \ref{th:iso} and \cite[Theorem 4.4]{HS})
$$
T(x)  =  T(e_n)J(x)  .
$$
The arguments in Theorems \ref{th5} and \ref{th:F:M} show that
$$
T(x)  =  T(e_n)J(x)  , ~\forall x\in E \left(
(\cM_1)_{e_n},\tau_1
\right),
$$
where   $J: \cM_1 \to  \cM_2 $ is a  Jordan $*$-isomorphism, which can be exended to a Jordan $*$-isomorphsim 
from $S
\left(
 (\cM_1) _{e_n},\tau_1
\right)$ onto 
$S\left(
 (\cM_2) _{J(e_n)},\tau_2
\right)$
for all $n$.

Let   $v_{n}|T(e_n)^*| =T(e_n)^*$ be the polar decomposition.
The argument used in Section~\ref{subs:KRZ} shows that
$$   |T(e_n) | J(\cdot )  $$ a surjective isometry from $E(\cA_n,\tau_1)$ onto $F(\cB_n,\tau_2)$,
where $\cB_n $ (respectively, $\cA_n $) is an atomless abelian von Neumann subalgebra of 
$J(e_n)\cM_2 J(e_n)$
 (respectively, $e_n \cM_1 e_n$)
containing all spectral projections of  
$|T(e_n) |  ({\bf 1}-z)$ 
 and $J(\cA_n)=\cB_n$. Moreover, 
there  is a trace-measure preserving  $*$-isomorphic from $\cA_n $ onto $L_\infty (0,\tau_1(e_n))$
 (respectively, from   $\cB_n$ onto   $L_\infty (0,\tau_2(J(e_n)) )$).

Denote 
$$\alpha_n =m (0,\tau_1(e_n))\mbox{ and } \beta_n =m (0,\tau_2(J(e_n))). $$
 Define
\begin{align*} E'(0,\beta_n):=  \left\{f\in S(0,\beta_n)
\Big|D_{\frac{\beta_n}{\alpha_n}}f  \in E(0,\alpha_n)
~\mbox{and}~ \norm{f}_{E'}:=
\frac{1}{\norm{\chi_{_{(0,\alpha_n )}}}_E}
\norm{D_{\frac{\beta_n}{\alpha_n}}f}_E\right\}  \end{align*}
and
\begin{align*}
F'(0,\beta_n):=
  F(0,\beta_n)  \mbox{ with } \norm{f}_{F'}=\frac{1}{\norm{\chi_{_{(0,\beta_n)}}}_F}\norm{f}_F ,~\forall
f\in F(0,\beta_n )  .
\end{align*}
Clearly $E'(0,\beta_n )$ and $F'(0,\beta_n )$ are symmetric function spaces.
The same argument used in Theorem \ref{infinite}  (see \eqref{d1}) yields that
$$\mu\left(|T(e_n)| J(e_n)\right ) = c_n: = \frac{\norm{\chi_{_{(0,\alpha_n )}}}_E}{\norm{\chi_{_{(0,\beta_n )}}}_F} $$
on $(0,\beta_n )$.
Hence, there exists a partial isometry  $u_n\in  \cM_2  $
with $l(u_n)=  l(T(e_n))   $ and $   r (u_n) =r(T(e_n))=J(e_n) $ (see e.g. \cite[Remark 3.5]{HSZ})
 such that
$$T(x)= c_n   u_n  J(x)   
 ,~\forall x\in E\left(
e_n \cM_1 e_n,\tau_1\right).$$

For any $m\ge n$,   there exists a partial isometry $u_m $
with $l(u_m)= l (T(e_m)) $ and $r(u_m)= r (T(e_m)) $
such that
$$T(x)= c_m  u_mJ(x),  ~\forall x\in E(e_m \cM_1 e_m,\tau_1).$$
We have
$$c_n  u_n J(e_n) = T(e_n)= 
c_m  u_m J(e_n ) .  $$
This shows that $c_m =c_n$ for all $m\ge n$.
Denote $c:=c_n$.
We have
$$ u_m J(e_n) = u_n J(e_n)  =u_n. $$ 
Note that $l(T(e_n))\uparrow {\bf 1}$ (see e.g.  the proof in \cite[Remark 3.5]{HSZ}).
   Since $J$ is normal, it follows that   $J(e_n)\uparrow {\bf 1}$, and
 $\{ u_n\}_{n\ge 1} $ converges to an operator $  u\in \cM_2$ in the strong operator topology (see e.g.  the proof of \cite[Proposition 2.17]{HSZ}, see also \cite{Yeadon}) with $$  u_n =  l(T(e_n)) \cdot u
\cdot  J(e_n)  =u
\cdot  J(e_n) .
 $$
We have 
\begin{align}\label{finite Tx}
T(x)= c u   J(x)   
 ,~\forall x\in E\left(
e_n \cM_1 e_n,\tau_1\right) 
\end{align}
and 
$$P(\cM_2)\ni l(T(e_n ))= u_n u_n^* = u J(e_n) u^* \uparrow uu^*. $$
This implies that $u$ is a partial isometry \cite[Section 1.7]{DPS}.

\begin{comment}
We claim that $u$ is a unitary operator. 
Indeed, 
assume by contradiction that $u u ^* \ne {\bf 1}$ or $u ^*u  \ne {\bf 1}$. 
If $u ^* u  \ne {\bf 1}$, then there exists a non-zero $\tau$-finite projection $p\le {\bf 1}-u ^* u $
such that $J^{-1}(p)$ is also $\tau$-finite. 
By St{\o}rmer's theorem \cite{Stormer},
there exists $w\in Z(P(\cM_1))$ such that $J(w\cdot)$ is a $*$-homomorphism and  
$J(({\bf 1}- w) \cdot)$ is a $*$-homomorphism. 
For simplicity, we assume that 
$ w \cdot J^{-1}(p)=  J^{-1}(p)$. 
Note that for any $e_n$, we have 
 \begin{align*}
T\left(
  J^{-1}(p) e_n 
\right) &=\sigma(F, F^\times )-\lim_{m\to \infty } T \left(e_m  J^{-1}(p) e_n\right) \\
&= \sigma(F, F^\times )-\lim_{m\to \infty } c u\cdot J\left(e_m  J^{-1}(p)  e_n\right)  \\
&= \sigma(F, F^\times )-\lim_{m\to \infty } c u\cdot J(e_m)J \left(  J^{-1}(p) \right) J(  e_n)  \\
&=   c u   \cdot    p  J(  e_n)\\
&=0. 
\end{align*}
Therefore, we have  
$$T(J^{-1}(p) ) = \sigma(F, F^\times )-\lim_{n\to \infty } 
T(  J^{-1}(p) e_n ) =0,$$
which contradicts the assumption the $T$ is an isometry. 
Hence, $u^* u  = {\bf 1}$.
\end{comment}
Since
 $T$ is  $\sigma(E,E^\times)-\sigma(F,F^\times)$-continuous (see Theorem \ref{weak-continous of T}) and   $\cup_{n\ge 1}e_n \cM_1 e_n$ is dense in $E(\cM_1,\tau_1) $ in the $\sigma(E,E^\times)$-topology (see Lemma \ref{dense}), 
it follows from \eqref{finite Tx} that  $$T(x)=cuJ(x)  $$
 for all $x\in E(\cM_1,\tau_1)$.
We have 
$$ l(u)\ge  \cup _{x\in E(\cM_1,\tau_1)\cap \cM_1 }  l(T(x))  = \cup _{x\in E(\cM_1,\tau_1 )} l(T(x) )={\bf 1},$$
which
shows that $l(u)={\bf 1}$. 
On the other hand, if $r(u)\ne {\bf 1}$, then there exists a $\tau$-finite projection $p\in \cM_1$ such that $J(p)\le {\bf 1}-r(u)$. We have $$\norm{T(p)}_F =\norm{cu J(p)}_E =0,$$
which contradicts the assumption that $T$ is an isometry. 
Hence, $r(u)={\bf 1}$ and  $u$ is unitary.

Let $\cA$ be an arbitrary atomless abelian weakly closed $*$-subalgebra   of $\cM_1$, which is (trace-measure preserving) $*$-isomorphic to $L_\infty (0,\infty )$.
For any $\tau$-finite projection $p \in \cA$,
we have $c  \norm{J(p)}_F=\norm{p}_E$.
We claim that $J(p)$ is $\tau$-finite.
Indeed, if $\tau_2(J(p)) =\infty$, then there exists a partition $\{p_n\}_{n}$ of $p$ such that
 $$\mbox{ $\tau_1(p_n)\to 0$
as $n\to \infty$ but $\tau_2(J(p_n))\to \infty$~as~$n\to \infty $ .}$$
Note that $0\ne \norm{c J(p)}_F= \norm{c J(p_n)}_{F}=\norm{J(p_n)}_E$.
This contradicts the assumption that $E(0,1)\ne L_\infty(0,1) $ \cite[p.118]{LT2}. 
Therefore, $\tau_2$ is a semifinite faithful normal trace on $J(\cA)$
and
$J$ is a surjective isometry from $E(\cA,\tau_1)$ onto $F(J(\cA),\tau_2)$.
By Theorem~\ref{infinite},
for any $\tau$-finite $p\in \cA$, we have
$$\frac {\tau_1(p)} { \tau_2(J(p)) }  \tau_2(J(x) ) = \tau_1(x), ~\forall x\in \cA.$$
Let $\{p_n\}_{n\ge 1}$ be a decreasing sequence with $p_n\downarrow 0$ and $\tau(J(p_n))<\infty $. 
In particular, we have $ \frac {\tau_1(p_n)} { \tau_2(J(p_n)) }  = \frac {\tau_1(p_1)} { \tau_2(J(p_1)) }  $ for all $n\ge 1$.
For any non-trivial $\tau$-finite projection $p$, we have 
$$ \frac {\tau_1(p_1 )} { \tau_2(J(p_1)) }   =\frac {\tau_1(p_n )} { \tau_2(J(p_n)) }   =\frac {\tau_1( ({\bf 1}-p_n)p ({\bf 1}-p_n) )} { \tau_2(J(({\bf 1}-p_n) p ({\bf 1}-p_n))) } .   $$
Taking $n\to \infty $, we obtain that  
$$ \frac {\tau_1(p)} { \tau_2(J(p)) }  $$
is a constant (denoted by $\alpha$) for all  non-trivial $\tau$-finite projections 
 $p\in \cM_1$.
Since $\tau_1$ and $\tau_2$ are normal, it follows that  
$\tau_2(J(x))= \alpha \cdot   \tau_1(x)$, $\forall x\in \cM_1$.
This, in turn, proves \eqref{nonsupinf}.

By Theorem \ref{infinite}, if
  that
\eqref{nonsupinf} does not hold for any $0<t\neq1$, then $c=1$ and $\alpha=1$ (i.e., $J$ is a trace-preserving Jordan $*$-isomorphism).
This completes the proof.
\end{proof}

\chapter{The uniqueness of the  symmetric structure in noncommutative $L_p$-spaces: a noncommutative Semenov--Abramovich--Zaidenberg Theorem}\label{Sec:AZ}

\section{Proof of Theorem \ref{AZ111}}

Below, we establish
a noncommutative version of the   Semenov--Abramovich--Zaidenberg Theorem, providing a proof for
   Theorem \ref{AZ111} in the  setting of finite traces.
%\begin{theorem}\label{8.1}
%Let $\cM$ be an atomless finite von Neumann algebra equipped with a faithful normal tracial state $\tau$.
%Let $E(\cM,\tau)$ be a symmetric space.
%If $E(\cM,\tau)$ is isometric to $L_p(\cN,\nu)$, $1\le p<\infty$, for some atomless von Neumann algebra $\cN$ equipped with a faithful normal tracial state $\nu$,  then $E(\cM,\tau)$ coincides with $L_p(\cM,\tau)$ up to a proportional norm.
%\end{theorem}
\begin{proof}[Proof of Theorem \ref{AZ111}]
Without loss of generality, we may assume that $\nu$ is a tracial state.

Since $E(\cM,\tau)$ is isomorphic to $L_p(\cN,\nu)$, it follows that $E(\cM,\tau)$ does not contain $c_0$-copies~\cite[Theorem 5.9.6]{DPS}.
Therefore, $E(0,1)$ does not contain a $c_0$-copy, i.e., $E(0,1)$ is a $KB$-space.
In particular, $E(\cM,\tau)$ is fully symmetric~\cite[Corollary 5.3.6]{DPS}.

Assume that $p=2$. % Order continuity  implies that $E(0,1)$ is fully symmetric.
The space
$E(\cM,\tau)$ isometric to $L_2(\cN,\nu)$ implies that $E(0,1)$ is isometric to a complemented subspace of $L_2(\cN,\nu)$,
 which again is a separable infinite-dimensional Hilbert space.
In particular, $E(0,1)$ is isometric to $L_2(0,1)$.
By~\cite[Theorem 1]{AZ} (see also \cite{Semenov}), 
we have 
$$E(0,1)=L_2(0,1)$$ up to a proportional norm.

Let $p\ne 2$.
There exists an isometry  $T$ from $E(\cM,\tau)$ onto $L_p(\cN,\nu)$, which is of the form
$$T(x)=AJ(x)z +J(x)B({\bf 1}-z),~x\in \cM,$$
where $z$ is a central projection in $\cM$, $J$ is a Jordan $*$-isomorphism from  $\cM$ onto $\cN$ and $A,B\in L_p(\cN,\nu)$~\cite{HS}.
Without loss of generality, we may assume that $A,B\ge 0$. 

There exists an atomless abelian von Neumann subalgebra $\cN'$ of $\cN$
containing $A$ and $B$~\cite{CKS}.
Let $\cN''=J^{-1}(\cN')$.
By the same argument as that  in the proof of Theorem~\ref{K-R-Z}, we obtain that $T$ is a surjective isometry from $E(\cN'')$ onto $L_p(\cN')$ with the Lebesgue measure.

Let (see \cite[Theorem 3.1.1]{LSZ}) $$E(0,1):=\{x\in S(0,1):\mu(x)=\mu(A) \mbox{ for some }A\in E(\cM,\tau)\}$$ with  $$\norm{x}_E:=\norm{A}_{E(\cM,\tau)}.$$
Note that there exists a trace-measure preserving isomorphism $I_1$ (respectively, $I_2$) from $S(\cN',\nu)$ (respectively, $S(\cN'',\tau)$) onto $S (0,1)$.
Hence,
$$I_2 \circ T\circ I_1^{-1} $$
is an isometry from $E(0,1)$ onto $L_p
(0,1)$.
By \cite[Theorem 1]{AZ}, $E(0,1)$ coincides with $L_p(0,1)$ and $$\norm{\cdot}_E=\norm{\chi_{_{(0,1)}}}_E\norm{\cdot}_p.$$
In other words,
  $E(\cM,\tau)=L_p(\cM,\tau)$ up to a proportional norm~\cite[Theorem 3.1.1]{LSZ}.
\end{proof}
\section{An infinite version of Theorem \ref{AZ111}}

Below, we establish Theorem \ref{AZ111} in the setting of general semifinite von Neumann algebras,
which is new even for symmetric function spaces on an infinite interval.
\begin{theorem}\label{AZ infinite NC}Let $1\le p<\infty$.
Let $\cM$ be an atomless von Neumann algebra equipped with a semifinite infinite faithful normal trace $\tau$.
If 
a symmetric space
$E(\cM,\tau)$ is isometric to $L_p(\cN,\nu)$ for some  atomless   von Neumann algebra $\cN$ equipped with a semifinite  faithful normal trace $\nu$.
Then, $E(0,\infty)$ coincides with $L_p(0,\infty )$ with $$\norm{\cdot}_E=\lambda \norm{\cdot}_{L_p}$$
for some $\lambda >0$.
\end{theorem}
\begin{proof}
We only consider the case when $\nu$ is infinite. The case when $\nu$ is finite follows from the same argument. 

Since $E(\cM,\tau)$ is isomorphic to $L_p(\cN,\nu)$, it follows that $E(\cM,\tau)$ does not contain $c_0$-copies~\cite[Theorem 5.9.6]{DPS}.
Therefore, $E(0,\infty )$ does not contain a $c_0$-copy, i.e., $E(0,\infty)$ is a $KB$-space.
In particular, $E(\cM,\tau)$ is fully symmetric~\cite[Corollary 5.3.6]{DPS}.
By \cite{HS}, there exists a Jordan $*$-isomorphism from $\cM$ onto $\cN$, which can be extended to 
a Jordan $*$-isomorphism from $S(\cM_e,\tau)$ onto $S\left(\cN_{J(e)},\nu\right)$ for any $\tau$-finite projection
$e$ such that $\nu(J(e))<\infty$, see Corollary \ref{lemma measure-measure}. 

\begin{enumerate}
\item Assume that $p=2$.
The space
$E(\cM,\tau)$ isometric to $L_2(\cN,\nu)$ implies that $E(0,a )$,  $0<a\le \infty$,  
 is isometric to a complemented subspace of $L_2(\cN,\nu)$,
 which again is a separable infinite-dimensional Hilbert space.
In particular, $E(0,a  )$ is isometric to $L_2(0,1)$.
%Let $T'$ be a surjective isometry from $E(0,\infty)$ onto $L_2(0,1)$.
%Assume by contradiction that $$\norm{\cdot}_E \ne \lambda \norm{\cdot}_{L_2}.$$
%By Zaidenberg's theorem \cite{Zaidenberg},
%we have  $$(T' f) (t)= a(t) f(\sigma (t)),~ \forall f\in E(0,\infty),~ \forall t\in (0,\infty ).$$
\item
Assume that $p\ne 2$.
Let $T$ be a surjective isometry from $E(\cM,\tau)$ onto  $L_p(\cN,\nu)$.
Arguing as in the proof of Theorem~\ref{th:infiniteKRZ},
we obtain that there exists an increasing net $\{e_i\}_{i\in I}$ of  $\tau$-finite projections in $\cM_1$
such that $\nu (J(e_i) )<\infty $ and $\sup_i e_i={\bf 1}$. 
By \cite[Theorem 1.2]{HS}, we  have 
$$x\mapsto  T(x) = J(x) T(e_i )z + T(e_i )J(x) ({\bf 1}-z) $$
is an isometry from $E(e_i \cM e_i, \tau)$ into $L_p(\cN,\nu)$,
and $$ x \mapsto  J( x )  |T(e_i)^*|z + |T(e_i)| J(x) ({\bf 1}-z) $$
is a surjective isometry from 
$E(e_i \cM e_i, \tau)$ onto $L_p(J(e_i )\cN J(e_i),\nu)$.
Indeed, 
for any $y\in L_p
\left(
J(e_i )\cN J(e_i),\nu\right)$, 
there exists $x\in E(\cM,\tau)$ such that 
$$ J(x)T(e_i)z+T(e_i) J(x)({\bf 1}-z)=
T(x)=yu_1^*  z+u_2y ({\bf 1}-z),$$
where $T(e_i)^*= u_1|T(e_i)^*|$ and $T(e_i) = u_2|T(e_i)|$ are the polar deompositions. 
In particular, $$ J( x )  |T(e_i)^*|z + |T(e_i)| J(x) ({\bf 1}-z) =y. $$
Noting that 
 $J(x)|T(e_i)^*|z  =y z$ and $|T(e_i)| J(x) ({\bf 1}-z)  =  y  ({\bf 1}-z)$, 
and $r(T(e_i)^* z ) = l(T(e_i)z) \le J(e_i)z $ and $r(T(e_i)({\bf 1}-z ) ) \le J(e_i) ({\bf 1}-z )$, we have
$$ 
 xJ^{-1} (z) =J^{-1}\left(
y  |T(e_i)^*|^{-1}z \right)  \in 
S(e_i \cM e_i, \tau)
  $$ 
and 
$$xJ^{-1}({\bf 1}-z) = J^{-1}\left(
|T(e_i)|^{-1}  y  ({\bf 1}-z)
\right) \in 
S(e_i \cM e_i, \tau) . $$
This shows that $ x  \in 
E(e_i \cM e_i, \tau) $. 
By Theorem \ref{AZ111},  
$  E(0, \tau(e_i))$ is isometric to $  L_p(0,1)$ for every $i$.
 \end{enumerate}

By Theorem \ref{AZ111}, for any $i\ge 1$, there exists a constant $\lambda _i$ such that  $$\norm{\cdot}_E=\lambda_i \norm{\cdot}_{L_p} $$ on $E(0,\tau(e_i))$ for every $i$.
For any $j\ge i$, we have $\lambda_i=\lambda _j$.
That is, there exists a constant $\lambda $ such that
$$\norm{f}_{E(0,\infty)} =\lambda \norm{f}_{L_p(0,\infty )}$$ for all finitely supported elements   $f\in E(0,\infty)$.
Since $E(0,\infty )$ has order continuous norm, it follows that
$$\norm{f}_{E(0,\infty)} =\lambda \norm{f}_{L_p(0,\infty )}$$
for all $f\in E(0,\infty)$.
The proof is complete.
\end{proof}

\begin{cor}\label{AZ infinite}
Let $1\le p<\infty$.
If a symmetric function space  $E(0,\infty)$ is isometric to $L_p(0,\infty )$,
then it coincides with $L_p(0,\infty )$ and
 $$\norm{\cdot}_E=\lambda \norm{\cdot}_{L_p}$$ for some $\lambda >0$.
\end{cor}

\appendix
\chapter{Hermitian operators on noncommutative symmetric spaces}
 \label{appendix}

In this appendix, we extend
the description of hermitian operators on a noncommutative symmetric space with order continuous norm
 established in  \cite[Theorem 1.3]{HS} to the setting of noncommutative symmetric spaces  having the Fatou property.
Note that the result in \cite[Theorem 1.3]{HS}
holds under a slightly more general condition that symmetric spaces are minimal (i.e., bounded operators with finite supports are   dense).
The main tool is the continuity of hermitian operators established in Theorem \ref{continous of hermitian}.

% Due to the lack of order-continuous norms of $E(\cM,\tau)$, there are many technical differences.  Sourour\cite[Lemma 1]{Sourour} obtained Lemma \ref{lemma:orthogonal} below in the setting of $B(\cH)$ by using  a result due to  Schneider and  Turner (see e.g. \cite[Lemma 3.1]{ST} and \cite[Lemma 9.2.7]{FJ}).
%Arazy gave  a self-contained alternative proof in the setting of complex sequence spaces\cite{Arazy85}.
%In the proof of the following lemma,
 %we adopt   Arazy's proof.
%Due to the  technical differences between the atomless case and atomic case, we provide a full proof below.

From now on, unless stated otherwise,  we  always assume that $\cM$ is an   atomless semifinite  von Neumann algebra or an atomic semifinite von Neumann algebra with all atoms having  the same traces (without loss of generality, we assume that $\tau(e)=1$ for any atom $e\in \cM$), and we assume that  $\tau$ is a semifinite faithful normal trace on $\cM$.
In particular,
when $\cM$ is atomless (resp. atomic), the set
$$E(0,\tau({\bf 1})):=\{f\in S(0,\tau({\bf 1})) :\mu(f)=\mu(x) \mbox{ for some } x\in E(\cM,\tau)\}$$
(resp.
$$\ell_E:=\{f\in  \ell_\infty  :\mu(f)=\mu(x) \mbox{ for some } x\in E(\cM,\tau)\})$$
is a symmetric function (resp. sequence) space\cite[Theorem 2.5.3]{LSZ}.
There exists a bijective correspondence between symmetric operator spaces and symmetric function/sequence spaces.
Therefore, if $\norm{\cdot}_E$ on $E(\cM,\tau)$ is not proportional to $\norm{\cdot}_2$ on $L_2(\cM,\tau)$, then $\norm{\cdot}_E$
is not proportional to   $\norm{\cdot}_2$ on $L_2(\cA,\tau)$ for any maximal abelian von Neumann subalgebra $\cA$ of $\cM$.

% Before proceeding to the proof of Lemma \ref{lemma:orthogonal}, we need the following well-known proposition.
%The following proposition is folklore.
The following proof of the following result is similar to \cite[Proposition  3.1]{HS}.
We omit the proof.
%For the sake of completeness, we provide a short proof below.
%For any symmetric function/sequence space $E(0,\infty )$/$\ell_E$, we denote by $\varphi_E$ the characteristic function of $E$.

\begin{proposition}\label{prop:tppR}
Assume that  $\cM$
 is an atomless semifinite von Neumann algebra or an atomic semifinite von Neumann algebra with all atoms having the same trace (without loss of generality, we assume that $\tau(e)=  1$ for any atom $e\in  \cM$).
Let $p\in \cF(\tau)$ be a  projection and let $E(\cM,\tau)$ be an arbitrary  noncommutative  strongly    symmetric   space.
   Then,  $$ \frac{\norm{p}_E  ^2 }{\tau(p)} p \in E(\cM,\tau)^*$$ is a support functional of $p\in E(\cM,\tau)$,
    i.e.,
 $ \tau\left(p\cdot  \frac{\norm{p} _E^2 }{\tau(p)} p  \right) =\norm{p}_E^2  =\norm{p}_E \norm{ \frac{\norm{p} _E^2 }{\tau(p)} p }_{E^*}  $.
In particular, for any bounded hermitian operator $T$ on $E(\cM,\tau)$, we have
\begin{align}\label{prop:tppR}
\tau(T(p) p)\in \mathbb{R}.
\end{align}
\end{proposition}
\begin{comment}\begin{proof}
We only consider the case when
$\cM$ is atomless. The atomic case follows from the same argument (see also \cite{Arazy85} or \cite[Theorem 5.2.13]{FJ}).

%Note that $$\norm{p}_{E^*}=\sup \left{ \int \mu(t;y)\mu(t;)  \right}$$

%Since $E(\cM,\tau)$ is separable, it follows that $E(\cM,\tau)$ is $\tau$-compact and therefore, there exists an atomless abelian von Neumann subalgebra $\cN$ of $\cM$  such that $p\in \cN$ and $S(\cN,\tau)$ is isomorphic onto to $S (0,\infty)$ (see e.g. \cite[Lemma 1.3]{CKS}).

Note that
$$  \norm{p }_{E^*}= \norm{p }_{E^\times } = \sup \left\{  \int_{0}^{\tau(p)} \mu(s;z )ds :z\in E(\cM,\tau), ~\norm{z}_E =  1 \right\} . $$
Since $\frac{\int_0^{\tau(p)} \mu(s;z )ds  }{{\tau(p)} } \mu( p )  = \frac{\int_0^{\tau(p)} \mu(s;z )ds }{{\tau(p)} } \chi_{(0,\tau(p))}   \prec \prec \mu(z)   $, $z\in E(\cM,\tau)$
  (see e.g. \cite[Section 3.6]{LSZ}), we obtain that $\frac{\int_0^{\tau(p)} \mu(s;z )ds  }{{\tau(p)} }  \norm{p}_E \le \norm{z}_E =1 $, and therefore,
$$\norm{p  }_{E^*}  \le \frac{\tau(p)}{\norm{p }_E }.$$
On the other hand,    we have\cite[Remark 3]{DP2014}
$$\tau(p  )\le  \norm{p}_{E^*} \norm{p}_{E}.$$
Hence, $\tau(p  )= \norm{p}_{E^*} \norm{p}_{E}$, i.e. $\tau\left(p\cdot  \frac{\norm{p} _E^2 }{\tau(p)} p  \right) = \norm{p}_E^2 =    \norm{p}_E \norm{ \frac{\norm{p} _E^2 }{\tau(p)} p }_{E^*}  $.
\end{proof}
\end{comment}

\begin{cor}\label{prop:tUUR}

Let $\cM$ be an  atomless von Neumann algebras (or an  atomic von Neumann algebras whose atoms have the same trace) equipped with  a semifinite faithful normal trace $\tau$.
Let $u\in \cF(\tau)$ be a  partial isometry and let $E(\cM,\tau)$ be an arbitrary   noncommutative  strongly    symmetric   space.
   Then,  $$ \frac{ \norm
{ u^*u } _E ^2 }{\tau(u^* u)} u^*   \in E(\cM,\tau)^*$$ is a support functional of $u\in E(\cM,\tau)$.
In particular, for any bounded  hermitian operator $T$ on $E(\cM,\tau)$, we have
\begin{align*}%\label{prop:tuuR}
\tau(T(u) u^*)\in \mathbb{R}.
\end{align*}
\end{cor}

\begin{lemma}
\label{lemma:orthogonal}
Let $\cM$ be an  atomless  von Neumann algebra (or an   atomic von Neumann algebra whose atoms have the same trace) equipped with a  semifinite faithful normal trace $\tau$.
Let $E(\cM ,\tau )$  be a  noncommutative symmetric  spaces  and  are not proportional to  $\norm{\cdot}_2$.
Let $x_1,x_2\in \cF(\tau )$   be two partial isometries such that $l(x_1) \perp l(x_2)$ and $r(x_1)\perp r(x_2)$.
Then, for any bounded hermitian operator $T:E(\cM,\tau)\to E(\cM,\tau )$, we have
$$\tau (T(x_1)x _2^* ) = 0.$$
Consequence,
  $x_1\in \cF( \tau)$ and $x_2\in L_1(\cM,\tau)\cap \cM$ with  $l(x_1) \perp l(x_2)$ and $r(x_1)\perp r(x_2)$, then $$\tau (T(x_1)x _2^* ) = 0.$$

Assume, in addition, that one of the following conditions holds:
\begin{enumerate}
\item    $E(\cM,\tau)$ is minimal;
\item  $E(\cM,\tau) $ has the Fatou property
and $E(\cM,\tau)\subset S_0(\cM,\tau)$;
\item  $E(\cM,\tau) $ has the Fatou property
and  $\cM$ is $\sigma$-finite.
\end{enumerate} 
If $x_1\in E(\cM,\tau)$ and $x_2\in L_1(\cM,\tau)\cap \cM$ with  $l(x_1) \perp l(x_2)$ and $r(x_1)\perp r(x_2)$, then $$\tau (T(x_1)x _2^* ) = 0.$$
\end{lemma}
\begin{proof}
We only prove the case for atomless von Neumann algebras. The atomic case follows by a  similar argument.

We only consider the case when $E(\cM,\tau)$ has the Fatou property. The minimal case follows from the same argument as that in \cite[Lemma 3.3]{HS}.
Observe that
  the Fatou property implies that $E(\cM,\tau)$ is fully symmetric~\cite[Corollary 5.1.12]{DPS}.

We first consider the case when $x_1$ and $x_2$ are two projections such that $x_1x_2=0$.

Fix a projection $p\in\cF(\tau )$. Since $\norm{\cdot}_E$ is not proportional to $\norm{\cdot}_2$, it follows that
there exists a set of pairwise orthogonal  projections $\{e_i\}_{1\le i\le n}$ having the same trace
such that 
 $$\sum_{i=1}^m e_i =p $$ for some $m\le n$
and  $\norm{\cdot}_E$ on $E(\cA)$ is not proportional to $\norm{\cdot}_2$ on  $L_2(\cA)$,
where $\cA$ is
 the abelian weakly closed $*$-algebra generated by $\{e_i\}_{1\le i\le n}$. By the same argument used in   \cite[Lemma 3.3]{HS},
for any $\tau$-finite projections $q\in \cM$ such that
 $pq=0$ and  $\tau(q)= \frac{1}{2^k }\tau(p)$, 
we have
\begin{align}\label{Tpq0}
\tau(T(p)q)=0.
\end{align}

The case when  $\tau(q )\ne \frac{1}{2^k}\tau( p )$ for all $k\ge 1$ follows by standard approximation argument.
Indeed,
let    $q  $ be  a  $\tau$-finite projection
 with $pq=0$.
For any $\varepsilon>0$,  there exists a set of   $\tau$-finite projections
  $\{q_i\}_{1\le i\le m}$  such that $$ \tau\left(q-\sum_{1\le i\le m}  q_i \right)\le \varepsilon $$ and
 $$ \tau(q_i)=\frac{1}{2^k }\tau(p), ~ 1\le i\le n, $$ for some $k\in \mathbb{N}$.
Note that
 \begin{align*}
~&  \tau\left( T(p) q\right)  \\
  =~& \tau\left( T(p  ) \left(q -\sum_{1\le i\le m} q_i \right)\right) + \tau\left( T\left(p \right) \sum_{1\le i\le m} q_i   \right) \\
   \stackrel{\eqref{Tpq0}}{=} & \tau\left( T(p  ) \left(q -\sum_{1\le i\le n} p _i \right)\right).
   \end{align*}
   Since
$$
\left|\tau\left( T(p ) \left(q -\sum_{1\le i\le m} q_i \right)\right)\right|\to 0
%\le
%tau\left( \left|T(p  ) \left(q -\sum_{1\le i\le m} q_i \right )\right| \right)= \norm{ T(p  ) \left(q -\sum_{1\le i\le m} q_i \right) }_1 \to 0
$$
as $\varepsilon\to 0$, % (see e.g. \cite[Lemma 3.10]{DPS2016} and \cite{DP2012})
%and
%\begin{align*}
% \left|\tau\left( T\left(p  -\sum_{1\le i\le n} p_i\right) \sum_{1\le i\le m} q _i  \right )\right|
% &\le \norm{T}\norm{p  -\sum_{1\le i\le n} p_i}_E\norm{\sum_{1\le i \le m } q _i }_{E^\times }\\
% &\le \norm{T}\norm{p  -\sum_{1\le i\le n} p_i}_E\norm{\sum_{1\le i\le m } q  }_{E^\times }\to 0
% \end{align*}
% as $\varepsilon\to 0$ because the norm $\norm{\cdot}_E$
%is order continuous,
it follows that
\begin{align*}%\label{Tpqdis0}
\left| \tau\left( T\left(p\right) q\right)\right|=0  .
\end{align*}

%If $x_1,x_2\in \cF(\tau)$  are partial isometries, then there exists a  reduced von Neumann algebra $\cN$ of $\cM$ generated by a $\tau$-finite  projection  $p\in   \cM$ (i.e., $\cN=p\cM p$), which  contains $x_1,x_2$.   Since $x_1$ and $x_2$ are disjointly supported from the left and from the right, it follows that $x_1+x_2$ is also a partial isometries.   By \cite[Chapter XIV, Lemma 2.1]{T3}, there exists a unitary element      $u\in \cN$ such that $u r_{x_1+x_2}u^*=l_{x_1+x_2}$.      Let $u'= x_1+x_2  +  u(p-r_{x_1 +x_2})$.      This is also a unitary element in $\cN$.   Taking a unitary element $v= u' +({\bf 1}- p)$ in $\cM$, we obtain that $v^*Tv$ is also a Hermitian operator (see e.g. \cite[p.22]{FJ} and \cite{KR}).
The general case when $x_1$ and $x_2$ are partial isometries in $\cF(\tau )$ such that $l(x_1) \perp l(x_2)$ and $r(x_1)\perp r(x_2)$ is reduced to the just considered case  via the same argument as in
  Corollary \ref{prop:tUUR} (see also \cite{HS}). Therefore,  we obtain that % By \eqref{Tpqdis0} above, we obtain that
 %$$0= \tau( v^*T(v (r(x_1)) r(x_2)  )= \tau(T  ( x_1) r(x_2) v^*  ) )= \tau(T  ( x_1)  x_2 ^*  ) )  ,$$
 $$  \tau\left(T  ( x_1)  x_2 ^*  ) \right )=0  ,$$
  which completes the proof of the first assertion.

Now, we prove  the second assertion. Since $\cM$ is $\sigma$-finite (or $E(\cM,\tau)\subset S_0(\cM,\tau)$),  
it follows from Lemma~\ref{lemma net and sequence} that 
there exist 
two sequences $\{y_n\}_{n\ge 1}$ and $\{y_n'\}_{n\ge 1}$
 in $\cF(\tau)$ such that $y_{n}\uparrow|x_1|$ and $y_{n'}\uparrow|x_{2}|$.
%
 %(resp. $y_n'$) are generated by spectral projections of $|x_1|$ (resp. $|x_2|$).
 Let $x_1=u_1|x_1|$ and  $x_2=u_2|x_2|$ be the polar decompositions.
 By the first assertion of the lemma, $$\tau(T(u_1y_n)(u_2y_m')^*)=0$$ for all $n$ and $m$. Since $u_1y_n\stackrel{\sigma(E,E^{\times})}{\longrightarrow}x_1$ and $T$ is $\sigma(E,E^{\times})-\sigma(F
,F^\times)$-continuous (see Theorem~\ref{continous of hermitian}), it follows that
$$T(u_1y_{n})\stackrel{\sigma(E,E^{\times})}{\longrightarrow}T(x_1)$$
as $n\to \infty $.
Note that $\{(u_2y_m')^*\}_{n=1}^\infty \subset\cF(\tau)\subset E(\cM,\tau)^\times$ (see \cite[Lemma 4.4.5]{DPS}).
 Hence,  $$\tau(T(x_1) y_m'u_2^*)=\tau(T(x_1)(u_2y_m')^*)=0$$ for all $m$.
Again, 
since $ y_m ' u _2 \to x_2^* $ as $m \to \infty$ in the $\sigma(L_1 \cap L_\infty ,L_1 \cup L_\infty) $-topology, 
 we have 
$$\tau(T(x_1) x_2^*  ) =0.$$
This completes the proof.
\begin{comment}
Consider the special case when
 $x_2\in \cF(\tau)$.
 In this case, we may assume, in addition,  that $\norm{y_m'-x_2}_\infty \to 0$.
 We obtain that
\begin{align}\label{0F}
|\tau(T(x_1)   |x_2|  u_2^* )| =|\tau(T(x_1)r(x_2) (y_m '- |x_2|) u_2^* )|\le   \norm{T(x_1)r(x_2)    }_1 \norm{y_m '- |x_2|}_\infty \to 0.
\end{align}

For the general case, let $p:=E^{|x_2|}( \delta,\infty )$, $\delta>0$,  be a $\tau$-finite spectral  projection of $|x_2|$ such that $\norm{|x_2|({\bf 1}-p)}_{L_1\cap L_\infty }\le \varepsilon$.
Hence, we obtain that
\begin{align*}
|\tau(T(x_1) |x_2| u_2^*  )|&\stackrel{\eqref{0F}}{=}
 |\tau(T(x_1)  ({\bf 1}-p) |x_2| u_2^* )|\\
&~\le \norm{T(x_1)}_{L_1+L_\infty } \norm{ ({\bf 1}-p) |x_2| u_2^* }_{L_1\cap L_\infty } \\
&~\le  \norm{T(x_1)}_{L_1+L_\infty } \cdot \varepsilon.
\end{align*}
%when $m\to \infty$.
Since $\varepsilon$ is arbitrarily taken, it follows that $|\tau(T(x_1) x_2^*  )| =0$.
\end{comment}
\end{proof}

Lemma \ref{lemma:orthogonal} together with
the argument used in   \cite[Corollary 3.4]{HS}
implies that following result.
\begin{cor}\label{cor:disjoint}

Let $\cM$ be an  atomless  von Neumann algebra (or  an  atomic von Neumann algebra whose atoms have the same trace) equipped with  a  semifinite faithful normal trace $\tau$.

Let $E(\cM,\tau)$ be a  symmetric space   affiliated with $\cM$, whose norm is not proportional to $\norm{\cdot}_2$.
Let $T$ be a bounded hermitian operator on $E(\cM,\tau)$.
For any $x\in \cF(\tau )$,
there exist    $y,z\in E(\cM,\tau)$ such that $T(x)=y+z$ and $r(y)\le r(x)$ and $l(z)\le l(x)$.

Assume, in addition, that one of the following conditions holds
\begin{enumerate}
\item    $E(\cM,\tau)$ is minimal;
\item  $E(\cM,\tau) $ has the Fatou property
and $E(\cM,\tau)\subset S_0(\cM,\tau)$;
\item  $E(\cM,\tau) $ has the Fatou property
and  $\cM$ is $\sigma$-finite.
\end{enumerate} 
For any $x\in E(\cM,\tau )$,
there exist    $y,z\in E(\cM,\tau)$ such that $T(x)=y+z$ and $r(y)\le r(x)$ and $l(z)\le l(x)$.

\end{cor}
%\begin{proof}
%Denote $A:=T(x) $. Note that
%$$
%A=l(x)Ar(x)+  l(x) A r(x)^\perp   + l(x)^\perp  Ar(x) +(l(x)^\perp A r(x)^\perp ).
%$$
%Assume that $l(x)^\perp A r(x)^\perp\ne 0$.
%Let
%$p\in \cF(\tau )$ be a $\tau$-finite projection such that
%  $z= l(x)^\perp A r(x)^\perp p \ne 0$.
%Let $zp =u |zp|$ be the polar decomposition.
%Then,
%$$
%u^*l(x)^\perp A r(x)^\perp p= u^* z p \ge 0,
%$$
%i.e.,
%$$\tau  (T(x ) r(x)^\perp pu^* l(x)^\perp  )= \tau  (u^*l(x)^\perp A r(x)^\perp p ) > 0 $$
%Note that $l\left(l\left(x\right)^\perp u p r(x)^\perp \right) \perp l(x) $  and $r\left(l(x)^\perp u p r(x)^\perp \right) \perp r(x) $.
%By Lemma \ref{lemma:orthogonal} above, we obtain that
%$$\tau  (u^*l(x)^\perp A r(x)^\perp p )=\tau  ( T(x ) r(x)^\perp p u^*l(x)^\perp )=  0 ,$$
% which is a contradiction.
% Taking $y=l(x) Ar(x) +l(x) ^\perp A r(x)$ and $z= l(x) A r(x)^\perp$, we complete the proof (note that the choices of $y$ and $z$ are not necessarily unique).
%\end{proof}

%Let's consider an arbitrary atomless abelian von neumann subalgebra $\cA$ of $\cM$.

The following lemma shows that any bounded  hermitian operator $T$ on $E(\cM,\tau)$ (whose norm is not proportional to $\norm{\cdot}_2$) maps the set of all  $\tau$-finite projections to a uniformly  bounded set in $\cM$, which should be compared with the estimates  given  in~\cite[Remark 2.5]{Sukochev}.
\begin{lemma}\label{lemmaptop}
Let $\cM$ be an  atomless  von Neumann algebra (or  an  atomic von Neumann algebra whose atoms have the same trace) equipped with a semifinite faithful normal trace $\tau$.
Let $E(\cM,\tau)$ be a  symmetric space % (have the Fatou property or it is minimal)
affiliated with $\cM$, whose norm is not proportional to $\norm{\cdot}_2$.
Assume that one of the following conditions holds:
\begin{enumerate}
\item    $E(\cM,\tau)$ is minimal;
\item  $E(\cM,\tau) $ has the Fatou property
and $E(\cM,\tau)\subset S_0(\cM,\tau)$;
\item  $E(\cM,\tau) $ has the Fatou property
and  $\cM$ is $\sigma$-finite.
\end{enumerate} 
Let $T$ be a bounded hermitian operator on  $E(\cM,\tau)$.
Then, $\norm{T(p)}_\infty \le 3 \norm{T}$ for any $\tau$-finite projection $p\in \cP(\cM)$.
\end{lemma}
\begin{proof}
It follows from
  the same argument used in \cite[Lemma 3.5]{HS}. 
We note that there the projection $q=E^{|p^\perp Ap |} (\norm{T}, \infty)
$ defined in the proof of \cite[Lemma 3.5]{HS} should be replaced by 
$q=E^{|p^\perp Ap |} (\norm{T}+\varepsilon, \infty)\ne 0$ for some $\varepsilon>0$. 
\end{proof}

\begin{comment}
The following lemma is the key auxiliary tool in the proof of Proposition \ref{prop:general:H}, which shows that any bounded hermitian operator $T$ on $E(\cM,\tau)$ is  a bounded operator from $(\cF(\tau),
\norm{\cdot}_\infty)$  into $(C_0(\cM,\tau),
\norm{\cdot}_\infty)$.
By applying the generalized Gleason theorem\cite[Theorem 5.2.4]{Hamhalter} (see also \cite{BJM,Mori,GS1,GS2} for related results), we succeed to prove all but one case
 in the following lemma.
However, one should note that the case when   a von Neumann algebra has $I_2$ direct summand is  exceptional, which can not be covered by the generalized  Gleason theorem. This special case  is rather complicated  and requires
careful study of the restriction of a  hermitian operator on the type $I_2$ summand.
\end{comment}

We denote by $C_0(\cM,\tau)$ the closure in the norm $\norm{\cdot}_\infty $ of the linear span of $\tau$-finite projections in $\cM $.
Equivalently, $$C_0(\cM,\tau)=\{a\in S(\cM,\tau): \mu(a)\in L_\infty (0,\infty ), ~\mu(\infty,a)=0\} =S_0(\cM,\tau)\cap \cM,$$
see \cite[Lemma 2.6.9]{LSZ}, see also \cite{DPS}. 

A slight modification of the argument in  used in \cite[Lemma 3.6]{HS}  yields the following lemma.
The main tool is the so-called generalized Gleason Theorem \cite{Hamhalter}.
\begin{lemma}\label{lemmaMtoM}
Let $\cM$ be an  atomless  von Neumann algebra (or an  atomic von Neumann algebra whose atoms have the same trace) equipped with a semifinite faithful normal trace $\tau$.
Let $E(\cM,\tau)$ be a  strongly symmetric space %(have the Fatou property or it is minimal)
affiliated with $\cM$, whose norm is not proportional to $\norm{\cdot}_2$.

Let $T$ be a bounded  hermitian operator on  $E(\cM,\tau)$.
\begin{enumerate}
\item If $E(\cM,\tau)$ is minimal or  has the Fatou property and 
$E(\cM,\tau)\subset S_0(\cM,\tau)$, then $T$ is a bounded operator from $(\cF(\tau),
\norm{\cdot}_\infty)$  into $(C_0(\cM,\tau) ,
\norm{\cdot}_\infty)$.
In particular, $T$ extends to  a bounded operator from  $C_0(\cM,\tau)$ into $C_0(\cM,\tau)$.
 \item If $E(\cM,\tau)$ has the Fatou property and $\cM$ is $\sigma$-finite, then  $T$ extends to  a bounded operator from  $\cM$ into $\cM$.
\end{enumerate}
\end{lemma}

\begin{proposition}\label{redu}\cite[Proposition 3.7]{HS}
Let $\cM$ be an  atomless   von Neumann algebra (or an atomic von Neumann algebra whose atoms have the same trace) equipped with a  semifinite faithful normal trace $\tau$.
Let $E(\cM ,\tau )$  be a noncommutative  strongly  symmetric  space % with Fatou property  (or minimal) and  are
whose norm is  not proportional to  $\norm{\cdot}_2$.
Assume that 
one of the following conditions holds
\begin{enumerate}
\item    $E(\cM,\tau)$ is minimal;
\item  $E(\cM,\tau) $ has the Fatou property
and $E(\cM,\tau)\subset S_0(\cM,\tau)$;
\item  $E(\cM,\tau) $ has the Fatou property
and  $\cM$ is $\sigma$-finite.
\end{enumerate} 
 Let $T$ be a bounded  hermitian operator on    $E(\cM,\tau)$.
Then, for any operator $x\in C_0(\cM,\tau)$ (or $  \cM $) and
 a $\tau$-finite projection  $p\in \cP(\cM)$ commuting with $|x|$,
we have
$$\langle Tx ,  pu^* \rangle_{ ( \cM , \cM ^*)}:=\tau(T(x) pu^*  )\in \mathbb{R}, $$
where $x=u|x|$ is the polar decomposition.
\end{proposition}
%\begin{proof}
%By a same argument as in \cite[Corollary 3.7]{HS}
%\end{proof}

%\begin{proof}
%We only consider the case when $x$ is self-adjoint. The
%proof of  the general case follows from the same argument by replacing Proposition \ref{prop:tppR} used below with  Corollary \ref{prop:tUUR}.
%
%Let $x_n:= \sum_{1\le k\le n} \alpha_k p_k \in \cF(\tau  ) $  be such that $x_n\to x$ in $\norm{\cdot}_\infty $, where $p_k$ are $\tau$-finite spectral projections of $x_n$ which commute    with $p$,  and $\alpha_k$ are real numbers. For each  $p_k$, we have
%\begin{align*}
%  & \langle Tp_k, p\rangle_{(C_0(\cM,\tau ),C_0(\cM,\tau )^*)}\\
%= ~&\langle T(p p_k), p p_k\rangle_{(C_0(\cM,\tau ),C_0(\cM,\tau )^*)} +\langle T(pp_k), p-pp_k\rangle_{(C_0(\cM,\tau ),C_0(\cM,\tau )^*)}  \\
%& ~+  \langle T(p_k -p p_k ), p  \rangle_{(C_0(\cM,\tau ),C_0(\cM,\tau )^*)} \\
%\stackrel{\eqref{lemma:orthogonal} }{=}  & \langle T(p p_k), p p_k\rangle_{(C_0(\cM,\tau ),C_0(\cM,\tau )^*)}\\
%~ = ~ &
%\tau(T(p p_k)  p p_k)
%\stackrel{\eqref{prop:tppR}}{  \in} \mathbb{R}.
%\end{align*}
%Hence, $\langle Tx_n , p\rangle_{(C_0(\cM,\tau ),C_0(\cM,\tau )^*)} \in \mathbb{R}$ for every $n$.
%Moreover, we have
%$$|\langle  Tx,p \rangle_{(C_0(\cM,\tau ),C_0(\cM,\tau )^*)} -\langle  Tx_n ,p \rangle_{(C_0(\cM,\tau ),C_0(\cM,\tau )^*)} | \le
%\norm{T}_{C_0(\cM,\tau)\to C_0(\cM,\tau) }\norm{x-x_n}_\infty  \norm{p}_1\to 0,$$
%which shows that $\langle Tx , p\rangle_{(C_0(\cM,\tau ),C_0(\cM,\tau )^*)}\in \mathbb{R}$.
%\end{proof}
%

%
 
\begin{proposition}\label{prop:general:H}
Let $\cM$ be an  atomless   von Neumann algebra
 (or  an atomic von Neumann algebra
 whose atoms have the same trace) equipped with a semifinite faithful normal trace $\tau$.
Let $E(\cM ,\tau )$  be a  noncommutative symmetric  space 
in the sense of Lindenstrauss and Tzafriri,  
%with Fatou property  (or minimal) and
whose norm is  not proportional to  $\norm{\cdot}_2$.

   Let $T$ be a bounded  hermitian operator on   $E(\cM,\tau)$.
\begin{enumerate}
\item If $E(\cM,\tau)\subset S_0(\cM,\tau)$, then $T$ can be extended to a bounded operator on  $C_0(\cM,\tau)$ (still denoted by $T$) and, for any
  operator $x\in C_0(\cM,\tau )$,  there exists a support functional  $x'$ in $C_0(\cM,\tau)^*$ of $x$ such that
 $\langle Tx, x' \rangle \in \mathbb{R}$.
  In particular, $T$ is a hermitian operator on $C_0(\cM,\tau)$.
 \item If $E(\cM,\tau)$ has the Fatou property and $\cM$ is $\sigma$-finite, then  $T$ extends to  a bounded hermitian operator from  $\cM$ into $\cM$.
\end{enumerate}

%Then, $T$ can be extended to a bounded operator on  $C_0(\cM,\tau)$ (still denoted by $T$) and, for any
 % operator $x\in C_0(\cM,\tau )$,  there exists a support functional  $x'$ in $C_0(\cM,\tau)^*$ of $x$ such that
% $\langle Tx, x' \rangle \in \mathbb{R}$.
 % In particular, $T$ is a hermitian operator on $C_0(\cM,\tau)$.
\end{proposition}
\begin{proof}
It follows from the same argument as used in  \cite[Proposition 3.8]{HS} by replacing the spectral projection $e^{|x|}(1-1/n,1]$ with a $\tau$-finite projection $p_n \le e^{|x|}(1-1/n,1]$ which commutes with $|x|$. 
\end{proof}

Recall that a derivation $\delta$ on an algebra $\cA$ is a linear operator satisfying the Leibniz rule.
Although it is known that a  derivation from $C_0(\cM,\tau)$ into $C_0(\cM,\tau)$ is not necessarily inner~\cite{BHLS,Huang} (see   \cite[Example 4.1.8]{Sakai} for  examples of    non-inner derivations   on $K(\cH)$), it is known  that every derivation $\delta$ from an arbitrary  von Neumann subalgebra of $\cM$ into $C_0(\cM,\tau)$ is inner, i.e., there exists an element  $a\in C_0(\cM,\tau)$ such that~\cite{BHLS2,Huang} $$\delta(\cdot)=[a,\cdot].$$
 On the other hand,
 every    derivation from $C_0(\cM,\tau)$ into $C_0(\cM,\tau)$ is spatial, i.e., it can be implemented by an element from $\cM$ (see e.g. \cite[Theorem 2]{Sakai} and \cite[Theorem 4.1]{Arveson}).
On the other hand, recall that every derivation from $\cM$ into $\cM$ is necessarily inner \cite{K66,Sakai}.

\begin{lemma}\cite[Lemma 3.9]{HS}\label{derivation}Let $\cM$ be a semifinite von Neumann algebra equipped with a semifinite faithful normal trace $\tau$.
Every derivation $\delta$ from $C_0(\cM,\tau)$ into $C_0(\cM,\tau)$ (or from $\cM$ into $\cM$) is spatial.
In particular, if $\delta$ is a $*$-derivation, then the element implementing $\delta$ can be chosen to be self-adjoint.
\end{lemma}
%\begin{proof}
%Since every derivation $\delta$ from $C_0(\cM,\tau)$ into $C_0(\cM,\tau)$ is ultraweakly continuous\cite[Lemma 1.3]{JP}, it follows that $\delta$ can be extended to a derivation from $\cM$ into $\cM$.
%Recall that Kadison--Sakai theorem that every derivation on a von Neumann algebra is inner\cite{Kadison,Sakai}.
%We obtain that there exists an operator $a\in \cM$ such that $\delta(x)=[a,x]$, $x\in C_0(\cM,\tau)$.
%The first statement follows from  \cite[Theorem 2]{Sakai} (or \cite[Theorem 4.1]{Arveson}) and the fact that $C_0(\cM,\tau)$ is a $C^*$-algebra.
%For the second statement, see e.g. \cite[Chapter 3.4, Remark 3.4.1]{Huang}.
%\end{proof}

By Corollary \ref{prop:general:H}, $T$ extends to   a bounded  hermitian operator from $C_0(\cM,\tau)$ into $ C_0(\cM ,\tau )$ (or from $\cM$ into $\cM$).
Recall that any hermitian operator $T$ on a $C^*$-algebra $\cA$  is the sum of  a left-multiplication by a self-adjoint operator in $\cA$ and
a $*$-derivation on $\cA$ (see e.g. \cite{PS1} and \cite[p.213]{Sinclair}).
It follows from  Lemma \ref{derivation}  that there exist self-adjoint elements $a,b\in \cM$ such that
$$Tx=ax+xb,~x\in C_0(\cM,\tau ).$$

\begin{cor}\label{cor:h:c0}
Let $\cM$ be an  atomless   von Neumann algebra (or  atomic von Neumann algebra whose atoms have the same trace) equipped with a semifinite faithful normal trace $\tau$.
Let $E(\cM ,\tau )$  be a  noncommutative  strongly symmetric  space %with Fatou property  (or minimal) and
whose norm is  not proportional to  $\norm{\cdot}_2$.

   Let $T$ be a bounded  hermitian operator on   $E(\cM,\tau)$.
\begin{enumerate}
\item If $E(\cM,\tau)$ is minimal or 
$E(\cM,\tau)\subset S_0(\cM,\tau)$ has the Fatou property, then $T$ can be extended to a bounded hermitian operator on  $C_0(\cM,\tau)$ (still denoted by $T$) and,   there exist self-adjoint elements $a,b\in \cM$ such that
$$Tx=ax+xb,~x\in C_0(\cM,\tau ).$$
 \item If $E(\cM,\tau)$ has the Fatou property and $\cM$ is $\sigma$-finite, then  $T$ extends to  a bounded hermitian operator from  $\cM$ into $\cM$, and,  there exist self-adjoint elements $a,b\in \cM$ such that
$$Tx=ax+xb,~x\in \cM .$$
\end{enumerate}
 \end{cor}

We now present  the main result of this appendix, which gives the full  description of hermitian operators   on a  symmetric space $E(\cM,\tau)$.

\begin{theorem}\label{th:her}
Let $\cM$ be an  atomless   von Neumann algebra (or an  atomic von Neumann algebra whose atoms have the same trace) equipped with semifinite faithful normal trace
 $\tau$.
Let $E(\cM ,\tau )$  be a  noncommutative symmetric  space  whose norm is not proportional to  $\norm{\cdot}_2$.
If one of the following conditions holds
\begin{enumerate}
\item    $E(\cM,\tau)$ is minimal;
\item  $E(\cM,\tau) $ has the Fatou property
and $E(\cM,\tau)\subset S_0(\cM,\tau)$;
\item  $E(\cM,\tau) $ has the Fatou property
and  $\cM$ is $\sigma$-finite,
\end{enumerate} then a bounded linear operator $T:E(\cM ,\tau )\to E(\cM ,\tau)$ is a hermitian operator if and only if  there exist  self-adjoint operators $a$ and $b$ in $\cM$ such that
\begin{align}\label{eq:her}
Tx=ax+xb,~x\in E(\cM,\tau).
\end{align}
In particular, $T$ can be extended to a bounded hermitian operator from $\cM$ to $\cM$.
\end{theorem}
\begin{proof}
The `if' part of the theorem  is obvious (see e.g. the argument in  \cite[p.71]{Sourour} or \cite[p. 167]{FJ2}).

By Corollary \ref{prop:general:H}, $T$ extends to   a bounded  hermitian operator from $C_0(\cM,\tau)$ into $ C_0(\cM ,\tau )$ (or from $\cM$ into $\cM$).
%Recall that any hermitian operator $T$ on a $C^*$-algebra $\cA$  is the sum of  a left-multiplication by a self-adjoint operator in $\cA$ and
%a $*$-derivation on $\cA$ (see e.g. \cite{PS1} and \cite[p.213]{Sinclair}).
It follows from  Lemma \ref{derivation}  that there exist self-adjoint elements $a,b\in \cM$ such that
%$$Tx=ax+xb,~x\in C_0(\cM,\tau ).$$
%Noting that $\cF(\tau )\subset C_0(\cM,\tau )$,  we obtain that
$$Tx=ax+xb,~x\in \cF(
\tau ) .$$
Fix $x\in E(\cM ,\tau )$. By Proposition \ref{proposition weakly closed}, there is a sequence $\{x_{n}\}\subset\mathcal{\cF}(\tau)$ such that $x_{n}\xrightarrow{\sigma(E,E^{\times})}x$. If $y\in E(\cM ,\tau )^\times$, then
\begin{eqnarray*}
& & \tau_1(y(ax_{n}+x_{n}b))\\
&=&\tau_1((ya)x_{n})+\tau_1((by)x_{n})\\
&\rightarrow&\tau_1((ya)x)+\tau_1((by)x)\\
 &=&\tau_1(y(ax))+\tau_1(y(xb))\\
&=&\tau_1(y(ax+bx)),
\end{eqnarray*}
as $n\to \infty$.
Therefore,
\begin{eqnarray*}
ax_{n}+x_{n}b\xrightarrow{\sigma(E,E^{\times})}ax+xb.
\end{eqnarray*}
Since $T$ is a hermitian operator, it follows that $T$ is $\sigma(E,E^{\times})$-continuous (see Theorem~\ref{continous of hermitian}). Therefore, we have
\begin{eqnarray*}
T(x)&=&\sigma(E,E^{\times})-\lim\limits_{n\rightarrow\infty}T(x_{n})\\
&=&\sigma(E,E^{\times})-\lim\limits_{n\rightarrow\infty}(ax_{n}+x_{n}b)\\
&=&ax+xb.
\end{eqnarray*}
%Since $E(\cM,\tau)$ has order continuous norm, it follows that  $\cF(\tau)$ is dense in $(E(\cM,\tau),\norm{\cdot}_E)$ (see e.g. \cite[Proposition 46]{DP2014} or \cite[Remark 2.9]{HSZ}).
%For any $x\in E(\cM,\tau)$, there exists a sequence  $\{y_n\}\subset \cF(\tau)$ such that $\norm{y_n-x}_E\to 0$.
%Hence,
%we obtain that
%$$Tx= \norm{\cdot}_F-\lim_n T(y_n) = \norm{\cdot}_F-\lim_n \left(ay_n +y_n b\right )=  ax+xb,~x\in E(\cM,\tau) .$$
This completes the proof.
\end{proof}

%Let $E(0,1)$ be a symmetric space. 
%Assume that there exists $p>1$ such that 
%$$E(0,1) =E(0,1)^{(p)}=\{f\in S(0,1)| |f|^p\in E(0,1)\}.$$
%This implies that $E(0,1) =E(0,1)^{(q)}$ for all $1\le q\le p$. 
%We have 
%$$E(0,1) =E(0,1)^{(q)} = (E(0,1)^{(q)})^{(q)} = \cdots = \{f\in S(0,1)| | f|^{np} \in E(0,1)\}. $$ 
%If $E(0,1)\supset L_r(0,1)$ for some $r>1$ (in particular, if it has non-trivial lower Boyd index), 
%then there exists a function $f(t):=t^{-1/(r+\varepsilon)} \in L_r \subset E$. 
%However, there exists $n$ such that $|f|^{np}\notin L_1$.
%Note that $E\subset L_1$. We have   $|f|^{np}\notin E\subset L_1$, which is a contradiction.

\backmatter

\end{document}